\documentclass[11pt]{article}
\usepackage[top=1in, bottom=1in, left=1.25in, right=1.25in]{geometry}
\usepackage{amsmath,amssymb,amsthm,xcolor,enumerate,bbm}
\usepackage[unicode,breaklinks=true,colorlinks=true]{hyperref}

\usepackage{theoremref}
\usepackage{mathrsfs}
\usepackage[]{authblk}
\usepackage{comment}
\usepackage{cancel}
\usepackage{graphicx}

\numberwithin{equation}{section}
\newtheorem{thm}{Theorem}[section]
\newtheorem{cor}{Corollary}[section]
\newtheorem{lem}[thm]{Lemma}
\newtheorem{prop}{Proposition}[section]

\theoremstyle{remark}
\newtheorem{remark}{Remark}[section]
\theoremstyle{definition}

\newcommand{\R}{\mathbb{R}}

\renewcommand{\div}{\mathop{\rm div}\nolimits}

\newcommand{\lec}{{\ \lesssim \ }}
\newcommand{\gec}{{\ \gtrsim \ }}

\newcommand{\EQS}[1]{\begin{equation}\begin{split} #1 \end{split}\end{equation}}
\newcommand{\EQs}[1]{\begin{equation*}\begin{split} #1 \end{split}\end{equation*}}

\allowdisplaybreaks
\begin{document}

\title
{Quantitative regularity for the MHD equations via the localization technique in frequency space}

\author[1]{\rm Baishun Lai}
\author[2]{\rm Shihao Zhang}
\affil[1]{\footnotesize MOE-LCSM, School of Mathematics and Statistics, Hunan Normal University, Changsha, Hunan 410081, P. R. China}
\affil[2]{\footnotesize School of Mathematics and Statistics, Hunan Normal University, Changsha, Hunan 410081, P. R. China}

\date{}
\maketitle

\begin{abstract}
In this paper, we employ the localization technique in frequency space developed by Tao in \cite{MR4337421} to investigate the quantitative estimates for the MHD equations. With the help of quantitative Carleman inequalities given by Tao in \cite{MR4337421} and the pigeonhole principle, we establish the quantitative regularity for the critical $L^3$ norm bounded solutions which enables us  explicitly quantify the blow-up behavior in terms of $L^3$ norm near a potential first-time singularity. Some technical innovations, such as introducing the corrector function, are required due to the fact that the scales are inconsistent between the magnetic field and the vorticity field.

\medskip

\emph{Key words}:  Blow-up rate; Quantitative analysis; Carleman inequalities; Critical space; Pigeonhole principle.

\medskip

\emph{2010 Mathematics Subject Classifications}: 35B44, 35B65, 35Q30, 76D03.
\end{abstract}

\renewcommand{\baselinestretch}{0.8}\normalsize
\tableofcontents
\renewcommand{\baselinestretch}{1.0}\normalsize

\section{Introduction and main results}
The purpose of this article is to study the
explicit quantification of the blow-up rate with respect to time
of the critical norm near a potential singularity for the Magneto-Hydrodynamics (MHD for short) equations, which is a branch of continuum mechanics that examines the flow of electrically conducting fluids under the influence of magnetic fields. Due to its numerous practical applications, such as magnetic separation and targeted delivery of drugs or radioisotopes via magnetic guidance, research on magnetohydrodynamics has attracted widespread attention. The movement of conductive fluids (such as liquid metal mercury, liquid sodium, and plasma) in a magnetic field is a result of the coupling of these two fundamental forces. Firstly, fluid motion generates an electric current that modifies the existing magnetic field. Secondly, the interaction between the current and the magnetic field produces a mechanical force in the fluid, accelerating it in a direction perpendicular to both the magnetic field and the current. Therefore, in mathematical physics, the MHD equations are derived by coupling the Navier-Stokes equations from fluid mechanics with Maxwell's equations from electromagnetic fields. For more physical background on the MHD equation, we refer the interested readers
to the reference \cite{MR1393572, MR98556, MR346289}. In three-dimensional space, the incompressible MHD equations on $(0,\infty)\times \R^3$ can be expressed as:
\begin{equation}\label{mhdeq}
\begin{cases}
\partial_{t}v-\Delta v+(v\cdot\nabla)v+\nabla\Pi=(H\cdot\nabla)H,\\
\partial_{t}H-\Delta H+(v\cdot\nabla)H-(H\cdot\nabla)v=0, \\
\div v=0,\quad\div H=0,
\end{cases}
\end{equation}
where $v: (0,\infty) \times \mathbb{R}^3 \to \R^3$, $H: (0,\infty) \times \mathbb{R}^3 \to \mathbb{R}^3$, $\Pi\triangleq \pi+\frac{|H|^2}{2}$, and $\pi: (0,\infty) \times \R^3 \to \R$  correspond to
the velocity field, magnetic field and pressure of the fluid, respectively. We note that, in the case $H\equiv0$, \eqref{mhdeq} reduces to the incompressible Navier-Stokes equation (NSE for short)
\begin{equation}\label{nseq}
\begin{cases}
\partial_{t}v+(u\cdot\nabla)v-\Delta v+\nabla \pi=0,\quad \text{in} \quad (0,\infty)\times \R^3\\
\div v=0,
\end{cases}
\end{equation}
which has a simple form but rich mathematical structure.  The above equation $\eqref{mhdeq}_1$ can be interpreted as the NSE \eqref{nseq} perturbed by an external force term $(H\cdot\nabla) H$, which is controlled by the linear equation $\eqref{mhdeq}_2$. Therefore, before describing our contribution, we first find it instructive to review the research progress on the singular solutions of Navier-Stokes equation \eqref{nseq}.

\subsection{Research progress of $3$D incompressible Navier-Stokes equations}

The research on singular solutions to the incompressible NSE \eqref{nseq} in $\R^3$ can be traced back to Leray's pioneering work in \cite{MR1555394}. In \cite{MR1555394}, Leray used the principle of extremal values for integral equations to prove that if the existence interval of a smooth solution $v$ is a finite interval $(0,T)$, where the first blow-up time of the solution is $t=T$, then it must exhibit the following blow-up behavior:
\begin{align}\label{ns.quan1}
\|v(t)\|_{L_x^p(\R^3)}\geq\frac{c(p)}{(T-t)^{\frac12(1-\frac3p)}},\quad p\in(3,\infty],
\end{align}
with some $c(p)$  depending only on $p$. However, it, for the critical case $p=3$,  is extremely complicated. In fact, the qualitative analysis about the $L^3(\R^3)$ norm of  potential singularity solutions  at the maximum existence time $T$  has been a well-known open problem for a long time, where the time $T$ is finite,
i.e. whether
\begin{align}\label{ns.case1}
\limsup\limits_{t\rightarrow T}\|v(t)\|_{L_x^3(\R^3)}=\infty
\end{align}
holds true or equivalently whether
\begin{align}\label{ns.case2}
v\in L^\infty(0,T; L^3(\R^3)) \Rightarrow v~\text{is~regular~at}~t=T.
\end{align}
 One of the important reasons for the complexity of this situation is that the so-called \emph{concentrated compaction} phenomenon might happen, that is, the condition $u\in L^\infty(0,T; L^3(\R^3))$ does not guarantee the fact that for any $ \varepsilon>0$, there exists $\Omega\subset \R^3$ makes $\|u\|_{L^\infty(0,T; L^3(\Omega))}<\varepsilon$. As a result of this phenomenon, one cannot directly prove  \eqref{ns.case1} or \eqref{ns.case2} is true by exploiting the usual regularity estimate of heat operator (for example the $L^p$-$L^q$ estimate). Until $2003$, Escauriaza, Seregin and \v{S}ver\'{a}k in their celebrated paper \cite{MR2005639}  suppressed  the concentration by using the rescaling procedure and  a backward uniqueness for parabolic operators and showed \eqref{ns.case1} is true.
 Precisely, the authors in \cite{MR2005639} assumed that \eqref{ns.case1} is not valid and then obtained a blow-up sequence via using Navier-Stokes rescaling\footnote{ if $(v,\pi)$ is the solution of the system \eqref{nseq}, then for any $\lambda>0$ the functions
\begin{align*}
 v^\lambda(t,x) \triangleq \lambda v(\lambda^2 t, \lambda x),~~\pi^\lambda(t,x) \triangleq \lambda^2 \pi(\lambda^2 t, \lambda x)
 \end{align*}
is also a solution of the system \eqref{nseq}.}. By using compactness techniques, they showed that the blow-up sequence converges to a nontrivial solution of the system \eqref{nseq}, often referred to as the limiting solution which satisfies the differential inequality corresponding to the heat operator due to Caffarelli-Kohn-Nirenberg type spatial localization technique \cite{MR673830}. Then, by employing the well-known Carleman inequality and the backward uniqueness of parabolic operators, they proved that the limiting solution must be zero. This is a contradiction. Throughout this proof, the localization technique in physical space and the Carleman inequality for parabolic operators play an important roles. Subsequently, the results of \cite{MR2005639} were extended to various critical space,
and the specific details can be found in \cite{MR3713543, MR3475661} and related  references. It is worth noting that the results mentioned in \cite{MR2005639, MR3475661}, and \cite{MR3713543} are qualitative and their proofs are derived by contradiction and compactness arguments.

It is natural to ask: for critical cases, is there a quantitative description of the singularity behavior of solutions similar to \eqref{ns.quan1}?  In a recent celebrated paper \cite{MR4337421}, Tao developed  a localization technique in frequency space and a quantitative version of the Carleman inequality from which he, by combining the pigeonhole principle, derived a explicit quantitative estimate for solutions of the NSE \eqref{nseq} belonging to the critical space $L^\infty(0,T; L^3(\R^3))$. As a result of this quantitative estimate, Tao showed in \cite{MR4337421} that if a finite energy solution $v$ (with Schwartz class initial data) that first loses smoothness at $T_\ast>0$ then
\begin{align}\label{ns.bur}
 \limsup_{t \to T_*^-} \frac{\|v(t)\|_{L^3_x(\R^3)}}{(\log\log\log( \frac{1}{T_*-t}))^c} = +\infty.
 \end{align}
with some small enough constant $c>0$.

 To illuminate the motivations of this paper in detail, we sketch Tao's strategy as follows (see Section 6 in \cite{MR4337421} for details). Assume that $u$ is a classical solution to the NSE \eqref{nseq} in $(0,T]\times\R^3$ satisfies
$$\|v\|_{L^\infty_t L^3_x((0,T) \times \R^3)} \leq A,$$
 with some sufficiently large absolute constant $A>C_0\gg 1$,  and there exists a universal constant $\varepsilon_0$ such that if the frequency is satisfied $N\geq N^\ast$ we have
$$N^{-1}\|P_N v\|_{L^\infty_t L^\infty_x([\frac T2,T] \times \R^3)}<\varepsilon_0,
$$
then $\|v\|_{L^\infty_t L^\infty_x([\frac78 T,T] \times \R^3)}$ can be estimated explicitly in terms of $A$ and $N^\ast$. Thus, the key to the problem is to find an upper bound on $N^\ast$. Employing the localization technique in frequency space and the quantitative version of the Carleman inequality, Tao showed in \cite{MR4337421} $N^\ast\simeq\exp(\exp(\exp(A^{C_0^7})))$ whose proof is divided into three steps:\\
$\mathbf{1)~Backward~frequency~bubbling.}$\\
Suppose $\|v\|_{L^\infty_t L^3_x([t_0-T,t_0] \times \R^3)}\leq A$ is such that $N_0^{-1}|P_{N_0}v(x_0,t_0)|>A^{-C_0}.$
Then for any $n\in\mathbb{N}$, there exists $N_n>0$ and $(t_n,x_n)\in (t_0-T,t_{n-1})\times\R^3$ such that
$$N_n^{-1}|P_{N_n}v(x_n,t_n)|>A^{-C_0},$$
with
$$x_n=x_0+O((t_0-t_n)^{\frac12}),\quad N_n\sim|t_0-t_n|^{-\frac12}.$$
$\mathbf{2)~The~vorticity~lower~bound~converted ~to~a~lower~bound~on~the~velocity.}$
\begin{itemize}
\item  {\em Transfer of concentration in Fourier space to physical space}. The previous step and $\|v\|_{L^\infty_t L^3_x((0,T) \times \R^3)}\leq A$ imply that for some small scales $S>0$ and $I_S\subset [t_0-S,t_0-A_3^{-O(1)}S]$ such that
\begin{align}\label{low bound-01}
 \int_{|x-x_0|\leq A_4^{O(1)} S^{\frac12}}|\nabla\times v(t',x)|^2\ \mathrm{d}x\geq A_3^{-O(1)}S^{-\frac12}\quad \text{for~all}\quad t'\in I_S
 \end{align}
 called as an enstrophy-type lower bound, where $A_j\triangleq A^{C_0^j}$ for all $j\in\mathbb{Z}^{+}$ and $A_{j+1}=A_j^{C_0}$.
\item {\em Large-scale  propagation of concentration.} Using quantitative versions of the Carleman inequalities which requires the ``epochs of regularity'', Tao showed that the enstrophy-type lower bound \eqref{low bound-01}
can be transferred from small scales $\{x:|x-x_0|\leq A_4^{O(1)} S^{\frac12}\}$ to large scale, i.e., for any $t'\in I_S$ and $S'=A^{-O(1)}S$, one has
\EQS{\label{low bound-02}
\int_{R<|x-x_0|\leq 2R}|\nabla\times v(t',x)|^2\ \mathrm{d}x\gtrsim \exp(-O(A_5^3R^3/S')(S')^{-\frac{1}{2}}
} with any $R\geq A_5S^{1/2}$.
\item  {\em Forward-in-time propagation of concentration.} With the help of the quantitative Carleman inequalities, one can propagate the lower bound on $I_S\times\{x: R<|x-x_0|\leq 2R\}$ forward in time until one returns  to original time $t_0$, which finally leads to
\begin{align}\label{ann.1}
 \int_{R_S<|x-x_0|<R'_S}|v(t_0,x)|^3\ \mathrm{d}x\geq \exp(-\exp(A_6^{O(1)})).
 \end{align}
\end{itemize}
$\mathbf{3)~Conclusion:~summing~scales~to~bound~TN_0^2.}$\\
Letting $S$ vary for certain permissible scale, the annuli in \eqref{ann.1} become disjoint. The sum of \eqref{ann.1} over such disjoint permissible annuli is bounded above by $\|v(t_0,x)\|_{L^3(\R^3)}$. This gives the desired bounded on $N_0$, i.e.
$$TN_0^2\lesssim \exp(\exp(\exp(A_6^{O(1)}))),$$
which implies that $N^\ast\simeq\exp(\exp(\exp(A^{C_0^7})))$. This, along with the classical energy method, yields
$$
\|v(t,x)\|_{L^\infty(\R^3)}\leq\exp(\exp(\exp(A^{C_0^7})))t^{-\frac12},\quad 0<t\leq T.
$$
Finally, assume by contradiction that \eqref{ns.bur} fails and take
$A=(\log\log\log( \frac{1}{T_*-t}))^c,$
one derives a contradiction by using the Prodi-Serrin-Ladyshenskaya criterion (see \cite{MR236541, MR126088, MR136885}). By now, a series of generalizations of Tao's work were promoted
 by Barkr, Prange \cite{MR4278282}, Palasek \cite{MR4334731, MR4502800}, and Hu et al. \cite{hu2024quantitativeboundsboundedsolutions}.
 Combining the local-in-space smoothing techniques (near the initial time) established by Jia and \v{S}ver\'{a}k \cite{MR3179576} with the quantitative Carleman inequality obtained by Tao \cite{MR4337421}, Barker and Prange in \cite{MR4278282} investigate the behavior of critical norms near a potential singularity to the solutions of \eqref{nseq} with Type I bound $\|v\|_{L_t^{\infty}L_x^{3,\infty}}\leq M$. Namely, if
$T^\ast$ is a first blow-up time and $(0,T^\ast)$ is a singular point, then
$$\|v\|_{L^{3}(B_R(0))}\geq C_M\log\Big(\frac{1}{T^\ast-t}\Big),\quad R=O({T^\ast-t}^{-\frac12}).$$
What's more, this blow-up rate is optimal for a class of potential non-zero backward discretely self-similar solutions. In addition, they quantified the result of Seregin \cite{MR2925135}, which say that if $v$ is a smooth finite-energy solution to the system \eqref{nseq} on $(-1,0)\times\R^3$ with
$$\limsup_n\|v(\cdot,t_{n})\|_{L^{3}(\R^3)}<\infty~\text{and}~t_{n}\uparrow 0,$$
 then for  $j\gg 1$
 $$
 \|v\|_{L^{\infty}\left(\R^3\times\big(\frac{t_{j+1}}{4},0\big)\right)}=O\Big(\frac{1}{\sqrt{-t_{j+1}}}\Big).
 $$  In \cite{MR4334731}, Palasek proved that if the solution of \eqref{nseq} satisfies the critical bound
 $$\|r^{1-\frac3{q}}v\|_{L^\infty_tL^q_x}\leq A,$$
 with $r=\sqrt{x_1^2+x_2^2}$, and $v,~q$ fall into one of two case:
 $$\text{either}~q\in(3,+\infty),~\text{or}~v~\text{is~axisymmetric~and}~q\in(2,3],
 \footnote{The condition that $v$ is an axisymmetric solution can actually be relaxed to its equivalence with another axisymmetric function. In other words, there exists an axisymmetric function $f:\R\times \R^3\rightarrow [0,\infty)$ and $C>0$ such that $C^{-1}f\leq|v|\leq Cf$.}$$
then the blow-up rate \eqref{ns.bur} can be improved  to
\begin{align*}
 \limsup_{t \to T_*^-} \frac{\|r^{1-\frac3{q}}v(t)\|_{L^q_x(\R^3)}}{(\log\log( \frac{1}{T_*-t}))^c} = +\infty
 \end{align*}
for some constant $c\in(0,+\infty)$.  Since then, Palasek \cite{MR4502800} got the quantitative regularity for solutions $v\in L^\infty_tL^d_x$ to the system \eqref{nseq} for case $d\geq 4$, which gives a quantification of the qualitative result obtained by Dong and Du \cite{MR2551795}. Very recently, Hu et al. \cite{hu2024quantitativeboundsboundedsolutions} studied the quantitative regularity and blow-up criterion of the classical solution to the NSE \eqref{nseq}  in $\dot{B}_{p,\infty}^{-1+\frac3{p}}(\R^3)~(3<p<\infty)$. Due to the low regularity in $\dot{B}_{p,\infty}^{-1+\frac3{p}}(\R^3)$, some new ideas are given to fix the related blocks in \cite{hu2024quantitativeboundsboundedsolutions}.
For more details in this direct, refer to \cite{MR4278282,hu2024quantitativeboundsboundedsolutions,MR4334731,MR4502800} and their references. Additionally, whether the blow-up rate \eqref{ns.bur} for this problem is optimal and whether the results can be extended to other critical spaces such as the Lorentz space $L^{3,\infty}(\R^3)$ remain open problems.

\subsection{Formulation of main results}

In this paper, as an attempt to understand Tao's idea in \cite{MR4337421}, we investigate  quantitative estimates for MHD equations. It is well-known \cite{MR716200} that the system \eqref{mhdeq} admits a local strong solution and a global energy weak solution for any given Schwartz class initial data. As in NSE \eqref{nseq}, the question of the regularity and uniqueness of weak solutions is still open. For convenience, we sketch  the research on the regularity criteria for solutions of the MHD system. The Ladyzhenskaya-Prodi-Serrin-type criteria to the MHD equations in terms of both the velocity field $v$ and the magnetic field $H$ is established by Wu in \cite{MR1915942, MR2026210}, which says: if the solution $(v, H)$ of the MHD equation satisfies
 $$\int_{0}^T(\|v\|^2_{L^\infty(\R^3)}+\|H\|^2_{L^\infty(\R^3)})\ \mathrm{d}t<\infty,$$
 then $(v, H)$ is smooth on $[0,T]$. Subsequently, people via some numerical experiments find that the velocity field should play a more important role than the magnetic field in the regularity theory of solutions, see for example \cite{MR1343429}. In fact, He and Xin \cite{MR2142366} and
Zhou \cite{MR2128731} have presented some regularity criteria to the MHD equations in terms of the velocity field only. Subsequently, Chen, Miao, and Zhang \cite{MR2336368, MR2452599} extended and improved upon the results of \cite{MR2142366} and \cite{MR2128731}. Additionally, Cao and Wu \cite{MR2595721} explored the scenario where the integrability condition is solely imposed on the directional derivative of the velocity field. They demonstrated that if the directional derivative of $v$, denoted as $\partial_{x_3} v$, satisfies the integrability condition
$$\partial_{x_3} v\in L^{p}(0,T; L^q(\R^3)),\quad\frac2p+\frac3q=1,\quad q\in(3,\infty),$$
then $(v, H)$ is regular in $\R^3 \times (0, T]$. For more details on the regularity criteria of weak solutions of \eqref{mhdeq},  we refer the reader to \cite{MR2334589,MR2165089,MR2398230} and the references therein.
For the limit case $L^\infty(0,T;L^3(\R^3))$, Mahalov, Nicolaenko and Shikin in \cite{MR2270882} considered the qualitative regularity of solution to \eqref{mhdeq}  and showed that
\begin{itemize}
\item  Suppose $(v,H,\pi)$ is a classical solution to MHD equations whose maximal time of existence $T_*$ is finite.  Then
\begin{align}\label{MNSbup1}
\limsup_{t \to T_*^-} \| (v,H)(t) \|_{L^3_x(\R^3)} = +\infty.
 \end{align}
 \end{itemize}
 Inspired by the breakthrough work of Tao \cite{MR4337421}, we quantify \eqref{MNSbup1}. Let us now state our first theorem
 \begin{thm}\label{main-1}  Let
$(v,H): [-1,1] \times \R^3 \to \R^3\times\R^3 $, $\pi: [-1,1] \times \R^3\to \R$ be a classical solution to the MHD equations with
\begin{equation}\label{u3}
 \|(v,H)\|_{L^\infty_t L^3_x([-1,1] \times \R^3)} \leq A
\end{equation}
for some $A>C_0 \gg 1$.  Then, for $j = 0,1,$  the following quantitative estimates hold
\EQS{\label{quanti-estimate}
 &|\nabla^j_x (v,H)(t,x)| \leq \exp\exp\exp(A^{C_0^8}), \\&
 |\nabla^j_x (\omega,J)(t,x)| \leq \exp\exp\exp(A^{C_0^8}),
 }
whenever $\frac78\leq t \leq 1$, $x \in \R^3$.
\end{thm}
\begin{remark}
In fact, the above quantitative estimates are valid for any $j\geq0$ by using Lemma \ref{theo.higher}.
\end{remark}

Theorem \ref{main-1}, along with MHD equations version of Ladyzhenskaya-Prodi-Serrin criteria, gives the following quantitative blow-up criterion.

\begin{thm}\label{main-2}  Let
$(v,H): [-1,1) \times \R^3 \to \R^3\times\R^3$, $\pi: [-1,1) \times \R^3 \to \R$ be a classical solution to the MHD equations which blows up at $t=1$.  Then
$$ \limsup_{t \to 1^-} \frac{\|(v,H)(t)\|_{L^3_x(\R^3)}}{(\log\log\log( \frac{1}{1-t}))^c} = +\infty,$$
for some constant $c\in(0,+\infty)$.
\end{thm}

 \subsection{New ingredients of the proof}
To obtain the quantitative estimates \eqref{quanti-estimate}, we, by using the same argument as \cite{MR4337421}, need to establish the upper bound for $N_0$\footnote{Under the scale-invariant assumption \eqref{u3}, if $N^{-1}\|P_{N}(v,H)\|_{L_{t}^{\infty}(0,1;L_{x}^{\infty}(\R^3))}<A^{-C}$ fails for $N=N_0$, what is an upper bound for $N_0$?}.
To gain the upper bound of $N_0$, the core of the proof  is to derive enstrophy-type lower bounds for the {\em corrector} function $W_c(t,x)$ which is defined below, the key ingredient is the quantitative Carleman inequalities. To exploit the quantitative Carleman inequality, we have to establish the related differential inequality. To this end, we need to derive the quantitative $L^\infty$ estimates for $u, \nabla u, \nabla H$, etc. However, to establish the related differential inequality, we have to overcome  blocks arose by inconsistent scales between  the magnetic field $H$ and the vorticity $w$.

For convenience, we first introduce the equations satisfied by $\nabla\times v(t,x)$ and $\nabla\times H(t,x)$, denoted by $\omega(t,x)$ and $J(t,x)$ respectively.  Now taking the curl on both sides of $\eqref{mhdeq}_1$ and $\eqref{mhdeq}_2$, one has
\begin{equation}\label{wJeq}
\begin{cases}
\partial_{t}\omega-\Delta \omega+(v\cdot\nabla)\omega-(\omega\cdot\nabla)v-(H\cdot\nabla)J
+(J\cdot\nabla)H=0,\\
\partial_{t}J-\Delta J+(v\cdot\nabla)J-(H\cdot\nabla)\omega-(J\cdot\nabla)v+(\omega\cdot\nabla)H=2R(v,H)
\end{cases}
\end{equation}
with
\begin{align*}
R(v,H)=\left(\begin{matrix}
 \partial_{3}v\cdot\partial_{2}H-\partial_{2}v\cdot\partial_{3}H\\
 \partial_{1}v\cdot\partial_{3}H-\partial_{3}v\cdot\partial_{1}H\\
 \partial_{2}v\cdot\partial_{1}H-\partial_{1}v\cdot\partial_{2}H\\
\end{matrix}
\right).
\end{align*}
This system can be viewed as a heat system with variable coefficients (in which the lower order coefficients v, H, $\nabla v$, $\nabla H$ depend on the velocity field and magnetic field).  However, it is impossible that
the pair ($\omega,J$) satisfies  the following differential inequality
$$\Big|(\partial_t-\Delta)(\omega, J)\Big|\leq C_{carl}^{-1}\widetilde{T}^{-1}|(\omega, J)|+C_{carl}^{-\frac12}\widetilde{T}^{-\frac12}|\nabla(\omega, J)|,$$
due to the additional term $R(v,H)$. Thus, we can not directly use system \eqref{wJeq} to derive the analogous  lower bound \eqref{ann.1}. To fixed this problem, we introduce $W=(H, \omega, \partial_{x_k}H)$, where $\partial_{x_k}H\triangleq H_{x_k},~(k=1,~2,~3)$ satisfies
\begin{equation}\label{Hxkeq-1}
 \partial_tH_{x_k}-\Delta H_{x_k}=(H_{x_k}\cdot\nabla)v
 +(H\cdot\nabla)v_{x_k}-(v_{x_k}\cdot\nabla)H-(v\cdot\nabla)H_{x_k}.
\end{equation}
 With the help of epochs of estimate and annuli of regularity for $v$ and $H$, we can obtain
\begin{align}\label{dif-con-1}
|(\partial_t-\Delta)W|&\leq \Big(C_{carl}^{-1}\widetilde{T}^{-1}
+C_{carl}^{-\frac32}\widetilde{T}^{-\frac32}\Big)|W| +\Big(C_{carl}^{-1}\widetilde{T}^{-1}
+C_{carl}^{-\frac12}\widetilde{T}^{-\frac12}\Big)|\nabla W|\nonumber\\
&\leq C_{carl}^{-\frac32}\widetilde{T}^{-\frac32}|W|
+C_{carl}^{-1}\widetilde{T}^{-1}|\nabla W|,
\end{align}
due to the fact that the scale of time $\widetilde{T}$ is  small. On the other hand, we notice that the MHD system \eqref{mhdeq}, as the NSE \eqref{nseq}, is invariant with respect to the following rescaling
$$
(v_{\lambda}(x,t), \pi_{\lambda}(x,t), H_{\lambda}(x,t)):=(\lambda v(\lambda x,\lambda^2t),\lambda^2 \pi(\lambda x,\lambda^2t),\lambda H(\lambda x,\lambda^2t))
$$
This scale-invariance property enable us to assign a ``dimension'' to the following quantities
\EQs{
\widetilde{T}^{-\frac{1}{2}}\,~\mbox{has dimension}\, -1; \quad
\partial_t \,~(\mbox{or}\,~\partial_x^2) \,~\mbox{has dimension}\, -2.
}
That is to say the scales (or dimension) of $\widetilde{T}^{-1}$ and $\partial_t$ (or $\partial_x^2$)  are equivalent. However, to exploit the quantitative version of the Carleman inequality, the dimensions of left and right sides of the differential inequality \eqref{lu} (or \eqref{weifen-3}) must be consistent  since the scale of time is small. Obviously, the scales of left and right sides of the differential inequality \eqref{dif-con-1} are inconsistent.
Thus we can not directly use quantitative Carleman inequalities for $W$ to derive  desired lower bound estimates. To overcome this block, some technical innovations are required. Precisely,
we perform a translation and scaling transformation to $W$ by setting
$$
v_\lambda (t,x)=\lambda v(t'_1-\lambda^2t,x_*+\lambda x),\quad H_\lambda (t,x)=\lambda H(t'_1-\lambda^2t,x_*+\lambda x),
$$
and
 $$
 (v_\lambda, H_\lambda)_{x_k} (t,x)=\lambda^2 (v_{y_k}, H_{y_k})(t'_1-\lambda^2t,x_*+\lambda x),\quad
(\omega_\lambda,J_\lambda) (t,x)=\lambda^2 (\omega,J)(t'_1-\lambda^2t,x_*+\lambda x),
$$
where $\lambda=\sqrt{\widetilde{T}}, y=x_*+\lambda x$ and $(t,x)\in[0,1]\times \R^3$. It is clear that $(v_\lambda, H_\lambda, (H_\lambda)_{x_k})$ is also a solution of \eqref{mhdeq}, \eqref{Hxkeq-1}, and
$W_{\lambda}\triangleq(H_\lambda, \omega_{\lambda}, (H_{\lambda})_{x_1}, (H_{\lambda})_{x_2}, (H_{\lambda})_{x_3})$ fulfills the following differential inequality
\EQS{\label{dif-con-2}
|\partial_t W_\lambda+\Delta W_\lambda|\leq \frac{1}{4} |W_\lambda|+\frac{1}{2}|\nabla W_\lambda|\quad\text{on}\quad {[0,1] \times \Omega}.
}
Here the space domain $\Omega$ is $\R^3$ or the given annulus, see Section \ref{6.upbound} for more details. The advantage of the differential inequality \eqref{dif-con-2} is that  the scale of time is normalised, and so the quantitative Carleman inequality is valid for $W_{\lambda}$. However, notices that $W_\lambda(x,t)\neq \lambda^2 W(y,t'_1-\lambda^2t)$ since scales are inconsistent between the magnetic field $H$ and the vorticity $w$, one can not derive the analogous enstrophy-type lower bound \eqref{low bound-02} for $W_{\lambda}$, which is the second main block of the paper.  The idea, fixing this difficulty, is  now to introduce a {\em corrector} function $W_c(t,x)$ defined by
$$
W_c(t,x)=T_3^{-1}|H(t,x)|^2+|\omega(t,x)|^2+\sum_{i=1}^{3}|H_{x_i}(t,x)|^2
$$
with $A_4^2N_0^{-2}\leq T_3\leq A_4^{-1}$. Using the quantitative Carleman inequality to $W_{\lambda}$ and the pigeonhole principle,  we, scaling back to the original variables, finally derive the exponentially small yet significant enstrophy-type lower bound for the {\em corrector} function $W_c(t,x)$ at the final moment of time $t_0$:
\begin{align*}
\int_{5\widetilde{R} \leq |x| \leq \frac{3A_6\widetilde{R}}{10}} W_c(t_0,x)\ \mathrm{d}x \gtrsim e^{-e^{A_6^{8}}} T_3^{-\frac{1}{2}},
 \end{align*}
 with $A_6T_3^{\frac12}\leq \widetilde{R}\leq e^{A_6^7} T_3^{\frac{1}{2}}$. Finally, the pigeonhole principle and H\"{o}lder's inequality enable us  convert the lower bound of $W_c(t_0,x)$ to the lower bound on $(v(t_0,x),H(t_0,x))$, i.e.,
 $$
\int_{B(\tilde{x}, \widetilde{r})} |v(t_0,x)|^3\ \mathrm{d}x +\int_{B(\tilde{x}, \widetilde{r})} |H(t_0,x)|^3\ \mathrm{d}x\gtrsim e^{-9e^{A_6^{11}}},
 $$
where $5\widetilde{R}\leq |\tilde{x}|\leq \frac{3A_6\widetilde{R}}{10}$ and $\widetilde{r}=e^{-e^{A_6^{11}}}T_3^{\frac12}$. The above lower bound finally leads to the the upper bound for $N_0$. Precisely,
\begin{prop}\label{main-est}
Let $(v,H,\pi): [-1,1] \times \R^3 \to \R^3\times \R^3\times\R$ be a classical solution of \eqref{mhdeq} satisfying
\begin{equation}\label{able}
 \|(v,H)\|_{L^\infty_t L^3_x([-1,1] \times \R^3)}\leq A.
\end{equation}
Suppose that there exists $(t_0,x_0)\in [0,1]\times\R^3$ and $N_0 > A_4$ such that
$$ |P_{N_0} (v,H)(t_0,x_0)| \geq A_1^{-1} N_0.$$
Then
$$ N_0 \leq e^{e^{e^{A_6^{12}}}}.$$
\end{prop}

\subsection{Notations}
\begin{enumerate}
\item we denote $\partial_t=\frac{\partial}{\partial t}, \partial_i=\frac{\partial}{\partial x_i}, \partial_{ij}=\frac{\partial^2}{\partial x_ix_j}$ and  adapt the notation
$$
\nabla^m~(\mbox{or}\,~D^m)=\partial_{x_1}^{m_1}\partial_{x_2}^{m_2}\partial_{x_3}^{m_3}\quad  \mbox{with} \quad m_1+m_2+m_3=m.
$$
\item Let $u,v$ be two vector fields, we define second order tensor product
$$
u\otimes v=(u_iv_j)_{1\leq i,j\leq 3}
$$
and $v\cdot\nabla u=v_i\cdot\partial_iu$. Throughout the paper, we use Einstein's convention for the sum on repeated indices.
\item The  notation $X = O(Y)$ means $X \lesssim Y$, i.e., there exists a positive constant $C$ such that $|X| \leq CY$.  We also use the notation $X \lesssim_j Y$, which means that $|X| \leq C_jY$ for some $C_j>0$ depending only $j$.
\item  The usual parabolic cylinder is defined  as
$$Q_r(a,s)\triangleq \{(x,t)~:~|x-a|<r,~s-r^2<t<s\}.$$
If $x_0 \in \R^3$ and $R > 0$, we use $B_R(x_0)$ to denote the ball $B(x_0, R)=\{ x \in \R^3: |x-x_0| \leq R \}$.
	\item We use the mixed Lebesgue norms
$$ \| v \|_{L^q_t L^r_x(I \times \Omega)} \triangleq \left( \int_I \| v(t) \|_{L^r_x(\Omega)}^q\ \mathrm{d}t \right)^{1/q}$$
where
$$ \| v(t) \|_{L^r_x(\Omega)} \triangleq  \left( \int_{\Omega} |v(t,x)|^r\ \mathrm{d}x \right)^{1/r}$$
with the usual modifications when $q=\infty$ or $r=\infty$. Here $\Omega\subset \R^3$ is a domain and $I\subset \R$ is a interval.
If $\Omega$ is a bounded domain,  we denote $C^{0,\alpha}(\overline{\Omega}),~0<\alpha\leq 1$, by the space of all uniformly $\alpha$-H\"{o}lder continuous functions
on $\overline{\Omega}$, i.e. functions $f:\overline{\Omega}\rightarrow\R$ for which there exists the constant $C>0$ such that
$$|f(x)-f (y)|\leq C|x-y|^\alpha,\quad\text{for~every}\quad x,y\in\overline{\Omega}.$$
The space $C^{0,\alpha}(\overline{\Omega})$ of $\alpha$-H\"{o}lder continuous functions on $\overline{\Omega}$ with norm
$$\|f\|_{C^{0,\alpha}(\overline{\Omega})}\triangleq \sup_{x\in\overline{\Omega}}|f(x)|+\sup_{x,y\in\overline{\Omega}}
\frac{|f(x)-f(y)|}{|x-y|^\alpha}$$
is a Banach space. The space $C^k(\overline{\Omega})$, of all $k$-times continuously differentiable functions with derivatives up to order $k$ continuous on $\overline{\Omega}$ is a Banach space when equipped with the norm
$$\|f\|_{C^{k}(\overline{\Omega})}\triangleq \sum_{|\alpha|\leq k}\|\nabla^\alpha f\|_{C^{0}(\overline{\Omega})}.$$
$C^{k,\alpha}{(\overline{\Omega})},~0< \alpha \leq 1$, consists of all functions $f\in C^k (\overline{\Omega})$ for which all the $k$th
derivatives are H\"{o}lder continuous with exponent $\alpha$, i.e. $\partial^\gamma f\in C^{0,\alpha}(\overline{\Omega})$ for
every multi-index $\gamma$ with $|\gamma|=k$.

  \item Given a Schwartz function $f: \R^3 \to \R$, we define the Fourier transform
$$ \hat f(\xi) \triangleq \int_{\R^3} f(x) e^{-2\pi i \xi \cdot x}\ \mathrm{d}x\quad \Big(\text{or}~\mathcal{F}(f)(\xi) \triangleq \int_{\R^3} f(x) e^{-2\pi i \xi \cdot x}\ \mathrm{d}x\Big),$$
and then for any $N>0$ we define the Littlewood-Paley projection $P_{\leq N}$ by the formula
$$ \widehat{P_{\leq N} f}(\xi) \triangleq \varphi(\frac{\xi}{N}) \hat f(\xi)$$
where $\varphi: \R^3 \to \R$ is a fixed bump function supported on $B(0,1)$ that equals $1$ on $B(0,\frac{1}{2})$.  We also define the companion Littlewood-Paley projections
\begin{align*}
P_N \triangleq P_{\leq N} - P_{\leq\frac{N}{2}}, \quad
P_{>N} \triangleq \mathbb{I} - P_{\leq N}, \quad
\widetilde{P}_N \triangleq P_{\leq2N} - P_{\leq\frac{N}{4}}
\end{align*}
where $\mathbb{I}$ denotes the identity operator; thus for instance
$$P_{\leq N} f= \sum_{k=0}^\infty P_{2^{-k} N} f,\quad P_{>N} f = \sum_{k=1}^\infty P_{2^k N} f$$
for Schwartz $f$.  Also we have
$$P_N = P_N \widetilde{P}_N.$$
These operators can also be applied to vector-valued Schwartz functions by working component by component.  These operators commute with other Fourier multipliers such as the Laplacian $\Delta$ and its inverse $\Delta^{-1}$, partial derivatives $\partial_i$, heat propagators $e^{t\Delta}$, and the Leray projection
$$\mathbb{P} \triangleq - \nabla \times \Delta^{-1} \nabla \times$$
to divergence-free vector fields. It is important to emphasize that if $N'\sim N''<< N$, then
\begin{align}\label{low20}
\widetilde{P}_{N}(P_{\leq N'}v P_{\leq N''}v)=0.
\end{align}
In fact,
\begin{align*}
&\mathcal{F}\Big[\widetilde{P}_{N}(P_{\leq N'}v P_{\leq N''}v)\Big](\xi)\nonumber\\
&\quad=\mathcal{F}
\Big[\widetilde{P}_{N}(P_{\leq N'}v_i P_{\leq N''}v_j)\Big]\nonumber\\
&\quad=\Big[\varphi(\frac{\xi}{2N'})-\varphi(\frac{4\xi}{N''})\Big]
\mathcal{F}\Big(
P_{\leq N'}v_iP_{\leq N''}v_j\Big)\nonumber\\
&\quad=\Big[\varphi(\frac{\xi}{2N})-\varphi(\frac{4\xi}{N})\Big]\int_{\R^3}
[\mathcal{F}(
P_{\leq N'}v_i)](\xi-\eta)[
\mathcal{F}(P_{\leq N''}v_j)](\eta)\ \mathrm{d}\eta\nonumber\\
&\quad=\Big[\varphi(\frac{\xi}{2N})-\varphi(\frac{4\xi}{N})\Big]\int_{\R^3}
\Big[\varphi\Big(\frac{\xi-\eta}{N'}\Big)\mathcal{F}(v_i)\Big]
\Big[\varphi\Big(\frac{\eta}{N''}\Big)\mathcal{F}(v_j)\Big]\ \mathrm{d}\eta\nonumber\\
&\quad=\Psi_1\cdot\Psi_2,
\end{align*}
with
$$\Psi_1(\xi)\triangleq \varphi(\frac{\xi}{2N})-\varphi(\frac{4\xi}{N}),\quad
\Psi_2(\xi)\triangleq \int_{\R^3}
\Big[\varphi\Big(\frac{\xi-\eta}{N'}\Big)\mathcal{F}(v_i)\Big]
\Big[\varphi\Big(\frac{\eta}{N''}\Big)\mathcal{F}(v_j)\Big]\ \mathrm{d}\eta.$$
It is clear that
$$\mathrm{supp} \Psi_1(\xi)\cap \mathrm{supp} \Psi_2(\xi)=\varnothing$$
due to $N'\sim N''<< N$. This leads to \eqref{low20}.
\end{enumerate}

\subsection{Plan of the paper}
In Section \ref{2.pre}, we discuss the various  tools used in this paper, such as the multiplier theorem, basic estimates of heat operator, the the quantitative $\epsilon$-regularity for higher order derivative, the quantitative Carleman inequality, etc.
Section \ref{7.proof} presents the proof of the main results of this paper.
Section \ref{6.upbound} derives the  upper bound of $N_0$ under the $L^3$ critical bounds of $(v,H)$ which is the core of this paper. In particular:
\begin{itemize}
\item In Section \ref{3.basic}, we present the basic estimates, which are built upon the breakthrough work by Tao \cite{MR4337421}.
\item In Section \ref{4.epoch}, we introduce new methods for proving the epoch regularity estimates of $v$ and $H$ and the annuli regularity estimates, both of which are crucial for proving the main results in Section \ref{6.upbound}.
\item In Section \ref{5.bubble}, we construct a sequence of frequency bubbles, which are essential for proving the upper bound of $N_0$.
\item In Section \ref{pf-N0}, we devote to the proof of Proposition \ref{main-est}. The core of the proof is that we transform the lower bound of $W_c$ into that of $v$ and $H$.  Notice that unlike Tao's method, which relies on the treatment of the vorticity equation, the direct application of Tao's method to convert the concentration compactness of $\omega$ and $J$ into that of $v$ and $H$ is no longer valid due to the influence of $R(v,H)$ in $\eqref{wJeq}_2$. To solve this problem, we consider the equation of $\nabla H$ and make use of the scaling invariant property of the system \eqref{mhdeq} to fix the blocks due to the fact that scales are inconsistent between the magnetic field $H$ and the vorticity $w$.
\end{itemize}
In the final Appendix \ref{app.a}, we primarily  devote to the proof of Lemma \ref{theo.higher}.

\section{Preliminaries}\label{2.pre}
Due to the localized technique in frequency space adopted in this paper, the multiplier theorem and the heat kernel estimate below will play an important role in this paper. We, for their  proofs, refer the interested reader to \cite{MR4337421}.

\subsection{Multiplier theorem, some basic estimates of the heat operator and the quantitative Carleman inequality}
\begin{lem}\textsuperscript{{\rm \cite{MR4337421}}}[Multiplier theorem]\label{mult}
Let $N>0$, and let $m: \R^3 \to \mathbb{C}$ be a smooth function supported on $B(0,N)$ that obeys the bounds
$$ |\nabla^j m(\xi)| \leq M N^{-j}$$
for all $0 \leq j \leq 100$ and some positive constants $M$.  Let $T_m$ denote the associated Fourier multiplier, i.e.
$$ \widehat{T_m f}(\xi) \triangleq m(\xi) f(\xi).$$
Then one has
\begin{equation*}
 \| T_m f \|_{L^q(\R^3)} \lesssim M N^{\frac{3}{p}-\frac{3}{q}} \|f\|_{L^p(\R^3)}
\end{equation*}
whenever $1 \leq p \leq q \leq \infty$ and $f: \R^3 \to \mathbb{R}$ is a Schwartz function. In particular, if $\Omega \subset \R^3$ is an open subset of $\R^3$, $A \geq 1$, and
$$\Omega_{\frac{A}{N}} \triangleq \{ x \in \R^3: \mathrm{dist}(x,\Omega) < \frac{A}{N}\},$$
then we have
\begin{equation}\label{local}
 \| T_m f \|_{L^{q_1}(\Omega)} \lesssim M N^{\frac{3}{p_1}-\frac{3}{q_1}} \|f\|_{L^{p_1}(\Omega_{\frac{A}{N}})} + A^{-50} M |\Omega|^{\frac{1}{q_1}-\frac{1}{q_2}} N^{\frac{3}{p_2}-\frac{3}{q_2}} \|f\|_{L^{p_2}(\R^3)}
\end{equation}
 whenever $1 \leq p_1 \leq q_1 \leq \infty$ and $1 \leq p_2 \leq q_2 \leq \infty$ are such that $q_2 \geq q_1$, and $|\Omega|$ denotes the volume of $\Omega$.
Thus for instance, we have the Bernstein inequalities
\begin{equation}\label{bern}
 \| \nabla^j f \|_{L^q(\R^3)} \lesssim_j N^{j + \frac{3}{p} - \frac{3}{q}} \|f\|_{L^p(\R^3)}
\end{equation}
whenever $1 \leq p \leq q \leq \infty$, $j \geq 0$, and $f$ is a Schwartz function whose Fourier transform is supported on $B(0,N)$, and from this, we drive
\begin{equation}\label{bern-2}
 \| P_N e^{t\Delta} \nabla^j f \|_{L^q(\R^3)} \lesssim_j \exp( - N^2 t / 20 ) N^{j + \frac{3}{p} - \frac{3}{q}} \|f\|_{L^p(\R^3)},
\end{equation}
for any $t>0$.  Summing this, we obtain the standard heat kernel bounds
\begin{equation}\label{bern-3}
 \| e^{t\Delta} \nabla^j f \|_{L^q(\R^3)} \lesssim_j t^{-\frac{j}{2} - \frac{3}{2p} + \frac{3}{2q}} \|f\|_{L^p(\R^3)}.
\end{equation}
\end{lem}
In the process of proving Lemma \ref{theo.higher}, the Lemmas \ref{lem.heat1}-\ref{lem.heat2-1} below are important tools. The proof of Lemma \ref{lem.heat1}-\ref{lem.heat2-1} are presented on page 397 of \cite{MR3616490}, and we omitted it's proof.
\begin{lem}\textsuperscript{{\rm \cite{MR3616490}}}\label{lem.heat1}
Assume that $l,~l',~r,~r'$ satisfy either
\begin{align*}
\begin{cases}
1\leq l\leq r\leq\infty, & 1<l'\leq r'<\infty,\\
\frac{n}{l}+\frac{2}{l'}\leq\frac{n}{r}+\frac{2}{r'}+2,
\end{cases}
\end{align*}
or
\begin{align*}
\begin{cases}
1\leq l\leq r\leq\infty, & 1<l'\leq r'=\infty,\\
\frac{n}{l}+\frac{2}{l'}<\frac{n}{r}+\frac{2}{r'}+2.
\end{cases}
\end{align*}
Then for any $T\in(0,\infty)$, there exists $c=c(n,~l,~l',~r,~r')$ such that
$$\|(\partial_t-\Delta)^{-1}f\|_{L_t^{r'}L_x^{r}((-T,0]\times \R^3)}\leq c\|f\|_{L_t^{l'}L_x^{l}((-T,0]\times \R^3)}.$$
\end{lem}

\begin{lem}\textsuperscript{{\rm \cite{MR3616490}}}\label{lem.heat2}
Assume that $m,~m',~r,~r'$ satisfy either
\begin{align*}
\begin{cases}
1\leq m\leq r\leq\infty, & 1<m'\leq r'<\infty,\\
\frac{n}{m}+\frac{2}{m'}\leq\frac{n}{r}+\frac{2}{r'}+1,
\end{cases}
\end{align*}
or
\begin{align*}
\begin{cases}
1\leq m\leq r\leq\infty, & 1<m'\leq r'=\infty,\\
\frac{n}{m}+\frac{2}{m'}<\frac{n}{r}+\frac{2}{r'}+1.
\end{cases}
\end{align*}
Then for any $T\in(0,\infty)$, there exists $c=c(n,~m,~m',~r,~r')$ such that
$$\|(\partial_t-\Delta)^{-1}\nabla f\|_{L_t^{r'}L_x^{r}((-T,0]\times \R^3)}\leq c\|f\|_{L_t^{m'}L_x^{m}((-T,0]\times \R^3)}.$$
\end{lem}

\begin{lem}\textsuperscript{{\rm \cite{MR3616490}}}\label{lem.heat2-1}
{\rm (i)} Let $f, g\in L_t^{\infty}L_x^{\infty}((-T,0]\times \R^3)$ and $T\in(0,\infty)$, then
\begin{align*}
&\|(\partial_t-\Delta)^{-1}(\nabla f+g)\|_{L_t^{\infty}C_x^{0,\alpha}((-T,0]\times \R^3)}\\
&\quad\leq c(\|f\|_{L_t^{\infty}L_x^{\infty}((-T,0]\times \R^3)}+\|g\|_{L_t^{\infty}L_x^{\infty}((-T,0]\times \R^3)})
\end{align*} 
for any $0<\alpha<1$.

{\rm (ii)}  Let $f, g\in L_t^{\infty}C^{k,\alpha}_x((-T,0]\times \R^3)$ with $k=0,1,2,\cdots$, then
\begin{align*}
&\|(\partial_t-\Delta)^{-1}(\nabla f+g)\|_{L_t^{\infty}C_x^{k+1}((-T,0]\times \R^3)}\\
&\quad\leq c(\|f\|_{L_t^{\infty}C^{k,\alpha}_x((-T,0]\times \R^3)}
+\|g\|_{L_t^{\infty}C^{k,\alpha}_x((-T,0]\times \R^3)})
\end{align*} 
for any $0<\alpha<1$.
\end{lem}

Lemmas \ref{carl-first}-\ref{carl-second} are the quantitative Carleman inequalities defined respectively on the circle and the sphere which play a key role in derivation of upper bound of $N_0$. Here, we don't plan to provide their proof due to limited the length of paper. we refer the interesting reader to \cite{MR4337421} (see Section 4).

\begin{lem}\textsuperscript{{\rm \cite{MR4337421}}}(First Carleman inequality)\label{carl-first}  Let $C_{carl}\in[1,\infty)$,
$T\in[1,\infty)$, and $0 < r_- < r_+$, where $C_{carl}$ is called Carleman's constant. Let ${\mathcal A}$ denote the space-time annulus
$$ {\mathcal A} \triangleq\{ (t,x) \in \R \times \R^3: t \in [0,T]; r_- \leq |x| \leq r_+ \}.$$
Let $U: {\mathcal A} \to \R^3$ be a smooth function and such that $U$ satisfies the differential inequality
\begin{equation}\label{lu}
 |\partial_tU+\Delta U| \leq  C_{carl}^{-1}T^{-1} |U| + C_{carl}^{-\frac12}T^{-\frac12} |\nabla U|\quad \text{on}\quad {\mathcal A}.
\end{equation}
Assume the inequality
\begin{align*}
r_-^2 &\geq 4 C_{carl} T. \label{sigma-2}
\end{align*}
Then one has the following bound
\begin{align*}
&\int_{0}^{\frac{T}4} \int_{10r_- \leq |x| \leq \frac{r_+}{2}}  T^{-1} |U|^2 + |\nabla U|^2\ dx dt \nonumber\\
&\quad\lesssim C_{carl}^2 e^{-\frac{r_- r_+}{4C_{carl} T}} \Big[\int\int_{\mathcal A} e^{\frac{2 |x|^2}{C_{carl} T}} \Big(T^{-1} |U(t,x)|^2 + |\nabla U(t,x)|^2\Big)\ dx dt \nonumber\\
&\quad\quad+ e^{\frac{2 r_+^2}{C_{carl} T}} \int_{r_- \leq |x| \leq r_+} |U(0,x)|^2\ dx\Big].
\end{align*}
\end{lem}
\begin{lem}\textsuperscript{{\rm \cite{MR4337421}}}(Second Carleman inequality)\label{carl-second}  Let $C_{carl}\in[1,\infty)$,
$T\in[1,\infty)$, $r>0$, and we define the cylindrical region
$$ {\mathcal C} \triangleq\{ (t,x) \in \R \times \R^3: t \in [0,T]; |x| \leq r \}.$$
Let $U: {\mathcal A} \to \R^3$ be a smooth function and such that $U$ satisfies the differential inequality
\begin{equation} \label{weifen-3}
 |\partial_tU+\Delta U| \leq  C_{carl}^{-1}T^{-1} |U| + C_{carl}^{-\frac12}T^{-\frac12} |\nabla U|\quad \text{on}\quad {\mathcal C}.
\end{equation}
Assume
\begin{align*}
r^2 &\geq 4000 T. \label{sigma-5}
\end{align*}
Then, for any $0 < t_1 \leq t_0 < \frac{T}{10000}$ one has the bound
 \begin{align*}
&\int_{t_0}^{2t_0} \int_{|x| \leq \frac{r}2} (T^{-1} |U|^2 + |\nabla U|^2) e^{-\frac{|x|^2}{4t}} \ dx dt\nonumber\\
&\quad\lesssim e^{-\frac{r^2}{500 t_0}}\int_0^T \int_{|x| \leq r} T^{-1} |U(t,x)|^2 +|\nabla U(t,x)|^2\ dx dt\nonumber\\
&\quad\quad+ \Big(\frac{t_0}{t_1}\Big)^{\frac32} (\frac{et_0}{t_1})^{\frac{Cr^2}{t_0}} \int_{|x| \leq r} |U(0,x)|^2 e^{-\frac{|x|^2}{4t_1}}\ dx.
\end{align*}
\end{lem}

\subsection{Quantitative version of $\epsilon$-regularity for MHD sysstem}

This subsection mainly devote to the quantitative estimates for for the higher-order derivatives of $(v,H)$ (see the Corollary \ref{cor1.fulu} for more details), which are key ingredients in the proof of Proposition \ref{vi} regarding the annuli of estimation for $(v, H)$. First, we present the following Lemma \ref{theo.smaller} which can be obtained directly by using the Caffarelli-Kohn-Nirenberg type iteration, and we will not elaborate on its proof here, and refer the interested reader to page 291 of \cite{MR3616490}.

\begin{lem}\textsuperscript{{\rm \cite{MR3616490}}}\label{theo.smaller}
There exist absolute constants $\varepsilon_*>0$ and $C>0$ such that if $(v, H, \Pi)$ is a suitable weak solution of the MHD equations \eqref{mhdeq} on $Q_1(0,0)$ and for some $\varepsilon<\varepsilon_*$
\begin{equation*}
\int\limits_{Q_1(0,0)}|v|^3+|H|^3+|\Pi|^\frac32\ \mathrm{d}x\mathrm{d}s\leq\varepsilon,
\end{equation*}
then
\begin{equation}\label{e.morreybound}
\|v\|_{L^\infty_tL^\infty_x(Q_{\frac12}(0,0))}
+\|H\|_{L^\infty_tL^\infty_x(Q_{\frac12}(0,0))}
\leq C\varepsilon^\frac13.
\end{equation}
Moreover, the following qualitative estimate holds  for any positive integer $k$,
\begin{align}\label{quali-L-1}
\max_{z\in Q_{\frac18}(0,0)}|\nabla^k (v,H)(z)|\leq C_k.
\end{align}
\end{lem}

In fact, the quantitative estimate \eqref{e.morreybound} is valid for $(\nabla^kv, \nabla^kH)$ for any $k>0$ by taking advantage of  the following lemma, i.e., the above $C_k$ in \eqref{quali-L-1} can be quantified.

\begin{lem}\label{theo.higher}
Let $(v, H, \Pi)$ be a distributional solution to the MHD equation \eqref{mhdeq} in $Q_{\frac12}(0,0)$. Furthermore, suppose $v,~H\in L^\infty(Q_{\frac12}(0,0))$, $\omega, J\in L^2(Q_{\frac12}(0,0))$ with
\begin{align}\label{small-v-H}
\|v\|_{L^\infty_tL^\infty_x(Q_{\frac12}(0,0))}+\|H\|_{L^\infty_tL^\infty_x(Q_{\frac12}(0,0))}
+\|\omega\|_{L^2_tL^2_x(Q_{\frac12}(0,0))}+\|J\|_{L^2_tL^2_x(Q_{\frac12}(0,0))}
<1.
\end{align}
Then
\EQS{\label{quanti-L-1}
&\|\nabla^k(v, H)\|_{L^\infty_tL^\infty_x(Q_{\frac{1}{8}}(0,0))}\\
&\quad\leq C'_k\Big(\|v\|_{L^\infty_tL^\infty_x(Q_{\frac12}(0,0))}+\|H\|_{L^\infty_tL^\infty_x
(Q_{\frac12}(0,0))}\\
&\quad\quad+\|\omega\|_{L^2_tL^2_x(Q_{\frac12}(0,0))}
+\|J\|_{L^2_tL^2_x(Q_{\frac12}(0,0))}\Big).
}
for any integer $k\geq0$. Here $C_k\in(0,\infty)$ is a universal constant.
\end{lem}

 \begin{remark}
 The quantitative version of Barker-Prange's result for Navier-Stokes system was presented in \cite{MR4278282} without proof.  Here, for  completeness, we give a full proof in Appendix \ref{Sec5}  by exploiting the localisation procedure and $L^q$-$L^p$ estimates of the heat operator which is different from the Serrin's procedure and of independent interest. However, compared with Barker-Prange's result, our quantitative estimates are improved by using
 the velocity field $v$ in place of vorticity field $\omega$ in the left-hand side of \eqref{quanti-L-1}. On the other hand, our quantitative result extend Barker-Prange's estimate from $j=1,2$ to any $j\geq1$.
 \end{remark}

\medskip

Combining  Lemmas \ref{theo.smaller} and \ref{theo.higher}, we immediately derive the following  corollary which is a core result of this subsection.

\begin{cor}\label{cor1.fulu}
There exist absolute constants $\varepsilon'_*\in(0,1)$  such that if $(v, H, \Pi)$ is a suitable weak solution of the MHD equations \eqref{mhdeq} on $Q_1(0,0)$ and for some $\varepsilon'<\varepsilon'_*$
\begin{equation*}
\int_{Q_1(0,0)}|v|^3+|H|^3+|\Pi|^\frac32\ \mathrm{d}x\mathrm{d}s\leq\varepsilon',
\end{equation*}
then for any positive integer $k>0$
\begin{equation*}
\|\nabla^kv\|_{L^\infty_tL^\infty_x(Q_{\frac{1}{8}}(0,0))}
+\|\nabla^kH\|_{L^\infty_tL^\infty_x(Q_{\frac{1}{8}}(0,0))}
\leq C'_k(\varepsilon')^\frac13
\end{equation*}
for some universal constant $C'_k>0$.
\end{cor}
\begin{proof}
First, by Lemma \ref{theo.smaller} we can conclude that there exists a constant $C>0$ such that
\begin{align}\label{v-H-small-1}
\|v\|_{L^\infty_tL^\infty_x(Q_{\frac12}(0,0))}
+\|H\|_{L^\infty_tL^\infty_x(Q_{\frac12}(0,0))}
\leq C(\varepsilon')^\frac13.
\end{align}
This, along with the following interpolation inequality and \eqref{v-H-small-1}
\begin{align*}
\|h\|_{W^{1,2}_x(B_{\frac12}(0))}\leq C
\|h\|_{L^2_x(B_{\frac12}(0))}^{\frac12}\|h\|^{\frac12}_{W^{2,2}_x(B_{\frac12}(0))},
\quad \text{for~all}~ h\in W^{2,2}_x(B_{\frac12}(0)),
\end{align*}
yields
\begin{align*}
\|\nabla (v,H)\|_{L^2_tL^2_x(Q_{\frac12}(0,0))}\leq C
\|(v,H)\|_{L^\infty_tL^\infty_x(Q_{\frac12}(0,0))}^{\frac12}
\|(v,H)\|_{L^2_tW^{2,2}_x(Q_{\frac12}(0,0))}
^{\frac12}\leq C(\varepsilon')^\frac13,
\end{align*}
 which implies \eqref{small-v-H} is holds. Together with Lemma \ref{theo.higher}, we derive that for any any positive integer $k>0$
 \begin{align*}
&\|\nabla^k(v, H)\|_{L^\infty_tL^\infty_x(Q_{\frac1{8}}(0,0))}\\
&\quad\leq C'_k(\|v\|_{L^\infty_tL^\infty_x(Q_{\frac12}(0,0))}+\|H\|_{L^\infty_tL^\infty_x
(Q_{\frac12}(0,0))}
+\|\omega\|_{L^2_tL^2_x(Q_{\frac12}(0,0))}+\|J\|_{L^2_tL^2_x(Q_{\frac12}(0,0))})\\
&\quad\leq C'_k(\varepsilon')^\frac13,
\end{align*}
which finishes the proof of Corollary \ref{cor1.fulu}.
\end{proof}

\section{The Proof of the Theorem \ref{main-1} and Theorem \ref{main-2}}\label{7.proof}
This section is devoted to the proof of Theorem \ref{main-1} and Theorem \ref{main-2}. Firstly, by using the upper bound estimate of $N_0$ which proof will be postponed to the Section \ref{6.upbound}, and the energy method
we will prove the Theorem \ref{main-1}. Secondly, with the help of the quantitative estimates of Theorem \ref{main-1}, we finish the proof of the Theorem \ref{main-2}


\begin{proof}[Proof of Theorem \ref{main-1}]
 First we note that, due to the Proposition \ref{main-est},
\begin{equation}\label{piano}
 \| P_{N} (v,H) \|_{L^\infty_t L^\infty_x([\frac{1}{2}, 1] \times \R^3)} \leq A_1^{-1} N
\end{equation}
whenever $N \geq N_*$ with
$$
N_* \triangleq\exp(\exp(\exp(A^{C_0^7}))).
$$
In the following, we will use the  classical energy method to derive the desired result. However, the solution $(v,H)$ don't belong to the energy function space. To fix this difficult,  we split $v = v^{\mathrm{l}} + v^{\mathrm{n}}$, $H = H^{l} + H^{n}$ on $[\frac{1}{2},1] \times \R^3$, where
\begin{equation*}
 v^{l}(t) \triangleq e^{(t+1)\Delta} v(-1,x),~H^{l}(t) \triangleq e^{(t+1)\Delta} H(-1,x).
\end{equation*}
This, along with \eqref{able} and \eqref{bern-3} yields
\begin{equation}\label{very}
 \| \nabla^j (v^{l},H^{l}) \|_{L^\infty_t L^p_x([-\frac{1}{2},1] \times \R^3)} \lesssim_j A,
\end{equation}
for any $3 \leq p \leq \infty$ and $j \geq 0$. Now let us denote the nonlinear component by $v^{n} \triangleq v - v^{l}, H^{n} \triangleq H - H^{l}$, and similarly split $\omega = \omega^{l} + \omega^{n}, J= J^{l} + J^{n}$. For convenience,  we introduce the nonlinear enstrophy
\EQS{\label{enstro-01}
\mathbb{E}(t) \triangleq\frac{1}{2} \int_{\R^3} |\omega^{n}(t,x)|^2+ |J^{n}(t,x)|^2\ \mathrm{d}x
}
for $t \in [\frac{1}{2},1]$.  From the equation \eqref{wJeq} and integration by parts we have
\begin{equation*}
\frac{\mathrm{d}}{\mathrm{d}t} \mathbb{E}(t) = -\mathbb{Y}_1(t) + \mathbb{Y}_2(t) + \mathbb{Y}_3(t) +\mathbb{ Y}_4(t) + \mathbb{Y}_5(t)+\mathbb{Y}_6(t),
\end{equation*}
where
\begin{align*}
\mathbb{Y}_1(t) &= \int_{\R^3} |\nabla \omega^{n}(t,x)|^2+ |\nabla J^{n}(t,x)|^2\ \mathrm{d}x\\
\mathbb{Y}_2(t) &= \int_{\R^3} \omega^{n} \cdot (H \cdot \nabla) J^{l}\ \mathrm{d}x
+\int_{\R^3} J^{n} \cdot (H \cdot \nabla) \omega^{l}\ \mathrm{d}x\\
&\quad- \int_{\R^3} \omega^{n} \cdot (v \cdot \nabla) \omega^{l}\ \mathrm{d}x
- \int_{\R^3} J^{n} \cdot (v \cdot \nabla) J^{l}\ \mathrm{d}x\\
\mathbb{Y}_3(t) &= \int_{\R^3} \omega^{n} \cdot (\omega^{n} \cdot \nabla) v^{l}\ \mathrm{d}x
+\int_{\R^3} J^{n} \cdot (J^{n} \cdot \nabla) v^{l}\ \mathrm{d}x\\
&\quad+2\int_{\R^3} T(v^{n},H^{l})\cdot J^{n}\ \mathrm{d}x
- \int_{\R^3} \omega^{n} \cdot (J^{n} \cdot \nabla) H^{l}\ \mathrm{d}x\\
&\quad- \int_{\R^3} J^{n} \cdot (\omega^{n} \cdot \nabla) H^{l}\ \mathrm{d}x\\
\mathbb{Y}_4(t) &= \int_{\R^3} \omega^{n} \cdot (\omega^{l} \cdot \nabla) v^{n}\ \mathrm{d}x
+\int_{\R^3} J^{n} \cdot (J^{l} \cdot \nabla) v^{n}\ \mathrm{d}x\\
&\quad+2\int_{\R^3} T(v^{l},H^{n})\cdot J^{n}\ \mathrm{d}x
- \int_{\R^3} \omega^{n} \cdot (J^{l} \cdot \nabla) H^{n}\ \mathrm{d}x\\
&\quad- \int_{\R^3} J^{n} \cdot (\omega^{l} \cdot \nabla) H^{n}\ \mathrm{d}x\\
\mathbb{Y}_5(t) &= \int_{\R^3} \omega^{n} \cdot (\omega^{l} \cdot \nabla) v^{l}\ \mathrm{d}x
+\int_{\R^3} J^{n} \cdot (J^{l} \cdot \nabla) v^{l}\ \mathrm{d}x\\
&\quad+2\int_{\R^3} T(v^{l},H^{l})\cdot J^{n}\ \mathrm{d}x
- \int_{\R^3} \omega^{n} \cdot (J^{l} \cdot \nabla) H^{l}\ \mathrm{d}x\\
&\quad- \int_{\R^3} J^{n} \cdot (\omega^{l} \cdot \nabla) H^{l}\ \mathrm{d}x\\
\mathbb{Y}_6(t) &= \int_{\R^3} \omega^{n} \cdot (\omega^{n} \cdot \nabla) v^{n}\ \mathrm{d}x
+\int_{\R^3} J^{n} \cdot (J^{n} \cdot \nabla) v^{n}\ \mathrm{d}x\\
&\quad+2\int_{\R^3} T(v^{n},H^{n})\cdot J^{n}\ \mathrm{d}x
- \int_{\R^3} \omega^{n} \cdot (J^{n} \cdot \nabla) H^{n}\ \mathrm{d}x\\
&\quad- \int_{\R^3} J^{n} \cdot (\omega^{n} \cdot \nabla) H^{n}\ \mathrm{d}x
\end{align*}
First, it is clear that from H\"older inequality, \eqref{very}, \eqref{able}
\EQS{\label{enst-1}
|\mathbb{Y}_2(t)|&\leq
\|\omega^n\|_{L^2(\R^3)} \|H\|_{L^3(\R^3)} \|\nabla J^l\|_{L^6(\R^3)}+ \|J^n\|_{L^2(\R^3)} \|H\|_{L^3(\R^3)} \|\nabla\omega^l\|_{L^6(\R^3)}\\
&\quad+ \|\omega^n\|_{L^2(\R^3)} \|v\|_{L^3(\R^3)} \|\nabla\omega^l\|_{L^6(\R^3)}
+ \|J^n\|_{L^2(\R^3)} \|v\|_{L^3(\R^3)} \|\nabla v^l\|_{L^6(\R^3)}\\
&\lesssim A^2(\|\omega^n\|_{L^2(\R^3)}+\|J^n\|_{L^2(\R^3)})\lesssim A^2\mathbb{E}^{\frac{1}{2}}(t)\lesssim A^4+\mathbb{E}(t).
}
Similarly, one has
\EQS{\label{enst-2}
|\mathbb{Y}_3(t)|+|\mathbb{Y}_4(t)|\leq A(\|\omega^n\|^2_{L^2(\R^3)}+\|J^n\|^2_{L^2(\R^3)})= A\mathbb{E}(t)
}
and
\EQS{\label{enst-3}
|\mathbb{Y}_5(t)|\leq  A^2(\|\omega^n\|_{L^2(\R^3)}+\|J^n\|_{L^2(\R^3)})\lesssim A^2\mathbb{E}^{\frac{1}{2}}(t)\leq A^4+\mathbb{E}(t)
}
To estimate  nonlinear term $\mathbb{Y}_6(t)$, more work needs to be done. Indeed, by  applying a Littlewood-Paley decomposition to $\mathbb{Y}_6(t)$, we find that
\begin{align*}
\mathbb{Y}_6(t) &= \sum_{N_1,N_2,N_3}\Big\{\int_{\R^3} P_{N_1}\omega^{n} \cdot P_{N_2}\omega^{n} \cdot \nabla P_{N_3}v^{n}\ \mathrm{d}x\\
&\quad+\int_{\R^3} P_{N_1}J^{n} \cdot P_{N_2}J^{n} \cdot \nabla P_{N_3}v^{n}\ \mathrm{d}x+2\int_{\R^3} T(P_{N_1}v^{n},P_{N_2}H^{n})\cdot P_{N_3}J^{n}\ \mathrm{d}x\\
&\quad-\int_{\R^3}P_{N_1}\omega^{n} \cdot P_{N_2}J^{n} \cdot \nabla P_{N_3}H^{n}\ \mathrm{d}x-\int_{\R^3} P_{N_1}J^{n} \cdot P_{N_2}\omega^{n} \cdot \nabla P_{N_3}H^{n}\ \mathrm{d}x\Big\}\\&
\triangleq\sum_{i=1}^{5}I_i,
\end{align*}
 First, we consider $I_1$. Notices that by using Plancherel's theorem
\begin{align*}
I_1&=\int_{\R^3}\mathcal{F}\Big[ (\sum_{N_1}P_{N_1}\omega^{n}) \cdot (\sum_{N_2}P_{N_2}\omega^{n})\Big](\xi) \cdot \mathcal{F}(\sum_{N_3}P_{N_3}\nabla v^{n})(\xi)\ \mathrm{d}\xi\nonumber\\
&=\int_{\R^3}\int_{\R^3}(\sum_{N_1}\Psi_3\widehat{\omega^{n}})(\eta)
(\sum_{N_2}\Psi_4\widehat{\omega^{n}})(\xi-\zeta)(\sum_{N_3}\Psi_5
\widehat{\nabla v^{n}})(\xi) \mathrm{d}\eta\mathrm{d}\xi
\end{align*}
where
$$ \Psi_3=\varphi(\frac{\eta}{N_1})-
\varphi(\frac{2\eta}{N_1}),\quad
\Psi_4=\varphi(\frac{\xi-\zeta}{N_2})-\varphi(\frac{2\xi-2\zeta}{N_2}),\quad
\Psi_5=\varphi(\frac{\xi}{N_3})-\varphi(\frac{2\xi}{N_3}).
$$
thus, the low-low interaction is vanish, i.e., $I_1\triangleq I_{1,1}+I_{1,2}+I_{1,3}$, where
\begin{align*}
I_{1,1} &=\sum_{N_1 \sim N_2}\sum_{ N_3\lesssim N_1 \sim N_2}
\int_{\R^3} P_{N_1} \omega^{n} \cdot (P_{N_2} \omega^{n} \cdot \nabla) P_{N_3} v^{n}\ \mathrm{d}x\\
I_{1,2}&=\sum_{N_1 \sim N_3}\sum_{ N_2\lesssim N_1 \sim N_3}\int_{\R^3} P_{N_1} \omega^{n} \cdot (P_{N_2} \omega^{n} \cdot \nabla) P_{N_3} v^{n}\ \mathrm{d}x\\
I_{1,3}&=\sum_{N_2 \sim N_3}\sum_{ N_1\lesssim N_2 \sim N_3}\int_{\R^3} P_{N_1} \omega^{n} \cdot (P_{N_2} \omega^{n} \cdot \nabla) P_{N_3} v^{n}\ \mathrm{d}x.
\end{align*}
To consider $I_{1,1}$, we first note that  from \eqref{bern}, \eqref{able}, \eqref{piano},
\begin{align*}
\sum_{N_3\lesssim N_1}\| P_{N_3} \nabla v^{n} \|_{L^\infty_x(\R^3)}&\leq \Big(\sum_{N_*\leq N_3\lesssim N_1}+\sum_{N_3\leq N_*}\Big)\| P_{N_3} \nabla v^n \|_{L^\infty_x(\R^3)}\\
&\leq \sum_{N_*\leq N_3\lesssim N_1} A_1^{-1}N_3^2+\sum_{N_3\leq N_*}AN_3^2\\
&\lesssim A_1^{-1}N_1^2+AN_*^2.
\end{align*}
This implies
\begin{align*}
I_{1,1} \lesssim A_1^{-1}\sum_{N_1}N_1^2\| P_{N_1} \omega^{n} \|_{L^2_x(\R^3)}^2+A N_*^2 \sum_{N_1}\| P_{N_1} \omega^{n} \|_{L^2_x(\R^3)}^2.
\end{align*}
Similarly,
\begin{align*}
I_{1,2} \lesssim A_1^{-1}\sum_{N_3}N_3^2\| P_{N_3} \omega^{n} \|_{L^2_x(\R^3)}^2+A N_*^2 \sum_{N_3}\| P_{N_3} \omega^{n} \|_{L^2_x(\R^3)}^2
\end{align*}
and
\begin{align*}
I_{1,3} \lesssim A_1^{-1}\sum_{N_2}N_2^2\| P_{N_2} \omega^{n} \|_{L^2_x(\R^3)}^2+A N_*^2 \sum_{N_2}\| P_{N_2} \omega^{n} \|_{L^2_x(\R^3)}^2.
\end{align*}
Thus, we finally derive
$$
I_1\lec A_1^{-1}\sum_{N}N^2\| P_{N} \omega^{n} \|_{L^2_x(\R^3)}^2+A N_*^2 \sum_{N}\| P_{N} \omega^{n} \|_{L^2_x(\R^3)}^2.
$$
By the similar calculation as above, one has
\begin{align*}
I_2,~I_3 &\lesssim A_1^{-1}\sum_{N}N^2\| P_{N} J^{n} \|_{L^2_x(\R^3)}^2+A N_*^2 \sum_{N}\| P_{N} J^{n} \|_{L^2_x(\R^3)}^2,
\end{align*}
and
\begin{align*}
I_4,~I_5&\lesssim A_1^{-1}\sum_{N}N^2(\| P_{N} \omega^{n} \|_{L^2_x(\R^3)}^2+\| P_{N}J^{n} \|_{L^2_x(\R^3)}^2)\nonumber\\
&\quad+A N_*^2 \sum_{N}(\| P_{N} \omega^{n} \|_{L^2_x(\R^3)}^2+\| P_{N_1}J^{n} \|_{L^2_x(\R^3)}^2).
\end{align*}
Combined with the above estimates, one finally obtains
\begin{align*}
\mathbb{Y}_6(t) &\lesssim A_1^{-1}\sum_{N}N^2(\| P_{N} \omega^{n} \|_{L^2_x(\R^3)}^2+\| P_{N}J^{n} \|_{L^2_x(\R^3)}^2)\\
&\quad+A N_*^2 \sum_{N}(\| P_{N} \omega^{n} \|_{L^2_x(\R^3)}^2+\| P_{N}J^{n} \|_{L^2_x(\R^3)}^2).
\end{align*}
On the other hand, from Plancherel's theorem we also have
\begin{align*}
&\mathbb{Y}_1(t) \sim \sum_{N_1}N_1^2(\| P_{N_1} \omega^{n} \|_{L^2_x(\R^3)}^2+\| P_{N_1}J^{n} \|_{L^2_x(\R^3)}^2),\\
&\mathbb{E}(t) \sim \sum_{N_1}(\| P_{N_1} \omega^{n} \|_{L^2_x(\R^3)}^2+\| P_{N_1}J^{n} \|_{L^2_x(\R^3)}^2),
\end{align*}
and hence
$$ \mathbb{Y}_6(t) \lesssim A_1^{-1} \mathbb{Y}_1(t) + A N_*^2 \mathbb{E}(t).$$
This, along with \eqref{enst-1}-\eqref{enst-3}, yields
\begin{align}\label{dtEt}
 \frac{\mathrm{d}}{\mathrm{d}t}\mathbb{ E}(t) + \mathbb{Y}_1(t) \lesssim A N_*^2 \mathbb{E}(t) + A^4.
 \end{align}
By the aid of Gronwall's inequality, one further derives for any $\frac{1}{2} \leq s \leq \tau \leq 1$
\begin{align*}
 \mathbb{E}(\tau) \lesssim e^{AN_*^{2}(\tau-s)}\mathbb{E}(s)+A^4\int_{s}^{\tau}e^{AN_*^{2}(\tau-\theta)}
 \mathrm{d}\theta.
\end{align*}
Thus, if $0<\tau-s \leq A^{-1} N_*^{-2}$ one has
\EQS{\label{Et2Et1}
 \mathbb{E}(\tau) \lesssim \mathbb{E}(s) + A^4.
 }
 Now we further claim that
\begin{equation}\label{122}
 \int_{-\frac{1}{2}}^1 \int_{\R^3} (|\nabla v^{n}|^2+|\nabla H^{n}|^2)\ \mathrm{d}x \mathrm{d}t \lesssim A^4.
\end{equation}
Notices that the nonlinear pair $(v^{n},H^n)$ fulfills
\begin{equation}\label{Vi}
 \partial_t v^{n} = \Delta v^{n} - \div(v \otimes v-H \otimes H) - \nabla \Pi,
\end{equation}
\begin{equation}\label{Hi}
 \partial_t H^{n} = \Delta H^{n} +(v \cdot \nabla)H-(H \cdot \nabla)v.
\end{equation}
Multiplying  \eqref{Vi}, \eqref{Hi} by $v^{n}$ and $H^{n}$, respectively,  and integrating with respect to $x$ over $\R^3$, we conclude that
\begin{equation}\label{12ddtvnl2Hnl2}
\begin{aligned}
  &\frac{1}{2} \frac{\mathrm{d}}{\mathrm{d}t} \int_{\R^3} (|v^{n}|^2+|H^{n}|^2)\ \mathrm{d}x+\int_{\R^3} (|\nabla v^{n}|^2+|\nabla H^{n}|^2)\ \mathrm{d}x\nonumber\\
   &\quad= \int_{\R^3} (\nabla v^{n}):(v\otimes v)\ \mathrm{d}x + \int_{\R^3} [(H\cdot\nabla) v-(v\cdot\nabla) H]\cdot H^{n}- (\nabla v^{n}):(H\otimes H)\ \mathrm{d}x\nonumber\\
   &\quad\triangleq I_6+I_7.
\end{aligned}
\end{equation}
As before, we denote the nonlinear component by $v^{n} \triangleq v-v^{l}$ and $H^{n} \triangleq H-H^{l}$  which fulfil by  using \eqref{able}, \eqref{very}
$$
\| (v^{n}, H^{n})\|_{L^\infty_t L^3_x([-\frac12,1] \times \R^3)}\lesssim A.
$$
It is clear that by a simple calculation
\begin{align*}\label{I1}
I_6=&\int_{\R^3} (\nabla v^{n}):(v\otimes v)\ \mathrm{d}x\nonumber\\
=&\int_{\R^3} (\nabla v^{n}):(v^{l}\otimes v+v^{n}\otimes v^{l})\ \mathrm{d}x\nonumber\\
\leq&\frac{1}{2} \int_{\R^3} |\nabla v^{n}|^2\ \mathrm{d}x+
C\|v^{l}\|^2_{L^{6}_x(\R^3)}\|v\|^2_{L^{3}_x(\R^3)}
+C\|v^{n}\|^2_{L^{3}_x(\R^3)}\|v^{l}\|^2_{L^{6}_x(\R^3)}
\nonumber\\
\leq&\frac{1}{2} \int_{\R^3} |\nabla v^{n}|^2\ dx+CA^4.
\end{align*}
Here we have used the fact
$$\int_{\R^3} (\nabla v^{n}):(v^{n}\otimes v^{n})\ \mathrm{d}x=
\int_{\R^3} (\partial_i v_j^{n})v_i^{n}v_j^{n}\ \mathrm{d}x=0,
$$
due to $\mathrm{div} v^{n}=0$. Similarly,
\begin{align*}
I_7\leq\frac{1}{2} \int_{\R^3} |\nabla H^{n}|^2\ \mathrm{d}x+CA^4.
\end{align*}
From this, we immediately derive the desired estimate \eqref{122}.
Thus, from \eqref{122}, we also have
$$\int_{\frac{3}{4}-A^{-1} N_*^{-2}}^{\frac{3}{4}} \mathbb{E}(t)\ \mathrm{d}t
\leq\int_{\frac{1}{2}}^1 \mathbb{E}(t)\ \mathrm{d}t \lesssim A^4,
$$
and hence by the pigeonhole principle,
 there is at least one time $t_1\in(\frac{3}{4}-A^{-1} N_*^{-2}, \frac{3}{4})$ with $\mathbb{E}(t_1) \lesssim A^5 N_*^2$.
Together with \eqref{Et2Et1}, we conclude that
\begin{align*}
&\mathbb{E}(t_2) \lesssim \mathbb{E}(t_1)+A^4\lesssim A^5 N_*^2\lesssim N_*^{3},\quad\text{for~all}~t_2\in[t_1, t_1+A^{-1} N_*^{-2}],\\
&\mathbb{E}(t_3) \lesssim \mathbb{E}(t_1+A^{-1} N_*^{-2})+A^4\lesssim A^5 N_*^2\lesssim N_*^{3},\quad\text{for~all}~ t_3\in[t_1+A^{-1} N_*^{-2}, t_1+2A^{-1} N_*^{-2}],\\
&\cdots\\
&\mathbb{E}(t_n) \lesssim \mathbb{E}(1-A^{-1} N_*^{-2})+A^4\lesssim A^5 N_*^2\lesssim N_*^{3},\quad\text{for~all}~ t_n\in[1-A^{-1} N_*^{-2}, 1].
\end{align*}
Through this iterative process, we can get $\mathbb{E}(t)\lesssim N_*^{3}$
for all $\frac{3}{4} \leq t \leq 1$, which, along with \eqref{dtEt}, yields
$$ \int_{\frac{3}{4}}^1 \mathbb{Y}_1(t)\ \mathrm{d}t=\int_{\frac{3}{4}}^1\int_{\R^3} |\nabla \omega^{n}(t,x)|^2+ |\nabla J^{n}(t,x)|^2\ \mathrm{d}x\ \mathrm{d}t \lesssim N_*^{3}.$$
 This further implies
\EQS{\label{Et2Et2}
\int_{\frac{3}{4}}^1\int_{\R^3} |\nabla^2 v(t,x)|^2+ |\nabla^2 H^{n}(t,x)|^2\ \mathrm{d}x\ \mathrm{d}t \lesssim N_*^{3}
}
Here, we have used the fact
\EQs{
\| \nabla v^{n} \|_{L^2(\R^3)}=\| \omega^{n} \|_{L^2(\R^3)},\quad
\| \nabla H^{n} \|_{L^2(\R^3)} = \| J^{n} \|_{L^2(\R^3)},
}
due to $\div v^{n}=\div H^{n}=0$.
 Iterating the estimate \eqref{Et2Et2} (The detailed iterative process is given in Proposition \ref{iii} below), we finally derive
$$ |(v,H)(t,x)|,~|\nabla (v,H)(t,x)|,~|\nabla^2 (v,H)(t,x)| \lesssim N_*^{30}$$
on $[\frac{7}{8},1] \times \R^3$.  This concludes Theorem \ref{main-1}.
\end{proof}
\begin{remark}\label{thm-14-rem}
 In contrast to the proof of Proposition \ref{main-est}, we use the enstrophy-type quantity \eqref{enstro-01} to derive the desired result, instead of the quantity
 \EQS{\label{remar-01}
 \int_{\R^3} |\omega^{n}(t,x)|^2+ |\nabla H^{n}(t,x)|^2\ \mathrm{d}x.
 }
 The main reason  is that
\begin{align*}
\int_{\R^3} (H\cdot\nabla) J^{n} \cdot \omega^{n} \ \mathrm{d}x+ \int_{\R^3}(H\cdot\nabla)\omega^n\cdot J^{n} \ \mathrm{d}x=0.
\end{align*}
 However, if we use the quantity \eqref{remar-01}, we have to deal with the term
 \begin{align*}
\int_{\R^3} (H\cdot\nabla) J^{n} \cdot \omega^{n} \ \mathrm{d}x+ \int_{\R^3}(H\cdot\nabla)v_{x_k}^n\cdot H_{x_k}^{n} \ \mathrm{d}x\neq 0,
\end{align*}
which is estimated as follows
 \begin{align*}
&\Big|\int_{\R^3} (H\cdot\nabla) J^{n} \cdot \omega^{n} \ \mathrm{d}x+ \int_{\R^3}(H\cdot\nabla)v_{x_k}^n\cdot H_{x_k}^{n} \ \mathrm{d}x\Big|\\
&\quad\leq \|H\|_{L^3(\R^3)}\|\nabla\omega^{n}\|_{L^2(\R^3)}\|\nabla H_{x_k}^{n}\|_{L^2(\R^3)}\lesssim AY_1(t).
\end{align*}
Hence, \eqref{dtEt} fails, which disable us to prove Theorem \ref{main-1} with the help of the Gronwall's inequality and bootstrap technique.
\end{remark}

\medskip
Now we are in a position to prove Theorem \ref{main-2}. The key ingredient of the proof is the quantitative estimates \eqref{quanti-estimate} of Theorem \ref{main-1}.

\begin{proof}[Proof of Theorem \ref{main-2}]
Let $c>0$ be a sufficiently small constant, and suppose for contradiction that
$$ \limsup_{t \to 1^-} \frac{\|(v,H)(t)\|_{L^3_x(\R^3)}}{[\log\log\log (\frac{1}{1-t})]^c} < +\infty$$
Thus, for some constant $M>1$, we know that there exists $0<\delta<\exp\{-\exp{\exp(M^{\frac{1}{c}})}\}$ such that
\begin{equation*}
 \|(v,H)(t)\|_{L^3_x(\R^3)}\leq M [\log\log\log(\frac{1}{1-t})]^c,\quad t\in [1-\delta, 1)
\end{equation*}
Applying Theorem \ref{main-1}, for $c$ small enough (for example we can choose that $2cC_0^8<1$). Then, we obtain
\begin{align*}
 \| (v,H)(t) \|_{L^\infty_x(\R^3)}
 &\lesssim \exp\Big\{\exp\Big[\exp\Big(M^{C_0^8}\cdot[\log\log\log( \frac{1}{1-t})]^{cC_0^8}\Big)\Big]\Big\}\nonumber\\
  &\lesssim \exp\Big\{\exp\Big[\exp\Big([\log\log\log( \frac{1}{1-t})]^{2cC_0^8}\Big)\Big]\Big\}\nonumber\\
  &\lesssim \exp\Big\{\exp\Big[\exp\Big(\frac12[\log\log\log( \frac{1}{1-t})]\Big)\Big]\Big\}\nonumber\\
  &\lesssim \exp\Big\{\exp\Big[\exp\Big(\log\log\log( \frac{1}{1-t})^{\frac{1}{10}}\Big)\Big]\Big\}\nonumber\\
  &\lesssim(\frac{1}{1-t})^{\frac{1}{10}},
\end{align*}
where we used the inequality as follows
\begin{align*}\log\log\log( \frac{1}{1-t})^{\frac{1}{10}}&=\log\log\Big(\frac{1}{10}\log( \frac{1}{1-t})\Big)\\
&=\log\Big[\log\frac{1}{10}+\log\log(\frac{1}{1-t})\Big]\\
&>\log\Big(\frac{1}{2}\log\log(\frac{1}{1-t})\Big)\\
&\quad=\log\frac{1}{2}+\log\log\log(\frac{1}{1-t})\\
&>\frac{1}{2}\log\log\log(\frac{1}{1-t})
\end{align*}
In particular, $v$ and $H$ are bounded in $L^2_tL^\infty_x((1-\delta, 1)\times\R^3)$, contradicting the classical blow-up criterion \cite{MR1915942,MR2026210}.
This concludes the proof of Theorem \ref{main-2}.
\end{proof}

\section{An upper bound estimate of $N_0$}\label{6.upbound}
In this section, we will prove Proposition \ref{main-est}, which provides an upper bound estimate for $N_0$, in four steps. Firstly,  in section \ref{3.basic}, we provided an estimate of the bounded total velocity, which is necessary for proving the construction of frequency bubbles in section \ref{5.bubble}. Secondly, in section \ref{4.epoch}, we proved that the epochs of estimation and annuli of estimation for $v$ and $H$, both of which are crucial for proving the main results in this section. Finally, in section \ref{pf-N0}, we  provided an upper bound estimate for $N_0$, which constitutes the proof of Proposition \ref{main-est}.

\subsection{Bounded total speed estimate}\label{3.basic}
In this section, we consider the bounded total speed estimate for MHD system, which is crucial for obtaining the epochs regularity and annuli regularity estimates of $(v,H)$, as well as for the construction of frequency bubbles.
Let us emphasize  that our strategy is due to  Foias, Guillop\'{e} and Temam in \cite{MR607552}, which is much more direct and simpler compared with Tao's argument in \cite{MR4337421}
\begin{prop}\label{ii}
Let $v,H,\pi,A$ obey the hypotheses of Proposition \ref{main-est}.
Then for any interval $I\subset[0,1]$, we have
\begin{equation}\label{bts}
 \| (v,H) \|_{L^1_t L^\infty_x( I \times \R^3 )} \lesssim A^{4}|I|^{\frac12}.
\end{equation}
\end{prop}

\begin{proof}
We will use the  energy method to derive the desired result.
First, we assume $I=[0,1]$. 
 Together with the pigeonhole principle,  and the estimate \eqref{122}, ensures that  there exists $t_0\in[-\frac12,-\frac14]$, such that
\begin{equation}\label{t0122}
 \int_{\R^3} |\nabla v^{n}(t_0)|^2+|\nabla H^{n}(t_0)|^2\ \mathrm{d}x  \lesssim A^4.
\end{equation}
To conclude the proof, we need to show
\EQS{\label{vnHnlapu-2}
\int_{t_0}^1\frac{
\|\Delta (v^{n},H^n)\|^2_{L^2_x(\R^3)}}{\left(A^2+\|\nabla v^n\|^2_{L^2_x(\R^3)}+\|\nabla H^n\|^2_{L^2_x(\R^3)}\right)^2}\ \mathrm{d}t\lec A^4.
}
We suppose this estimate is valid at the moment. Then from this and multiply inequality, one has
\begin{align*}
&\int_{0}^1\|v^n\|_{L^\infty_x(\R^3)}+\|H^n\|_{L^\infty_x(\R^3)}\ \mathrm{d}t\\
&\quad\lesssim \int_{0}^1\|\nabla v^n\|^{\frac12}_{L^2_x(\R^3)}\|\Delta v^n\|^{\frac12}_{L^2_x(\R^3)}+\|\nabla H^n\|^{\frac12}_{L^2_x(\R^3)}\|\Delta H^n\|^{\frac12}_{L^2_x(\R^3)}\ \mathrm{d}t\\
&\quad\lesssim \int_{0}^1\Big(A^2+\|\nabla v^n\|^2_{L^2_x(\R^3)}+\|\nabla H^n\|^2_{L^2_x(\R^3)}\Big)^{\frac14}
\Big(\|\Delta v^n\|^{\frac12}_{L^2_x(\R^3)}+\|\Delta H^n\|^{\frac12}_{L^2_x(\R^3)}\Big)\ \mathrm{d}t\\
&\quad\lesssim \int_{0}^1 A^2+\|\nabla v^n\|^2_{L^2_x(\R^3)}+\|\nabla H^n\|^2_{L^2_x(\R^3)}\ \mathrm{d}t+\int_{0}^1
\frac{\|\Delta v^{n}\|^2_{L^2_x(\R^3)}+\|\Delta H^{n}\|^2_{L^2_x(\R^3)}}{\left(A^2+\|\nabla v^n\|^2_{L^2_x(\R^3)}+\|\nabla H^n\|^2_{L^2_x(\R^3)}\right)^2}\ \mathrm{d}t\\
&\quad \lesssim A^4.
\end{align*}
Together with \eqref{very}, we finally get
\begin{align}\label{VHL1L}
\| v \|_{L^1_t L^\infty_x([0,1] \times \R^3)}\lesssim A^4,\quad\| H \|_{L^1_t L^\infty_x([0,1] \times \R^3)}
\lesssim A^4.
\end{align}
Now we remain to show \eqref{vnHnlapu-2}. Indeed, multiplying \eqref{Vi}, \eqref{Hi}  by $\Delta v^n$, $\Delta H^n$, respectively, one concludes that
\begin{align}\label{vnHnlapu}
&\frac12\frac{\mathrm{d}}{\mathrm{d}t}\int_{\R^3} |\nabla (v^{n},H^{n})|^2dx+\int_{\R^3}|\Delta (v^{n},H^{n})|^2\ \mathrm{d}x\nonumber\\
&\quad=\int_{\R^3}(v\cdot\nabla)v\cdot\Delta v^n \ \mathrm{d}x+\int_{\R^3}(v\cdot\nabla)H\cdot\Delta H^n \ \mathrm{d}x
\nonumber\\
&\quad\quad-\int_{\R^3}(H\cdot\nabla)H\cdot\Delta v^n \ \mathrm{d}x-\int_{\R^3}(H\cdot\nabla)v\cdot\Delta H^n \ \mathrm{d}x\nonumber\\
&\quad=\int_{\R^3}(v^n\cdot\nabla)v^n\cdot\Delta v^n \ \mathrm{d}x+\int_{\R^3}(v^n\cdot\nabla)v^l\cdot\Delta v^n \ \mathrm{d}x+\int_{\R^3}(v^l\cdot\nabla)v^n\cdot\Delta v^n \ \mathrm{d}x\nonumber\\
&\quad\quad+\int_{\R^3}(v^l\cdot\nabla)v^l\cdot\Delta v^n \ \mathrm{d}x+\int_{\R^3}(v^n\cdot\nabla)H^n\cdot\Delta H^n \ \mathrm{d}x+\int_{\R^3}(v^n\cdot\nabla)H^l\cdot\Delta H^n \ \mathrm{d}x
\nonumber\\
&\quad\quad+\int_{\R^3}(v^l\cdot\nabla)H^n\cdot\Delta H^n \ \mathrm{d}x
+\int_{\R^3}(v^l\cdot\nabla)H^l\cdot\Delta H^n \ \mathrm{d}x-\int_{\R^3}(H^n\cdot\nabla)H^n\cdot\Delta v^n \ \mathrm{d}x
\nonumber\\
&\quad\quad-\int_{\R^3}(H^n\cdot\nabla)H^l\cdot\Delta v^n \ \mathrm{d}x-\int_{\R^3}(H^l\cdot\nabla)H^n\cdot\Delta v^n \ \mathrm{d}x-\int_{\R^3}(H^l\cdot\nabla)v^l\cdot\Delta v^n \ \mathrm{d}x\nonumber\\
&\quad\quad-\int_{\R^3}(H^n\cdot\nabla)v^n\cdot\Delta H^n \ \mathrm{d}x-\int_{\R^3}(H^n\cdot\nabla)v^l\cdot\Delta H^n \ \mathrm{d}x-\int_{\R^3}(H^l\cdot\nabla)v^n\cdot\Delta H^n \ \mathrm{d}x\nonumber\\
&\quad\quad-\int_{\R^3}(H^l\cdot\nabla)v^l\cdot\Delta H^n \ \mathrm{d}x\triangleq\sum_{i=1}^{16}K_i,
\end{align}
where according to H\"older's inequality, Young's inequality, Gagliardo-Nirenberg inequality and \eqref{very}, we can give the estimates of $K_i~(i=1,2,\ldots,16)$ as follows
\EQs{
|K_1|&\leq \|v^n\|_{L^6(\R^3)}\|\nabla v^n\|_{L^3(\R^3)}\|\Delta v^n\|_{L^2(\R^3)}\leq \|v^n\|^{\frac32}_{L^6(\R^3)}\|\Delta v^n\|^{\frac32}_{L^2(\R^3)}\\
&\leq\frac{1}{18}\|\Delta v^n\|^{2}_{L^2(\R^3)}+C\|\nabla v^n\|^{6}_{L^2(\R^3)},\\
|K_2|&\leq \|\nabla v^l\|_{L^3(\R^3)}\|v^n\|_{L^6(\R^3)}\|\Delta v^n\|_{L^2(\R^3)}\leq\frac{1}{18}\|\Delta v^n\|^{2}_{L^2(\R^3)}+C\|\nabla v^l\|^{2}_{L^3(\R^3)}\|\nabla v^n\|^{2}_{L^2(\R^3)}\\
&\leq\frac{1}{18}\|\Delta v^n\|^{2}_{L^2(\R^3)}+CA^2\|\nabla v^n\|^{2}_{L^2(\R^3)}.
}
By a similarly calculation, one has
\EQs{
&|K_3|+|K_6|+|K_7|+|K_{10}|+|K_{11}|+|K_{14}|+|K_{15}|\\&
\leq\frac{3}{18}\|\Delta v^n\|^{2}_{L^2(\R^3)}+\frac{2}{9}\|\Delta H^n\|^{2}_{L^2(\R^3)}+CA^2(\|\nabla v^n\|^{2}_{L^2(\R^3)}+\|\nabla H^n\|^{2}_{L^2(\R^3)}),
}
\EQs{
|K_4|+|K_8|+|K_{12}|+|K_{16}|\leq \frac{2}{9}\left(\|\Delta H^n\|^2_{L^2(\R^3)}+\|\Delta v^n\|^2_{L^2(\R^3)}\right)+CA^4,
}
and
\EQs{
|K_5|+|K_9|+|K_{13}|\leq \frac{1}{6} \left(\|\Delta H^n\|^2_{L^2(\R^3)}+\|\Delta v^n\|^2_{L^2(\R^3)}\right)+C(\|\nabla v^n\|^{6}_{L^2(\R^3)}+\|\nabla H^n\|^{6}_{L^2(\R^3)}).
}
Substituting the estimate of $K_i~(i=1,2,\ldots,16)$ into \eqref{vnHnlapu}, we get
\begin{align*}
&\frac{\mathrm{d}}{\mathrm{d}t}
\Big(\|\nabla v^n\|^2_{L^2_x(\R^3)}+\|\nabla H^n\|^2_{L^2_x(\R^3)}\Big)+
\|\Delta v^{n}\|^2_{L^2_x(\R^3)}+\|\Delta H^{n}\|^2_{L^2_x(\R^3)}\nonumber\\
&\quad \lesssim \|\nabla v^n\|^6_{L^2_x(\R^3)}+\|\nabla H^n\|^6_{L^2_x(\R^3)}+A^2(\|\nabla v^n\|^2_{L^2_x(\R^3)}+\|\nabla H^n\|^2_{L^2_x(\R^3)})+A^4\nonumber\\
&\quad \lesssim \Big(\|\nabla v^n\|^2_{L^2_x(\R^3)}+\|\nabla H^n\|^2_{L^2_x(\R^3)}+A^2\Big)^3,
\end{align*}
this implies that
\begin{align}\label{vnHnlapu-1}
&\frac{\mathrm{d}}{\mathrm{d}t}
\Big[-(A^2+\|\nabla v^n\|^2_{L^2_x(\R^3)}+\|\nabla H^n\|^2_{L^2_x(\R^3)})^{-1}\Big]+\frac{
\|\Delta v^{n}\|^2_{L^2_x(\R^3)}+\|\Delta H^{n}\|^2_{L^2_x(\R^3)}}{\left(A^2+\|\nabla v^n\|^2_{L^2_x(\R^3)}+\|\nabla H^n\|^2_{L^2_x(\R^3)}\right)^2}\nonumber\\
&\quad \lesssim A^2+
\|\nabla v^n\|^2_{L^2_x(\R^3)}+\|\nabla H^n\|^2_{L^2_x(\R^3)}.
\end{align}
Integrating with respect to $t$ over $[t_0, 1]$ for \eqref{vnHnlapu-1} and using \eqref{122}, \eqref{t0122}, we can obtain
\begin{align*}
&\int_{t_0}^1\frac{
\|\Delta (v^{n},H^n)\|^2_{L^2_x(\R^3)}}{\left(A^2+\|\nabla v^n\|^2_{L^2_x(\R^3)}+\|\nabla H^n\|^2_{L^2_x(\R^3)}\right)^2}\ \mathrm{d}t\nonumber\\
&\quad \lesssim \int_{t_0}^1 A^2+\|\nabla v^n\|^2_{L^2_x(\R^3)}+\|\nabla H^n\|^2_{L^2_x(\R^3)}\ \mathrm{d}t+\frac{\|\nabla v^n(t_0)\|^2_{L^2_x(\R^3)}+\|\nabla H^n(t_0)\|^2_{L^2_x(\R^3)}}{A^4}\lesssim A^4.
\end{align*}
which is the desired estimate \eqref{vnHnlapu-2}.

Secondly, if $I=[t_1,t_2]\subset\subset [0,1]$, we can invoke the scaling and translation argument to derive the desired result. In fact, we take
\EQS{\label{scaling-01}
v^\lambda(t,x) \triangleq \lambda v(\lambda^2 t+t_1, \lambda x),\quad H^\lambda(t,x) \triangleq \lambda H(\lambda^2 t+t_1, \lambda x),
}
with $\lambda=\sqrt{t_2-t_1}$, then $(v^\lambda,H^\lambda)$ is also a classical solution of \eqref{mhdeq} in $\left[-\frac{1+t_1}{t_2-t_1},\frac{1-t_1}{t_2-t_1}\right]\times \R^3$ with
$[-1,1]\subset \left[-\frac{1+t_1}{t_2-t_1},\frac{1-t_1}{t_2-t_1}\right]$. Notices that $\|(v^\lambda,H^\lambda)\|_{L^\infty_t L^3_x}=\|(v,H)\|_{L^\infty_t L^3_x}\leq A$, then by \eqref{VHL1L}, we have
$$
\| (v^\lambda,H^\lambda) \|_{L^1_t L^\infty_x([0,1] \times \R^3)}\lesssim A^4,
$$
from which one has
\begin{align}\label{VHL1I.est}
\|(v,H)
\|_{L^1_t L^\infty_x(I \times \R^3)}=\sqrt{t_2-t_1}\|(v^\lambda,H^\lambda) \|_{L^1_t L^\infty_x([0,1] \times \R^3)}\lesssim A^4|I|^{\frac12},
\end{align}
for any interval $I\subset[0,1]$.
\end{proof}

\subsection{Epochs of estimation and annuli of estimation for $v$ and $H$}\label{4.epoch}
The role of this section in the paper is crucial, particularly in the estimation of the upper bound for $N_0$ using quantitative versions of the Carleman inequality. Here, the standard argument, such as the bootstrap
technique, is employed.
\begin{prop}\label{iii}
Let $v,\pi,H,A$ obey the hypotheses of Proposition \ref{main-est}. For any interval $I\subset[0, 1]$, there is a subinterval $I' \subset I$ with $|I'| \gtrsim A^{-8} |I|$ such that
$$ \| \nabla^j (v,H)\|_{L^\infty_t L^\infty_x(I' \times \R^3)} \lesssim A^{35} |I|^{-\frac{j+1}{2}}\quad\text{for}\quad j=0,1,2.$$
\end{prop}

\begin{proof}
Firstly, we assume $I=[0,1]$.
Taking the gradient of \eqref{Vi}, \eqref{Hi},
then taking the inner product with $\nabla v^{n}$ and $\nabla H^{n}$ respectively,
 we conclude that
\begin{align*}
&\frac{\mathrm{d}}{\mathrm{d}t}\Big[\frac{1}{2}\int_{\R^3} (|\nabla v^{n}|^2+\int_{\R^3} |\nabla H^{n}|^2)\ \mathrm{d}x\Big]
+\int_{\R^3} (|\nabla^2 v^{n}|^2+|\nabla^2 H^{n}|^2)\ \mathrm{d}x\\
&~~=\int_{\R^3} \mathrm{div}(v \otimes v-H\otimes H)\cdot \Delta v^n\ \mathrm{d}x+\int_{\R^3} \Big[(v \cdot\nabla) H-(H\cdot\nabla) v\Big]\cdot \Delta H^n\ \mathrm{d}x\triangleq I_8+I_9.
\end{align*}
The estimates for $I_8$ as follows
\begin{align}\label{I3}
I_8=&\int_{\R^3} \mathrm{div}(v \otimes v-H\otimes H)\cdot \Delta v^n\ \mathrm{d}x\nonumber\\
\leq&\frac{1}{2} \|\nabla^2 v^{n}\|^2_{L^{2}_x(\R^3)}+
C\|\mathrm{div}(v \otimes v)\|^2_{L^{2}_x(\R^3)}+C\|\mathrm{div}(H \otimes H)\|^2_{L^{2}_x(\R^3)}
\nonumber\\
\leq&\frac{1}{2} \|\nabla^2 v^{n}\|^2_{L^{2}_x(\R^3)}+
C\|v\|^2_{L^{6}_x(\R^3)}\|\nabla v\|^2_{L^{3}_x(\R^3)}
+C\|H\|^2_{L^{6}_x(\R^3)}\|\nabla H\|^2_{L^{3}_x(\R^3)}
\nonumber\\
\leq&\frac{1}{2} \|\nabla^2 v^{n}\|^2_{L^{2}_x(\R^3)}+C\Big(A^2+\|v^n\|^2_{L^{6}_x(\R^3)}\Big)
\Big(A^2+\|\nabla v^n\|^2_{L^{3}_x(\R^3)}\Big)\nonumber\\
&\quad+C\Big(A^2+\|H^n\|^2_{L^{6}_x(\R^3)}\Big)
\Big(A^2+\|\nabla H^n\|^2_{L^{3}_x(\R^3)}\Big)\nonumber\\
\leq&\frac{1}{2} \|\nabla^2 v^{n}\|^2_{L^{2}_x(\R^3)}+C\Big(A^2+E_1(t)\Big)
\Big(A^2+E_1^{\frac{1}{2}}(t)\|\nabla^2 v^n\|_{L^{2}_x(\R^3)}\Big)\nonumber\\
&+C\Big(A^2+E_2(t)\Big)
\Big(A^2+E_2^{\frac{1}{2}}(t)\|\nabla^2 H^n\|_{L^{2}_x(\R^3)}\Big)\nonumber\\
\leq&\frac{5}{8} \|\nabla^2 v^{n}\|^2_{L^{2}_x(\R^3)}+\frac{1}{8}\|\nabla^2 H^{n}\|^2_{L^{2}_x(\R^3)}+C\Big[A^4+A^4 \mathcal{E}(t)+ \mathcal{E}(t)^3
\Big],
\end{align}
where
$$ \mathcal{E}(t)\triangleq E_1(t)+E_2(t)= \frac{1}{2} \int_{\R^3} |\nabla v^{n}(t,x)|^2+ |\nabla H^{n}(t,x)|^2\ \mathrm{d}x, \quad t \in [0,1].$$
Similarly, we also have
\begin{align}\label{I4}
I_9
\leq&\frac{5}{8} \|\nabla^2 H^{n}\|^2_{L^{2}_x(\R^3)}+\frac{1}{8} \|\nabla^2 v^{n}\|^2_{L^{2}_x(\R^3)}+C\Big[A^4+A^4\mathcal{E}(t)+\mathcal{E}(t)^3
\Big]
\end{align}
From \eqref{I3} and \eqref{I4}, for any $t \in [0,1]$ we have
\begin{align}\label{eat}
\frac{\mathrm{d}}{\mathrm{d}t}\mathcal{E}(t)+\Big(\|\nabla^2 v^{n}\|^2_{L^{2}_x(\R^3)}+\|\nabla^2 H^{n}\|^2_{L^{2}_x(\R^3)}\Big)
\leq C\Big[A^4+A^4\mathcal{E}(t)+\mathcal{E}(t)^3\Big].
\end{align}
On the other hand, by the pigeonhole principle and \eqref{122}, we can find a time $t_1 \in [0,\frac{1}{2}]$ such that $\mathcal{E}(t_1)  \lesssim A^4.$  Now
we define
$$t^\ast=\sup\{~t~|\sup_{t_1\leq s\leq t}|\mathcal{E}(s)|\leq CA^4\},$$
then  from \eqref{eat}, we know that for any $t\in [t_1,t^\ast]$
\begin{align*}
\mathcal{E}(t) &\leq \mathcal{E}(t_1)+C\int_{t_1}^{t^\ast}A^4+A^4\mathcal{E}(t)+{\mathcal{E}(t)}^2\ \mathrm{d}t\\
&\leq A^4+CA^{12}(t^\ast-t_1)\leq CA^{4},
\end{align*}
from which we can show that $t^\ast=t_1+cA^{-8}$ and $c>0$ is a small absolute constant.
Thus,
\begin{align}\label{EtCA}
\mathcal{E}(t)\leq CA^{4}\quad\text{for}\quad t \in [t_1, t_1 + c A^{-8}].
\end{align}
For simplicity, we set $\tau(s) \triangleq t_1 + scA^{-8}$, inserting \eqref{EtCA} back into \eqref{eat} one has
$$ \frac{\mathrm{d}}{\mathrm{d}t}\mathcal{E}(t)+ \Big(\|\nabla^2 v^{n}\|^2_{L^{2}_x(\R^3)}+\|\nabla^2 H^{n}\|^2_{L^{2}_x(\R^3)}\Big)\lesssim A^{12}\quad\text{for}\quad t \in [\tau(0), \tau(1)],$$
and hence by the fundamental theorem of calculus
\begin{equation}\label{nog}
 \int_{\tau(0)}^{\tau(1)} \int_{\R^3} (|\nabla^2 v^{n}|^2+|\nabla^2 H^{n}|^2)\ \mathrm{d}x \mathrm{d}t \lesssim A^4.
\end{equation}
 This, along with Gagliardo-Nirenberg inequality
\begin{align*}
 \| v^{n} \|_{L^\infty_x(\R^3)} \lesssim \| \nabla v^{n} \|_{L^2_x(\R^3)}^{\frac{1}{2}} \| \nabla^2 v^{n} \|_{L^2_x(\R^3)}^{\frac{1}{2}},\quad\| H^{n} \|_{L^\infty_x(\R^3)} \lesssim \| \nabla H^{n} \|_{L^2_x(\R^3)}^{\frac{1}{2}} \| \nabla^2 H^{n} \|_{L^2_x(\R^3)}^{\frac{1}{2}},
 \end{align*}
 the H\"older's inequality, \eqref{EtCA}, yields
$$ \| (v^{n},H^{n}) \|_{L^4_t L^\infty_x([\tau(0), \tau(1)] \times \R^3)} \lesssim A^2.$$
Hence by \eqref{very}, we have
\begin{equation}\label{fract}
 \|(v,H)\|_{L^4_t L^\infty_x([\tau(0), \tau(1)] \times \R^3)} \lesssim A^2.
\end{equation}
Furthermore, from Sobolev embedding and \eqref{nog} one has
$$ \| (\nabla v^{n},\nabla H^{n}) \|_{L^2_t L^6_x([\tau(0), \tau(1)] \times \R^3)}\lesssim
A^2,
$$
which, along with \eqref{very} yields that
\begin{equation}\label{nob}
 \| (\nabla v,\nabla H)\|_{L^2_t L^6_x([\tau(0), \tau(1)] \times \R^3)} \lesssim A^2.
\end{equation}
These are  subcritical regularity estimates which  can be further improved through iterative methods.\\
{\bf Step 1}. {\em "Bootstrap" estimate \eqref{fract}:  from $L^4_t L^\infty_x$ to $L^\infty_t L^\infty_x$.}
In fact, from \eqref{able} that
\begin{align*}
\|(v,H)(t) \|_{L^\infty_x(\R^3)} &\lesssim  \| e^{(t-\tau(0)) \Delta}(v,H)(\tau(0))\|_{L^\infty_x(\R^3)}\nonumber\\
&\quad+\int_{\tau(0)}^t\|e^{(t-t')\Delta} \{ \mathbb{P}\mathrm{div}(H \otimes H-v \otimes v),\nabla\times(v \times H)\}(t')\|_{L^\infty_x(\R^3)}\ \mathrm{d}t'\nonumber\\
&\lesssim A(t-\tau(0))^{-\frac{1}{2}}
+ \int_{\tau(0)}^t (t-t')^{-\frac{1}{2}} (\| v(t) \|_{L^\infty_x(\R^3)}^2+\| H(t) \|_{L^\infty_x(\R^3)}^2)\ \mathrm{d}t'.
\end{align*}
This, together with \eqref{fract} and Hardy-Littlewood-Sobolev inequality inequality, yields
\begin{align}\label{vHL8-1}
 \| (v,H) \|_{L^8_t L^\infty_x([\tau(0.1), \tau(1)] \times \R^3)} \lesssim A^{4}.
\end{align}
Repeating the above argument, we now also see for $t \in [\tau(0.2), \tau(1)]$ that
\begin{align*}
\|(v,H)(t) \|_{L^\infty_x(\R^3)}\lesssim A(t-\tau(0.1))^{-\frac{1}{2}}
+ \int_{\tau(0.1)}^t (t-t')^{-\frac{1}{2}} (\| v(t) \|_{L^\infty_x(\R^3)}^2+\| H(t) \|_{L^\infty_x(\R^3)}^2)\ \mathrm{d}t',
\end{align*}
which, along with \eqref{vHL8-1} and  the H\"older's inequality, leads to
\begin{equation}\label{leo}
 \| (v,H) \|_{L^\infty_t L^\infty_x([\tau(0.2), \tau(1)] \times \R^3)} \lesssim A^{8}.
\end{equation}
{\bf Step 2}.{\em "Bootstrap" estimate \eqref{nob}:  from $L^2_t L^6_x$ to $L^\infty_t L^\infty_x$.} Notices that
 for $t \in [\tau(0.3), \tau(1)]$  one has
\begin{align}\label{ler-1}
 (\nabla v,\nabla H)(t) &= e^{(t-\tau(0.2)) \Delta}\nabla (v,H)(\tau(0.2))\nonumber\\
 &\quad+ \int_{\tau(0.2)}^t e^{(t-t')\Delta}\nabla \{\mathbb{P} \mathrm{div}(H \otimes H-v \otimes v),\nabla\times(v \times H)\}(t')\ \mathrm{d}t'.
\end{align}
  From \eqref{able} and \eqref{bern-3}, we have
\begin{align*}
\|(\nabla v,\nabla H)(t) \|_{L^\infty_x(\R^3)} &\lesssim A(t-\tau(0.2))^{-1}\nonumber\\
&\quad+ \int_{\tau(0.2)}^t (t-t')^{-\frac{3}{4}} \Big(\| \mathrm{div}(v \otimes v) \|_{L^6_x(\R^3)}+\| \mathrm{div}(H \otimes H) \|_{L^6_x(\R^3)}\Big)\ \mathrm{d}t'\nonumber\\
&\quad
+ \int_{\tau(0.2)}^t (t-t')^{-\frac{3}{4}} \Big(\|v \cdot \nabla H \|_{L^6_x(\R^3)}+\| H \cdot \nabla v\|_{L^6_x(\R^3)}\Big)\ \mathrm{d}t',
\end{align*}
where from \eqref{nob}, \eqref{leo} and H\"older's inequality one has
\begin{align*}
&\Big\|\{\div (v \otimes v),\div (H\otimes H), v\cdot\nabla H, H\cdot\nabla v\}\Big\|_{L^2_t L^6_x([\tau(0.2), \tau(1)] \times \R^3)}
\lesssim A^{10},
\end{align*}
Thus, combining with \eqref{ler-1} and the Hardy-Littlewood-Sobolev inequality, it is concluded that
\begin{equation}\label{v4wuH4wu}
\| (\nabla v, \nabla H )\|_{L^4_t L^\infty_x([\tau(0.3), \tau(1)] \times \R^3)} \lesssim A^{10}.
\end{equation}
Therefore, from \eqref{leo}, \eqref{v4wuH4wu} and H\"older's inequality, we can get
\begin{align}\label{vvHH-1}
\Big\|\{\div (v \otimes v),\div (H\otimes H), v\cdot\nabla H, H\cdot\nabla v\}\Big\|_{L^4_t L^\infty_x([\tau(0.3), \tau(1)] \times \R^3)}
\lesssim A^{18}.
\end{align}
Furthermore, for any $t \in [\tau(0.4), \tau(1)]$ we can obtain
\begin{align*}
\|(\nabla v,\nabla H)(t) \|_{L^\infty_x(\R^3)}&\lesssim A(t-\tau(0.3))^{-1}\\
&\quad+\int_{\tau(0.3)}^t (t-t')^{-\frac{1}{2}} \Big(\| \mathrm{div}(v \otimes v) \|_{L^\infty_x(\R^3)}+\| \mathrm{div}(H \otimes H) \|_{L^\infty_x(\R^3)}\Big)\ \mathrm{d}t'\\
&\quad+ \int_{\tau(0.3)}^t (t-t')^{-\frac{1}{2}} \Big(\| v \cdot\nabla H) \|_{L^\infty_x(\R^3)}+\|H \cdot\nabla v\|_{L^\infty_x(\R^3)}\Big)\ \mathrm{d}t',
\end{align*}
 and hence by \eqref{vvHH-1} we have
\begin{align}\label{wJwu}
\| (\nabla v ,\nabla H)\|_{L^\infty_t L^\infty_x([\tau(0.4), \tau(1)] \times \R^3)} \lesssim A^{18}.
\end{align}

In order to estimates  $ \nabla^2 v $ and $\nabla^2 H$, we
introduce a cut-off function $\Xi(t)\in C_c^{\infty}([0,\tau(1)])$ such that $0\leq\Xi\leq1$ and
\begin{align*}
\Xi(t)=\begin{cases}
1, & t \in[\tau(0.5),\tau(1)],\\
\mathrm{smooth}, & t\in(\tau(0.4),\tau(0.5)),\\
0, & t\in [0,\tau(0.4)].
\end{cases}
\end{align*}
Thus we set $S(t,x)\triangleq\Xi(t)\omega(t,x),~M(t,x)\triangleq\Xi(t)J(t,x)$ which, respectively, fulfill in $(\tau(0.4),\tau(1))\times\R^3$ 
\EQs{
\left\{
\begin{array}{ll}
\partial_{t}S-\Delta S=\omega\partial_t\Xi
+(S\cdot\nabla)v+(H\cdot\nabla)M-(v\cdot\nabla)S
+(M\cdot\nabla)H,\\
S(\tau(0.4),x)=0,
\end{array}
\right.
}
and
\quad\EQs{
\left\{
\begin{array}{ll}
\partial_{t}M-\Delta M=J\partial_t\Xi
+(H\cdot\nabla)S+(M\cdot\nabla)v-(v\cdot\nabla)M
+(S\cdot\nabla)H+2\Xi R(v,H),\\
M(\tau(0.4),x)=0.
\end{array}
\right.
}
By using the maximal regularity Theorem of parabolic equations (see the page 117 of \cite{MR3616490}), we can obtain
\begin{align}\label{2w26}
&\|\partial_tS \|_{L^2_t L^6_x([\tau(0.4), \tau(1)] \times \R^3)}+\|\nabla^2S\|_{L^2_t L^6_x([\tau(0.4), \tau(1)] \times \R^3)}\nonumber\\
&\lesssim
A^{8}\Big(\|\nabla M \|_{L^2_t L^6_x([\tau(0.4), \tau(1)] \times \R^3)}+\|\nabla S \|_{L^2_t L^6_x([\tau(0.4), \tau(1)] \times \R^3)}\Big)+A^{20}\nonumber\\
&\lesssim
A^{8}\Big[\Big(\int_{\tau(0.4)}^{\tau(1)}\|M \|^{\frac43}_{L^\infty_x(\R^3)}\|\nabla^2 M \|^{\frac23}_{L^6_x(\R^3)}\ \mathrm{d}t\Big)^{\frac12}+\Big(\int_{\tau(0.4)}^{\tau(1)}\|S \|^{\frac43}_{L^\infty_x(\R^3)}\|\nabla^2 S \|^{\frac23}_{L^6_x(\R^3)}\ \mathrm{d}t\Big)^{\frac12}\Big]+A^{20}\nonumber\\
&\lesssim
A^{8}(\tau(1)-\tau(0.4))^{\frac13}\Big[\|J\|^{\frac23}_{L^\infty_t L^\infty_x([\tau(0.4), \tau(1)] \times \R^3)}\|\nabla^2 M \|^{\frac{1}{3}}_{L^2_t L^6_x([\tau(0.4), \tau(1)] \times \R^3)}\nonumber\\
&\quad+
\|\omega\|^{\frac23}_{L^\infty_t L^\infty_x([\tau(0.4), \tau(1)] \times \R^3)}\|\nabla^2 S \|^{\frac{1}{3}}_{L^2_t L^6_x([\tau(0.4), \tau(1)] \times \R^3)}\Big]\nonumber\\
&\lesssim A^{18}\Big(\|\nabla^2 M \|^{\frac{1}{3}}_{L^2_t L^6_x([\tau(0.4), \tau(1)] \times \R^3)}+\|\nabla^2 S \|^{\frac{1}{3}}_{L^2_t L^6_x([\tau(0.4), \tau(1)] \times \R^3)}\Big)+A^{20}\nonumber\\
&\leq \frac{1}{3}\Big(\|\nabla^2 M \|_{L^2_t L^6_x([\tau(0.4), \tau(1)] \times \R^3)}+\|\nabla^2 S \|_{L^2_t L^6_x([\tau(0.4), \tau(1)] \times \R^3)}\Big)+CA^{27}.
\end{align}
Similarly, we have
\begin{align*}
&\|\partial_t M\|_{L^2_t L^6_x([\tau(0.4), \tau(1)] \times \R^3)}+\|\nabla^2 M\|_{L^2_t L^6_x([\tau(0.4), \tau(1)] \times \R^3)}\nonumber\\
&\quad\leq \frac{1}{3}\Big(\|\nabla^2 M \|_{L^2_t L^6_x([\tau(0.4), \tau(1)] \times \R^3)}+\|\nabla^2 S \|_{L^2_t L^6_x([\tau(0.4), \tau(1)] \times \R^3)}\Big)+CA^{27},
\end{align*}
which combining with \eqref{2w26}, we have
\begin{align*}
\|\nabla^2 \omega\|_{L^2_t L^6_x([\tau(0.5), \tau(1)] \times \R^3)}+\|\nabla^2 J\|_{L^2_t L^6_x([\tau(0.5), \tau(1)] \times \R^3)}\lesssim A^{27},
\end{align*}
which, along with the Biot-Savart law and the Calder\'{o}n-Zygmund estimates (see the page 380 of \cite{MR3616490}), yields
\begin{align}\label{2w262J26}
\|(\nabla^2 u, \nabla^2 H)\|_{L^2_t L^6_x([\tau(0.5), \tau(1)] \times \R^3)}&=\|(\nabla^2 (-\Delta)^{-1}\nabla\times(\omega, J)\|_{L^2_t L^6_x([\tau(0.5), \tau(1)] \times \R^3)}\nonumber\\
&\lesssim \Big(\int_{\tau(0.5)}^{\tau(1)}\|\nabla(\omega, J)\|^2_{L^6_x(\R^3)}\ \mathrm{d}t\Big)^{\frac12}\nonumber\\
&\lesssim \Big(\int_{\tau(0.5)}^{\tau(1)}\|(\omega, J)\|^{\frac43}_{L^\infty_x(\R^3)}\|\nabla^2(\omega, J)\|^{\frac23}_{L^6_x(\R^3)}\ \mathrm{d}t\Big)^{\frac12}\nonumber\\
&\lesssim A^{19},
\end{align}
Now, taking the gradient of $\eqref{ler-1}$,
 we use Duhamel's principle again for $t \in [\tau(0.6), \tau(1)]$ that
\begin{equation*}
\begin{aligned}
 \|(\nabla^2 v,\nabla^2 H)(t)\|_{L^\infty_x(\R^3)}\lesssim
 A^{18}(t-\tau(0.5))^{-\frac{1}{2}} +\int_{\tau(0.5)}^t (t-t')^{-\frac{3}{4}}(F,G)(t')\ \mathrm{d}t',
 \end{aligned}
\end{equation*}
where
\begin{equation*}\label{ler}
\begin{aligned}
&F(t)=\|\partial_kH_i\partial_iH_j+H_i\partial_{ik}H_j-
\partial_kv_i\partial_iv_j-v_i\partial_{ik}v_j\|_{L^6_x{(\R^3)}},\\
&G(t)=\|\partial_kH_i\partial_iv_j+H_i\partial_{ik}v_j-
\partial_kv_i\partial_iH_j-v_i\partial_{ik}H_j\|_{L^6_x{(\R^3)}}.
\end{aligned}
\end{equation*}
By estimates  \eqref{leo}, \eqref{wJwu} and \eqref{2w262J26}, we have
\begin{align*}
&\|\partial_kH_i\partial_iH_j+H_i\partial_{ik}H_j-
\partial_kv_i\partial_iv_j-v_i\partial_{ik}v_j\|
_{L^2_t L^6_x([\tau(0.5), \tau(1)] \times \R^3)} \lesssim A^{27},\\
&\|\partial_kH_i\partial_iv_j+H_i\partial_{ik}v_j-
\partial_kv_i\partial_iH_j-v_i\partial_{ik}H_j\|_{L^2_t L^6_x([\tau(0.5), \tau(1)] \times \R^3)} \lesssim A^{27}.
\end{align*}
Repeat the proof process in \eqref{v4wuH4wu}, we have
\begin{equation*}
\begin{aligned}
 \|(\nabla^2 v,\nabla^2 H)(t)\|_{L^4_t L^\infty_x([\tau(0.6), \tau(1)] \times \R^3)}\lesssim A^{27},
 \end{aligned}
\end{equation*}
and then
$$\|(\nabla^2 v,\nabla^2 H)(t)\|_{L^\infty_t L^{\infty}_x([\tau(0.7), \tau(1)] \times \R^3)}\lesssim A^{35}.
$$
Finally, if $I=[\widetilde{t}_1,\widetilde{t}_2]\subset\subset [0,1]$, we can borrow the scaling and translation transformation \eqref{scaling-01} and follow almost verbatim the proof of \eqref{VHL1I.est}  to derive the desired result. Here, we omit its details.
\end{proof}

\medskip

 To invoke  the Lemma \ref{carl-second}, we had to prove the annuli of estimates for $v$ and $H$. It is worth noting that, compared to the approach using covering theories in Tao \cite{MR4337421}, we, as in \cite{MR4278282},  provided a simple and direct proof from a new perspective by using the $\varepsilon$-regularity method.
\begin{prop}\label{vi}
Let $(v,H,p): [-1,1] \times \R^3 \to \R^3\times \R^3\times\R$ be a classical solution of \eqref{mhdeq} satisfying \eqref{able}. There exist a absolute constant $\varepsilon'_\ast\in(0,\frac14]$ such that if $0 < T' \leq \frac1{64}$, $R_0 \geq 1$ and
$$\mu'=-\frac{\log\varepsilon'_\ast}{\log A}.$$
Then there exists a scale
$$ 16R_0\sqrt{T'} \leq \widetilde{R} \leq 16e^{2\mu' A^{\mu'+2}} R_0\sqrt{T'},$$
such that on the region
$$ \Omega := \{ (t,x) \in [-T',0] \times \R^3: \widetilde{R} \leq |x| \leq \frac{A^{\mu'}}{4} \widetilde{R} \},$$
we have
$$\| \nabla^j (v,H) \|_{L^\infty_t L^\infty_x(\Omega)} \lesssim A^{-\frac{\mu'}{3}} (T')^{-\frac{j+1}{2}}\quad\text{for}\quad j=0,1,2.$$

\end{prop}

\begin{proof}
Fix any $R_0\geq 1$,  we can from \eqref{able} and Calder\'{o}n-Zygund inequality obtain
$$\sum_{k=0}^\infty\int_{-1}^0\int\limits_{(A^{\mu'})^{k}R_0\leq|x|
\leq(A^{\mu'})^{k+1}R_0}|v|^3+
|H|^3+|\Pi|^{\frac32}\ \mathrm{d}x\mathrm{d}t\leq A.$$
By the pigeonhole principle, there exists $k_0\in\{0,1,\ldots,\lceil A^{\mu'+1}\rceil\}$ such that
$$\int_{-1}^0\int\limits_{(A^{\mu'})^{k_0}R_0\leq|x|\leq(A^{\mu'})^{k_0+1}R_0}
|v|^3+|H|^3+|\Pi|^{\frac32}\ \mathrm{d}x\mathrm{d}t\leq A^{-\mu'}.$$
Let us define $R'\triangleq A^{\mu' k_0}R_0$, then
$$ R_0 \leq R' \leq e^{2\mu' A^{\mu'+2}} R_0,$$
and
\begin{align*}
\int_{-1}^0\int_{R'\leq|x|\leq A^{\mu'}R'}
|v|^3+|H|^3+|\Pi|^{\frac32}\ \mathrm{d}x\mathrm{d}t\leq A^{-\mu'}.
\end{align*}
which implies
$$
\sup_{x'\in\Gamma} \int_{-1}^0\int_{B_1(x')}
|v|^3+|H|^3+|\Pi|^{\frac32}\ \mathrm{d}x\mathrm{d}t\leq A^{-\mu'}=\varepsilon'_\ast.
$$
with
$$
\Gamma \triangleq \{ x : R'+1 \leq |x| \leq A^{\mu'} R'-1 \}.
$$
Here, $\varepsilon'_\ast$ is defined in Corollary \ref{cor1.fulu}. Then, by applying Corollary \ref{cor1.fulu}, one has for $j=0,1,2$
$$
\sup_{x'\in\Gamma}\|\nabla^j (v,H)\|_{L^\infty(Q_{\frac1{8}}(x',0))}\leq C'_jA^{-\frac{\mu'}{3}}
$$
which also implies
\begin{align}\label{jvjhL}
\|\nabla^j (v,H)\|_{L^\infty(\Gamma\times[-\frac{1}{64},0])}
\leq C'_jA^{-\frac{\mu'}{3}}.
\end{align}
Now to conclude the proof,  we take a scaling transformation
$$
 (v^{\lambda'},H^{\lambda'})(t,x) \triangleq \lambda' (v,H)({\lambda'}^2 t, \lambda' x)\quad j=0,1,2,
 $$
with $\lambda'=8\sqrt{T'}$,  which is also the classical solution of \eqref{mhdeq} in $[-\frac1{64T'},\frac1{64T'}]\times \R^3$ with $[-1,1]\subset [-\frac1{64T'},\frac1{64T'}]$.
Then from \eqref{jvjhL}, one has for $j=0,1,2$
\begin{align*}
&\|\nabla^j(v,H)\|_{L^\infty_tL^\infty_x([-T',0]\times(B(0,\frac{A^{\mu'}}{4} \widetilde{R}) \backslash B(0,\widetilde{R}))}\\
&\quad=\frac1{(\lambda')^{j+1}}
\|\nabla^j(v^{\lambda'},H^{\lambda'})\|_{L^\infty_tL^\infty_x([-\frac{1}{64},
0]\times(B(0,\frac{A^{\mu'}}{2} R') \backslash B(0,2R'))}\\
&\quad\lesssim A^{-\frac{\mu'}{3}}(T')^{-\frac{j+1}{2}}
\end{align*}
with $\widetilde{R}=16\sqrt{T'} R'$. Here
$$
\{x~:~2R'\leq|x|\leq\frac{A^{\mu'}}{2}R'\}\subset\Gamma,
$$
due to
$$A^{\mu'} R'-1\geq\frac{A^{\mu'}}{2}R'\geq2R'\geq R'+1.$$
\end{proof}

\subsection{Frequency bubbles of concentration}\label{5.bubble}
\begin{prop}\label{iv}
Let $(v,H,\pi): [-1,1] \times \R^3 \to \R^3\times \R^3\times\R$ be a classical solution of \eqref{mhdeq} satisfying \eqref{able}.
 If  there exists $(t_1,x_1)\in [0,1] \times \R^3$
\begin{equation*}
 |P_{N_1} (v,H)(t_1,x_1)| \geq A_1^{-1} N_1
\end{equation*}
with $N_1 \geq A_3$. Then there exists $(t_2,x_2) \in [-1,t_1] \times \R^3$ and $N_2 \in [A_2^{-1} N_1, A_2 N_1]$ such that
$$ A_3^{-1} N_1^{-2} \leq t_1-t_2 \leq A_3 N_1^{-2},\quad  |x_2-x_1| \leq A_4 N_1^{-1}  $$
and
\begin{equation*}
 |P_{N_2} (v,H)(t_2,x_2)| \geq A_1^{-1} N_2.
 \end{equation*}
 \end{prop}

\begin{proof}
First, Let us define $(v_{N_1}, H_{N_1})(t,x)=N_1^{-1}(v,H)\Big(\frac{t}{N_1^2},\frac{x}{N_1}\Big)$ which is a solution of \eqref{mhdeq}  in $(-N_1^2,N_1^2)\times \R^3$. Notices that by a simple calculation,
\begin{align*}
P_1(v_{N_1}, H_{N_1})(t,x)=N_1^{-1}P_{N_1}(v,H)\Big(\frac{t}{N_1^2},\frac{x}{N_1}\Big),
\end{align*}
which implies that
\begin{align}\label{n1-norm}
|P_1(v_{N_1}, H_{N_1})(x_1,t_1)| \geq A_1^{-1}.
\end{align}
Due to translation invariance, we also assume $(x_1,t_1)=(0,0)$. In the following, we consider $(v_{N_1}, H_{N_1})$ in $(-N_1^2,N_1^2)\times \R^3$ which still denoted by $(v,H)$ for simplicity.
Now, assume for contradiction that the claim fails, then for all $A_2^{-1} \leq N \leq A_2$, we have
\begin{align}\label{PN2vLwu2}
\| P_{N} (v,H)\|_{L^\infty_t L^\infty_x( [-A_3, -A_3^{-1}] \times B(0, A_4) )} \leq A_1^{-1} N.
\end{align}
Now we claim that the estimate \eqref{PN2vLwu2} is valid in $ [-A_3, 0] \times B(0, A_4)$. To do this, we apply $P_{N}$ to both sides of \eqref{mhdeq} and  find that by \eqref{able} and \eqref{bern}
$$ \|P_{N} \Delta (v,H) \|_{L^\infty_x(\R^3)} \lesssim N^3\|(v,H) \|_{L^3_x(\R^3)} \lesssim AN^3,$$
On the other hand, using H\"older inequality, one derives
$$\| v \otimes v-H \otimes H \|_{L^{3/2}_x(\R^3)}\lesssim \|v\|^2_{L^3_x(\R^3)}+\|H\|^2_{L^3_x(\R^3)}\lesssim A^2,$$
$$\| v \times H \|_{L^{3/2}_x(\R^3)}\lesssim \|v\|_{L^3_x(\R^3)}\|H\|_{L^3_x(\R^3)}\lesssim A^2.$$
This, together Lemma \ref{mult}, yields
$$ \| P_{N} \mathbb{P} \div(v \otimes v-H \otimes H) \|_{L^\infty_x(\R^3)} \lesssim N^3\| v \otimes v-H \otimes H \|_{L^{3/2}_x(\R^3)}\lesssim A^2N^3,$$
$$ \| P_{N} \nabla\times(v\times H) \|_{L^\infty_x(\R^3)} \lesssim N^3\| v \times H\|_{L^{3/2}_x(\R^3)}\lesssim A^2N^3,$$
from which one further obtains
\EQS{\label{dtPNvH}
\| \partial_tP_{N} (v,H) \|_{L^\infty_tL^\infty_x([-A_3^{-1},0]\times\R^3)} &\lesssim \| P_{N} \Delta (v,H)\|_{L^\infty_tL^\infty_x([-A_3^{-1},0]\times\R^3)} \\&
 \quad+ \| P_{N} \mathbb{P} \div(v \otimes v-H \otimes H) \|_{L^\infty_tL^\infty_x([-A_3^{-1},0]\times\R^3)}\\&
 \quad+ \| P_{N} \nabla\times(v\times H) \|_{L^\infty_tL^\infty_x([-A_3^{-1},0]\times\R^3)}\\&
 \lesssim A^2N^3.
 }
Therefore, it is clear that for any $t\in[-A_3^{-1},0]$
\begin{align*}
\| P_{N} v(\cdot,t)\|_{L^\infty_x(B(0, A_4))}&\leq\Big|\| P_{N} v(\cdot,t)\|_{ L^\infty_x(B(0, A_4))}-\|P_{N} v(\cdot,-A_3^{-1})\|_{L^\infty_x(B(0, A_4))}\Big|+ A_1^{-1} N\\
&=\Big| \int_{-A_3^{-1}}^{t}\partial_s\|P_{N} v(\cdot,s)\|_{L^\infty_x(B(0, A_4))}\ \mathrm{d}s\Big|+ A_1^{-1} N\\
&\leq A_3^{-1}\| \partial_tP_{N} v \|_{L^\infty_tL^\infty_x([-A_3^{-1},0]\times\R^3)}+ A_1^{-1} N\\
&
\leq A_3^{-1}A^2N^3+ A_1^{-1} N\lesssim A_1^{-1} N.
\end{align*}
Similarly,
$$
\| P_{N_2} H\|_{L^\infty_t L^\infty_x( [-A_3^{-1}, 0] \times B(0, A_4) )} \lesssim A_1^{-1} N.
$$
The above calculation, along  with \eqref{PN2vLwu2}, yields
\EQS{\label{PN2vLwu2-1}
 \| P_{N} (v,H)\|_{L^\infty_t L^\infty_x( [-A_3, 0] \times B(0, A_4) )} \lesssim A_1^{-1} N, \quad A_2^{-1} \leq N \leq A_2.
 }
In the following, we split our proof into four steps for the clarity.\\
{\bf Step 1}.\quad The $L^\infty_tL^{\frac32}_x$ estimates of $P_{N}(v,H)$ for $N\geq A_2^{-1}.$\\
For $t \in [-A_3,0]$, using Duhamel's formula, \eqref{mhdeq} and H\"older inequality, it follows that
\EQS{\label{PNvHL32}
&\|P_{N} (v,H)(\cdot,t)\|_{L^{\frac{3}{2}}_x(B(0, A_4))} \\
&\quad\leq \| e^{(t+2A_3)\Delta} P_{N} (v,H)(-2A_3) \|_{L^{\frac{3}{2}}_x(B(0,A_4))} \\
&\quad\quad + \int_{-2A_3}^t \| e^{(t-t')\Delta} P_{N}  \{\mathbb{P}\mathrm{div}(v\otimes v-H\otimes H),\nabla\times(v\times H)\}(t')\|_{L^{\frac{3}{2}}_x(\R^3)}\ \mathrm{d}t'\\
&\quad\lesssim AA_4 e^{-\frac{N^2(t+2A_3)}{20}} +\int_{-2A_3}^t e^{-\frac{N^2(t-t')}{20}} N\Big(\|v\|_{L^{3}_x(\R^3)}+\|H\|_{L^{3}_x(\R^3)}\Big)\ \mathrm{d}t'\\
&\quad\lesssim AA_4 e^{-\frac{N^2A_3}{20}}+A^2N^{-1}\lesssim A^2N^{-1}.
}
{\bf Step 2}.\quad The $L^\infty_tL^{1}_x$ estimates of $P_{N}(v,H)$ for $A_2^{-\frac12} \leq N \leq A^3A_4^{40}$.\\
For $t \in [-\frac{A_3}{2},0]$, from \eqref{bern-2} and \eqref{able} we can obtain
\begin{align}\label{PN2vL1}
\|P_{N} v(t)\|_{L^1_x(B(0, \frac{5A_4}{8}))} &\leq \| e^{(t+A_3)\Delta} P_{N} v(-A_3) \|_{L^1_x(B(0, \frac{5A_4}{8}))} \nonumber\\
&\quad+ \int_{-A_3}^t \| P_{N}e^{(t-t')\Delta}\mathbb{P} \mathrm{div}[\widetilde{P}_{N}(v \otimes v-H\otimes H)](t')\|_{L^1_x(B(0, \frac{5A_4}{8}))}\ \mathrm{d}t'\nonumber\\
&\lesssim A A_4^2 e^{- \frac{N^2A_3}{40}}+ \int_{-A_3}^tN e^{-\frac{N^2(t-t')}{20}}\Big(\|\widetilde{P}_{N}(v \otimes v-H\otimes H)\|_{L^1_x(B(0,\frac{3A_4}{4})}\nonumber\\
&\quad\quad+ A_4^{-50}A_4\| \widetilde{P}_{N}(v \otimes v-H\otimes H) \|_{L^{\frac{3}{2}}_x(\R^3)}\Big)\ \mathrm{d}t'\nonumber\\
&\lesssim A^3N^{-2}+N^{-1}\Big(\|\widetilde{P}_{N}(v \otimes v)\|_{L^\infty_t L^1_x([-A_3, 0] \times B(0,\frac{3A_4}{4}))}\nonumber\\
&\quad\quad
+\|\widetilde{P}_{N}(H \otimes H)\|_{L^\infty_t L^1_x([-A_3, 0] \times B(0,\frac{3A_4}{4}))}\Big)+A_4^{-40}N^{-1}.
\end{align}
In order to estimate the  $L^\infty_t L^1_x([-A_3, 0] \times B(0,\frac{3A_4}{4}))$ for $\widetilde{P}_{N}(v \otimes v, H \otimes H)$,
we can write
\begin{equation*}
\begin{aligned}
\widetilde{P}_{N}(v \otimes v)=\widetilde{P}_{N}(P_{>\frac{N}{100}}v \otimes v)+\widetilde{P}_{N}(P_{\leq\frac{N}{100}}v \otimes P_{>\frac{N}{100}}v),
\end{aligned}
\end{equation*}
where we have used the fact that \eqref{low20}. From \eqref{local}, \eqref{able}, \eqref{PNvHL32} and H\"older inequality, we have
\begin{align*}
 &\| \widetilde{P}_{N}(P_{>{\frac{N}{100}}}v \otimes v, P_{\leq\frac{N}{100}}v \otimes P_{>\frac{N}{100}}v)\|_{L^\infty_t L^1_x( [-A_3, 0] \times B(0, \frac{3A_4}{4}) )}\\
 &\quad\lesssim \| (P_{>\frac{N}{100}}v \otimes v, P_{\leq\frac{N}{100}}v \otimes P_{>\frac{N}{100}}v)\|_{L^\infty_t L^{1}_x( [-A_3, 0] \times B(0, A_4) )}\\
  &\quad\quad+A_4^{-50}A_4\| (P_{>\frac{N}{100}}v \otimes v,P_{\leq\frac{N}{100}}v \otimes P_{>\frac{N}{100}}v)\|_{L^\infty_t L^{\frac32}_x( [-A_3, 0] \times \R^3 )}\\
  &\quad\lesssim A\sum_{N'>\frac{N}{100}}\| P_{N'} v\|_{L^\infty_t L^{\frac{3}{2}}_x( [-A_3, 0] \times B(0, A_4) )}\\
  &\quad\quad +\| P_{\leq\frac{N}{100}}v \|_{L^\infty_t L^{3}_x( [-A_3, 0] \times B(0, A_4) )}
  \|P_{>\frac{N}{100}}v\|_{L^\infty_t L^{\frac{3}{2}}_x( [-A_3, 0] \times B(0, A_4) )}
  +A^{-50}_4A_4A^2\\
 &\quad\lesssim A^3 \sum_{N'>\frac{N}{100}}{N'}^{-1}+A^3 N^{-1}+A_4^{-40}\lesssim A^3 N^{-1}+A_4^{-40}.
\end{align*}
We thus have
\begin{align*}
\|\widetilde{P}_{N}(v \otimes v)\|_{L^\infty_t L^1_x([-A_3, 0] \times B(0,\frac{3A_4}{4}))}\lesssim A^3 N^{-1}+A_4^{-40}.
\end{align*}
Repeat the calculation process, obviously we can get
\begin{align*}
\|\widetilde{P}_{N}(H \otimes H)\|_{L^\infty_t L^1_x([-A_3, 0] \times B(0,\frac{3A_4}{4}))}\lesssim A^3 N^{-1}+A_4^{-40}.
\end{align*}
Combining with \eqref{PN2vL1}, we have the estimate
\EQS{\label{PN2vL1-1}
\|P_{N} v\|_{L^\infty_t L^1_x([-A_3, 0] \times B(0, \frac{5A_4}{8}))} \lesssim A^3 N^{-2}+A_4^{-40}N^{-1}\lesssim A^3 N^{-2}
}
due to $A_2^{-\frac12} \leq N \leq A^3A_4^{40}$. Similarly,
\EQS{\label{PN2vL1-2}
\|P_{N} H\|_{L^\infty_t L^1_x([-A_3, 0] \times B(0, \frac{5A_4}{8}))} \lesssim A^3 N^{-2}+A_4^{-40}N^{-1}\lesssim A^3 N^{-2}.
}
for $A_2^{-\frac12} \leq N \leq A^3A_4^{40}$.\\
{\bf Step 3}.\quad The $L^\infty_tL^{2}_x$ estimates of $P_{N}(v,H)$ for $A_2^{-\frac13} \leq N \leq A_2^{\frac13}$.\\
By using Duhamel's formula, \eqref{mhdeq} and the triangle inequality as before, for any $t \in [-\frac{A_3}{3},0]$ we have
\begin{align}\label{tin3A30PN2v}
\|P_{N} v(t)\|_{L^2_x(B(0, \frac{A_4}{4}))} &\leq \| e^{(t+\frac{A_3}{2})\Delta} P_{N} v(-\frac{A_3}{2}) \|_{L^2_x(B(0, \frac{A_4}{4}))} \nonumber\\
&\quad+ \int_{-\frac{A_3}{2}}^t \| P_{N}e^{(t-t')\Delta}\mathbb{P}\mathrm{div}[\widetilde{P}_{N}(v \otimes v-H\otimes H)](t')\|_{L^2_x(B(0, \frac{A_4}{4}))}\ \mathrm{d}t'\nonumber\\
&
\lesssim A A_4^{\frac{1}{2}} e^{- \frac{N^2A_3}{120}}+N^{\frac12}\Big(\|\widetilde{P}_{N}(v \otimes v)\|_{L^\infty_t L^1_x([-\frac{A_3}{2}, 0] \times B(0,\frac{A_4}{3}))}\nonumber\\
&\quad
+\|\widetilde{P}_{N}(H \otimes H)\|_{L^\infty_t L^1_x([-\frac{A_3}{2}, 0] \times B(0,\frac{A_4}{3}))}\Big)+A_4^{-40}.
\end{align}
In order to estimate $\|\widetilde{P}_{N}(v \otimes v)\|_{L^\infty_t L^1_x([-\frac{A_3}{2}, 0] \times B(0,\frac{A_4}{3}))}$, we  split
\begin{align*}
&\widetilde{P}_{N}(v(t') \otimes v(t'))\\
&\quad=\sum_{N'\sim N^{''} \lesssim N}
\widetilde{P}_{N}(P_{N'} v(t') \otimes P_{N^{''}} v(t'))
+\sum_{N'\lesssim N^{''}\sim N}
\widetilde{P}_{N}(P_{N'} v(t') \otimes P_{N^{''}} v(t'))\\
&\quad\quad+\sum_{N^{''}\lesssim N^{'} \sim  N}
\widetilde{P}_{N}(P_{N'} v(t') \otimes P_{N^{''}} v(t'))
+\sum_{N\lesssim N'\sim N^{''} }
\widetilde{P}_{N}(P_{N'} v(t') \otimes P_{N^{''}} v(t')).
\end{align*}
Here $\sum_{N'\sim N^{''} \lesssim N}
\widetilde{P}_{N}(P_{N'} v(t') \otimes P_{N^{''}} v(t'))$, called by the ``low-low" term, disappears due to \eqref{low20}. Then, we use the triangle inequality, H\"older inequality, \eqref{local}, and \eqref{PN2vLwu2-1}-\eqref{PN2vL1-2} to deduce
\begin{align}\label{PN2vvLwuL2}
&\|\widetilde{P}_{N}(v \otimes v)\|_{L^\infty_t L^1_x([-\frac{A_3}{2}, 0] \times B(0,\frac{A_4}{3}))}\nonumber\\
&\quad\lesssim \sum_{N'\lesssim N^{''}\sim N}\|P_{N'}v \otimes P_{N^{''}}v\|_{L^\infty_t L^1_x([-\frac{A_3}{2}, 0] \times B(0,\frac{A_4}{2}))}\nonumber\\
&\quad\quad+\sum_{N\lesssim N'\sim N^{''} }\|P_{N'}v \otimes P_{N^{''}}v\|_{L^\infty_t L^{1}_x([-\frac{A_3}{2}, 0] \times B(0,\frac{A_4}{2}))}+A_4^{-40}\nonumber\\
&\quad\lesssim \Big(\sum_{N'\leq A_2^{-1}}+\sum_{A_2^{-1}\leq N'\lesssim N}
\Big)\|P_{N'}v\|_{L^\infty_t L^\infty_x([-\frac{A_3}{2}, 0] \times B(0,\frac{A_4}{2}))}\|P_{N}v\|_{L^\infty_t L^1_x([-\frac{A_3}{2}, 0] \times B(0,\frac{A_4}{2}))}\nonumber\\
&\quad\quad+\sum_{N\lesssim N'\leq A_2}
\|P_{N'}v\|_{L^\infty_t L^{1}_x([-\frac{A_3}{2}, 0] \times B(0,\frac{A_4}{2}))}\|P_{N'}v\|_{L^\infty_t L^\infty_x([-\frac{A_3}{2}, 0] \times B(0,\frac{A_4}{2}))}\nonumber\\
&\quad\quad+\sum_{N'\geq A_2}
\|P_{N'}v\|^{\frac{3}{2}}_{L^\infty_t L^{\frac{3}{2}}_x([-\frac{A_3}{2}, 0] \times B(0,\frac{A_4}{2}))}\|P_{N'}v\|^{\frac{1}{2}}_{L^\infty_t L^\infty_x([-\frac{A_3}{2}, 0] \times B(0,\frac{A_4}{2}))}+A_4^{-40}\nonumber\\
&\quad\lesssim \sum_{N'\leq A_2^{-1}}AN'A^3N^{-2}+\sum_{A_2^{-1}\leq N'\lesssim N}A_1^{-1}N'A^3N^{-2}\nonumber\\
&\quad\quad+
\sum_{N\lesssim N'\leq A_2}A^{3}(N')^{-2}A_1^{-1}N'
+\sum_{N'\geq A_2}(A^2(N')^{-1})^{\frac32}(AN')^{\frac12}+A_4^{-40}\nonumber\\
&\quad\lesssim A^4A_2^{-1}N^{-2}+A_1^{-1}A^3N^{-1}
+A^4A_2^{-1}+A_4^{-40}.
\end{align}
Similarly, we also have
\begin{align*}
\|\widetilde{P}_{N}(H \otimes H)\|_{L^\infty_t L^1_x([-\frac{A_3}{2}, 0] \times B(0,\frac{A_4}{3}))}
\lesssim A^4A_2^{-1}N^{-2}+A_1^{-1}A^3N^{-1}
+A^4A_2^{-1}+A_4^{-40},
\end{align*}
which, along with \eqref{tin3A30PN2v} and \eqref{PN2vvLwuL2}, yields
\begin{align*}
\|P_{N} v\|_{L^\infty_t L^2_x([-\frac{A_3}{3}, 0] \times B(0, \frac{A_4}{4}))}
\lesssim A^{3}A_1^{-1}N^{-\frac12}.
\end{align*}
Similarly,
\begin{align*}
\|P_{N} H\|_{L^\infty_t L^2_x([-\frac{A_3}{3}, 0] \times B(0, \frac{A_4}{4}))}
\lesssim A^{3}A_1^{-1}N^{-\frac12}.
\end{align*}
{\bf Step 4}.\quad End of the proof.\\
First, using the hypothesis \eqref{n1-norm} and Duhamel's formula, one has
\begin{align}\label{P1vv-HH}
A_1^{-1} &\leq|P_1 (v,H)(0,0)|\nonumber\\
&\leq |e^{\frac{A_3}{4} \Delta} P_1 (v,H)(-\frac{A_3}{4},0)| \nonumber\\
&\quad+ \int_{-\frac{A_3}{4}}^0 |e^{-t'\Delta} P_1 \widetilde{P}_1\{\mathbb{P}\mathrm{div}[ (v \otimes v-H \otimes H)], \nabla\times(v\times H)\}(t',0)|\ \mathrm{d}t'\nonumber\\
&\lesssim Ae^{-\frac{A_3}{80}}+\int_{-\frac{A_3}{4}}^0 e^{\frac{t'}{20}} \Big(\|\widetilde{P}_1(v \otimes v-H \otimes H, v\times H)\|_{L^\infty_x(B(0,\frac{A_4}{8}))}+A_4^{-50}A^2\Big) \mathrm{d}t'\nonumber\\
&\leq {\frac{1}{2}}A_1^{-1}+C \int_{-\frac{A_3}{4}}^0 e^{\frac{t'}{20}} \Big(\|\widetilde{P}_1(v \otimes v)\|_{L^\infty_x(B(0,\frac{A_4}{8}))}\nonumber\\& \quad+
  \|\widetilde{P}_1 (H \otimes H)\|_{L^\infty_x(B(0,\frac{A_4}{8}))}+\|\widetilde{P}_1 (v\times H)\|_{L^\infty_x(B(0,\frac{A_4}{8}))}\Big) \mathrm{d}t'.
 \end{align}
 To conclude the proof,  we need to consider $\|\widetilde{P}_1(v \otimes v) \|_{L^\infty_t L^\infty_x([-\frac{A_3}{4}, 0] \times B(0,\frac{A_4}{8}))}$, $\|\widetilde{P}_1(H \otimes H) \|_{L^\infty_t L^\infty_x([-\frac{A_3}{4}, 0] \times B(0,\frac{A_4}{8}))}$ and $\|\widetilde{P}_1(v \times H) \|_{L^\infty_t L^\infty_x([-\frac{A_3}{4}, 0] \times B(0,\frac{A_4}{8}))}$, respectively.  Indeed,
 \begin{align*}
&\|\widetilde{P}_{1}(v \otimes v)\|_{L^\infty_t L^\infty_x([-\frac{A_3}{4}, 0] \times B(0,\frac{A_4}{8}))}\\
&\quad\lesssim \sum_{N'\lesssim N^{''}\sim 1}\|P_{N'}v \otimes P_{N^{''}}v\|_{L^\infty_t L^\infty_x([-\frac{A_3}{4}, 0] \times B(0,\frac{A_4}{4}))}\\
&\quad\quad+\sum_{1\lesssim N'\sim N^{''} \leq A_2^{\frac13}} \|P_{N'}v \otimes P_{N^{''}}v\|_{L^\infty_t L^{1}_x([-\frac{A_3}{4}, 0] \times B(0,\frac{A_4}{4}))}\\
&\quad\quad+\sum_{ N'\sim N^{''} \geq A_2^{\frac13}} \|P_{N'}v \otimes P_{N^{''}}v\|_{L^\infty_t L^{1}_x([-\frac{A_3}{4}, 0] \times B(0,\frac{A_4}{4}))}+A_4^{-40}\\
&\quad\lesssim\sum_{A_2^{-1}\leq N'\lesssim N^{''}\sim 1}\|P_{N'}v\|_{L^\infty_t L^\infty_x([-\frac{A_3}{4}, 0] \times B(0,\frac{A_4}{4}))} \|P_{N^{''}}v\|_{L^\infty_t L^\infty_x([-\frac{A_3}{4}, 0] \times B(0,\frac{A_4}{4}))}\\
&\quad\quad+\sum_{N'\leq A_2^{-1},~N^{''}\sim 1}\|P_{N'}v\|_{L^\infty_t L^\infty_x([-\frac{A_3}{4}, 0] \times B(0,\frac{A_4}{4}))} \|P_{N^{''}}v\|_{L^\infty_t L^\infty_x([-\frac{A_3}{4}, 0] \times B(0,\frac{A_4}{4}))}\\
&\quad\quad+\sum_{1\lesssim N'\sim N^{''}\leq A_2^{\frac13}}\|P_{N'}v\|_{L^\infty_t L^2_x([-\frac{A_3}{4}, 0] \times B(0,\frac{A_4}{4}))} \|P_{N^{''}}v\|_{L^\infty_t L^2_x([-\frac{A_3}{4}, 0] \times B(0,\frac{A_4}{4}))}\\
&\quad\quad+\sum_{A_2^{\frac13}\leq N'\sim N^{''}}\|P_{N'}v\|^{\frac32}_{L^\infty_t L^{\frac32}_x([-\frac{A_3}{4}, 0] \times B(0,\frac{A_4}{4}))} \|P_{N^{''}}v\|^{\frac12}_{L^\infty_t L^\infty_x([-\frac{A_3}{4}, 0] \times B(0,\frac{A_4}{4}))}+A_4^{-40}\\
&\quad\lesssim\sum_{A_2^{-1}\leq N'\lesssim N^{''}\sim 1}(A_1^{-1}N')(A_1^{-1}N^{''})
+\sum_{N'\leq A_2^{-1},~N^{''}\sim 1}(AN')(AN^{''})\\
&\quad\quad+\sum_{1\lesssim N'\sim N^{''}\leq A_2^{\frac13}}\Big[A^3A_1^{-1}(N')^{-\frac{1}{2}}\Big]^2
+\sum_{A_2^{\frac13}\leq N'\sim N^{''}}[A^2(N')^{-1}]^{\frac32}(AN')^{\frac12}+A_4^{-40}\\
&\quad\lesssim A_1^{-2}+A^2A_2^{-1}+
A^{6}A_1^{-2}+A^{\frac{7}{2}}A_2^{-\frac{1}{3}}
+A_4^{-40}\lesssim A^{6}A_1^{-2}.
\end{align*}
Similarly, we have
\begin{align*}
\|\widetilde{P}_{1}(H \otimes H)\|_{L^\infty_t L^\infty_x([-\frac{A_3}{4}, 0] \times B(0,\frac{A_4}{8}))}+
\|\widetilde{P}_{1}(v \times H)\|_{L^\infty_t L^\infty_x([-\frac{A_3}{4}, 0] \times B(0,\frac{A_4}{8}))}\lesssim A^{6}A_1^{-2}.
\end{align*}
Putting the above estimate into \eqref{P1vv-HH} we get
$$ A_1^{-1} \lesssim A^{6}A_1^{-2}$$
which derives a contradiction.
That is there exist $(\widetilde{t},\widetilde{x})\in[-A_3, -A_3^{-1}] \times B(0,A_4)$ and $\widetilde{N}\in[A_2^{-1}, A_2]$ such that
\begin{align*}
|P_{\widetilde{N}}(v_{N_1},H_{N_1})(\widetilde{t},\widetilde{x})|
 \geq A_1^{-1}\widetilde{N}.
\end{align*}
On the other hand,
\begin{align*}
N_1P_{\widetilde{N}}(v_{N_1},H_{N_1})(\widetilde{t},\widetilde{x})
&=P_{\widetilde{N}}(v,H)(\frac{\widetilde{t}}{N_1^{2}},
\frac{\widetilde{x}}{N_1})\\
&=N_1^3\int_{\frac{\widetilde{N}}{4}\leq|\xi|\leq \widetilde{N}}e^{2\pi i\widetilde{x}\cdot\xi}[\varphi(\frac{\xi}{\widetilde{N}})-
\varphi(\frac{2\xi}{\widetilde{N}})](\widehat{v},\widehat{H})(\frac{
\widetilde{t}}{N_1^{2}},N_1\xi)\ \mathrm{d}\xi\\
&=\int_{\frac{N_1\widetilde{N}}{4}\leq|\xi'|\leq N_1\widetilde{N}}e^{2\pi i\frac{\widetilde{x}}{N_1}\cdot\xi'}[\varphi(\frac{\xi'}{N_1\widetilde{N}})-
\varphi(\frac{2\xi'}{N_1\widetilde{N}})](\widehat{v},\widehat{H})(
\frac{\widetilde{t}}{N_1^{2}},\xi')\ \mathrm{d}\xi'\\
&=P_{N_1\widetilde{N}}(v,H)(t_2,x_2),
\end{align*}
with
$(t_2,x_2)=(\frac{\widetilde{t}}{N_1^{2}},
\frac{\widetilde{x}}{N_1})\in[-A_3N_1^{-2}, -A_3^{-1}N_1^{-2}] \times B(0,A_4N_1^{-1})$.
Therefore,
\begin{align*}
|P_{N_2}(v,H)(t_2,x_2)| \geq A_1^{-1}N_2.
\end{align*}
with $N_2=N_1\widetilde{N}\in[A_2^{-1}N_1, A_2N_1]$. This, along with the translation invariance of system \eqref{mhdeq}, concludes the proof.
\end{proof}

\begin{cor}\label{v}
 Let $(v,H,p): [-1,1] \times \R^3 \to \R^3\times \R^3\times\R$ be a classical solution of \eqref{mhdeq} satisfying \eqref{able} and $N_0>A_4$. For $t_0=1$, if there exists  $x_0\in \R^3$  such that
$$ |P_{N_0} (v,H)(t_0,x_0)| \geq A_1^{-1} N_0.$$
Then for every $A_4 N_0^{-2} \leq T_1 \leq A_4^{-1} $, there exist
$$ (t_1,x_1)\in[t_0-T_1, t_0 - A_3^{-3} T_1] \times B(x_0,A_4^{2}T_1^{\frac{1}{2}}),$$
and
$$ A_3^{-1} T_1^{-\frac{1}{2}} \leq N_1 \leq A_3^{\frac12} T_1^{-\frac{1}{2}},$$
such that
$$ |P_{N_1} (v,H)(t_1,x_1)| \geq A_1^{-1} N_1.$$
\end{cor}

\begin{proof}
By iteratively applying Proposition \ref{iv}, we may find a sequence
$$(t_0,x_0), (t_1,x_1), \dots, (t_n,x_n) \in [-1,1]\times\R^3,$$
and $N_0,N_1,\dots,N_n>0$ for some $n \geq 1$, with the properties
\begin{align}
&|P_{N_i} (v,H)(t_i,x_i)| \geq A_1^{-1} N_i \label{iter-1}\\
&A_2^{-1} N_{i-1} \leq N_i \leq A_2 N_{i-1}, \label{iter-2}\\
&A_3^{-1} N_{i-1}^{-2} \leq t_{i-1} - t_i \leq A_3 N_{i-1}^{-2}, \label{iter-3}\\
&|x_i - x_{i-1}| \leq A_4 N_{i-1}^{-1},\label{iter-4}
\end{align}
for all $i=1,2,\dots,n$.
 Firstly, we claim the above  iterations $n<\infty$. Indeed, since $(v,H)$ is a classical solution of \eqref{mhdeq}, one has
$$A_1^{-1}N_i\leq |P_{N_i} (v,H)(t_i,x_i)|\leq \|(v,H)\|_{L^\infty_tL^\infty_x([-1,1]\times\R^3)}\triangleq L<\infty,$$
which, together with \eqref{iter-3}, implies
 \begin{align*}
 t_{i-1}-t_i\geq A_3^{-1}A_1^{-2}L^{-2}.
 \end{align*}
This leads us to infer
$$n\leq 2A_3A_1^2L^2<\infty.$$
Secondly, we prove $t_n<1-T_1$ for any  $T_1\in [A_4 N_0^{-2}, A_4^{-1} ]$. In fact, we notice by Proposition \ref{iv} that if the iteration \eqref{iter-1}-\eqref{iter-4} stops, either $N_n < A_3$  or $t_n$ is close to $-1$.
If $N_n < A_3 $, then by \eqref{iter-3}, \eqref{iter-2}
$$
t_{n-1} - t_n \geq A_3^{-1} N_{n-1}^{-2} \geq A_3^{-1}A_2^{-2} N_n^{-2} > A_3^{-3} A_2^{-2} .
$$
This implies
$$t_n <1 -A_3^{-3} A_2^{-2}<1-A_4^{-1}\leq 1-T_1\leq 1-A_4N_0^{-2}<t_1,$$
due to $t_{n-1}< 1.$
On the other hand, if $t_n \in [-1,0)$,  it is clear that
$$
t_n <0< 1 -A_4^{-1}\leq 1-T_1\leq 1-A_4N_0^{-2}<1-A_3N_0^{-2}<t_1,
$$
i.e., $t_n<1-T_1<t_1$. Now, we define
$$
m=\max\{0\leq i<n~|~ t_i \geq 1 - T_1\},
$$
then $1 \leq m \leq n-1~(n\geq 2)$ and $t_{m+1}<1-T_1$.  Next, we split our proof into two steps. \medskip

{\bf Step 1}.\quad{\em For any $T_1\in [A_4 N_0^{-2}, A_4^{-1} ]$, there exists some $1\leq i\leq m$ such that
\EQS{\label{aim-1-1}
  A_3^{-1} T_1^{-\frac{1}{2}}\leq N_i \lec  A_3^{\frac{1}{2}} T_1^{-\frac{1}{2}}, \quad 1-T_1\leq t_i\lesssim 1-A_3^{-3}T_1
}}
In fact, from \eqref{iter-3},
\begin{equation*}
T_1<1-t_{m+1}=\sum_{i=1}^{m+1}(t_{i-1}-t_i)\leq\sum_{i=0}^m A_3 N_i^{-2}.
\end{equation*}
this, together with the  pigeonhole principle,   shows that there exists  some $i\in \{0,1\dots,m\}$ such that
\EQS{\label{aim-1}
N_i^{-1} \geq\frac{1}{m+1} A_3^{-\frac12} T_1^{\frac{1}{2}}.
}
 If the above inequality is valid for $i=0$, then
$$A_4A_3^{-1} T_1\lesssim A_4 N_0^{-2}\leq T_1,$$
which derives a contradiction. Thus \eqref{aim-1} is valid only for some $1 \leq i \leq m$. Finally, \eqref{aim-1}, along with \eqref{iter-3} and \eqref{iter-2}, implies
\begin{align*}
T_1\geq 1- t_i \geq t_{i-1} - t_i\geq A_3^{-1} N_{i-1}^{-2}\geq A_3^{-1}A_2^{-2} N_i^{-2} \geq A_3^{-2}N_i^{-2} \gec A_3^{-3} T_1.
\end{align*}
 This, together with \eqref{aim-1},  leads to \eqref{aim-1-1}.

{\bf Step 2}. {\em Let $i$ be given by Step 1, then
  \EQS{\label{aim-2-0}
  |x_i - x_0|\lesssim A_4^{2} T_1^{\frac{1}{2}}.
  }}
 In fact, from \eqref{iter-4}  it is clear that
\EQS{\label{aim-2-1}
|x_i - x_0|\leq \sum_{k=1}^{i} |x_k - x_{k-1}|\leq A_4 \sum_{k=1}^{i} N_{k-1}^{-1}.
}
To conclude the proof, we have to estimate $\sum_{k=1}^{i} N_{k-1}^{-1}$. First, we notices that  from \eqref{dtPNvH} we have for all $t \in [t_i - A_3^{-1} N_i^{-2}, t_i]~(i\in \{0,1\dots,m-1\})$
\begin{align*}
&\| P_{N_i} (v,H)(t_i) \|_{L^\infty_x(\R^3)}-\| P_{N_i} (v,H)(t) \|_{L^\infty_x(\R^3)}\\
&\quad\leq\Big\| \int_t^{t_i}\partial_sP_{N_i} (v,H)(s)\ \mathrm{d}s\Big\|_{L^\infty_x(\R^3)}\\
&\quad\leq(t_i-t)\|\partial_tP_{N_i} (v,H)\|_{L^\infty_tL^\infty_x([0,1]\times\R^3)}\\
&\quad\leq A_3^{-2}A^2N_i.
\end{align*}
This, together with  \eqref{iter-1}, yields
$$\| P_{N_i} (v,H)(t) \|_{L^\infty_x(\R^3)}\gtrsim A_1^{-1} N_i. 
$$
Besides, 
applying \eqref{bts} and \eqref{iter-1}, we conclude that
\begin{align*}
 \sum_{i=0}^{m-1}{\int^{t_i}_{t_i-A_3^{-1} N_i^{-2}} A_1^{-1} N_i}\ \mathrm{d}t
 &\lesssim\sum_{i=0}^{m-1}{\int^{t_i}_{t_i-A_3^{-1}N_i^{-2}}
\|P_{N_i}(v,H)\|_{L^\infty_x(\R^3)}\ \mathrm{d}t}\\
  &\lesssim\|(v,H)\|_{L^1_t L^\infty_x([1-T_1,1]\times\R^3)}\\
&\lesssim A^4 T_1^{\frac{1}{2}},
\end{align*}
and thus
\EQS{\label{aim-2-2}
\sum_{i=0}^{m-1} N_i^{-1} \lesssim A_3^{2} T_1^{\frac{1}{2}}.
}
Here we have used the fact
$$
\bigcup\limits_{i=1}^{m-1}[t_i - A_3^{-1} N_i^{-2}, t_i]\subset [1-T_1,1],$$
and$$
(t_i - A_3^{-1} N_i^{-2}, t_i)\cap (t_j - A_3^{-1} N_j^{-2}, t_j)=\varnothing \quad \mbox{for}\quad 0\leq i\neq j\leq m-1.
$$
If $i\leq m-1$, then we immediately derive  \eqref{aim-2-0} from \eqref{aim-2-1} and \eqref{aim-2-2}. Now, if $i=m$, we notice from \eqref{aim-1-1} that  $N_i\lec  A_3T_1^{\frac{1}{2}}$, then we can extend the sum
\eqref{aim-2-2} to the final index $m$. Thus, \eqref{aim-2-0} is still valid. This completes the proof.
\end{proof}

\subsection{The proof of Proposition \ref{main-est}}\label{pf-N0}
\begin{proof}[The proof of Proposition \ref{main-est}]
Thanks to translation invariance, we assume $(t_0,x_0)=(0,0)$.
By Corollary \ref{v}, we can know that for any $A_4 N_0^{-2} \leq T_1 \leq A_4^{-1}$, there exists
\begin{equation*}
 (t_1,x_1)\in [-T_1, -A_3^{-3} T_1] \times B(0, A_4^{2} T_1^{\frac12}),
\end{equation*}
and
\begin{equation}\label{nando}
A_3^{-1} T_1^{-\frac12}\leq N_1 \leq A_3^{\frac12} T_1^{-\frac12},
\end{equation}
such that
$$ |P_{N_1} (v, H)(t_1,x_1)| \geq A_1^{-1} N_1.$$
The rest of the proof is divided into five steps.\\
{\bf Step 1}.\quad{\em  Transfer of concentration in Fourier space to physical space.} The purpose of this step is to prove the following estimate:
\begin{align}\label{stip}
\int_{B(0, A_4^{4} T_1^{\frac{1}{2}})} |\omega(t,x)|^2+|\nabla H(t,x)|^2+T_1^{-1}|H(t,x)|^2\ \mathrm{d}x \geq C_\ast A_4^{\frac{17}{2}} T_1^{-\frac{1}{2}}.
\end{align}
for all $t\in I''$. Here, $I''\subset\subset [t_1, t_1 + A_1^{-2} N_1^{-2}] \cap [-T_1, -A_3^{-3} T_1]$ is defined below.
From the Biot-Savart law
$$ P_{N_1} v(t_1,x_1) = (-\Delta)^{-1} P_{N_1} \nabla \times \widetilde{P}_{N_1} \omega(t_1,x_1),$$
and hence by \eqref{local}, \eqref{bern}, \eqref{able} we have
\begin{align*}
A_1^{-1} N_1
\leq&\|P_{N_1} (v,H)(t_1,x)\|_{L^\infty(B(x_1,\frac{A_1}{2N_1}))}\nonumber\\
 \leq& C_\ast N_1^{-1} \| \widetilde{P}_{N_1} (\omega, \nabla H)(t_1) \|_{L^\infty(B(x_1,\frac{A_1}{N_1}))} + C_\ast A_1^{-50} N_1^{-1} \| \widetilde{P}_{N_1}  (\omega, \nabla H)(t_1) \|_{L^\infty(\R^3)}\nonumber\\
\leq& C_\ast N_1^{-1} \| \widetilde{P}_{N_1} (\omega, \nabla H)(t_1)\|_{L^\infty(B(x_1,\frac{A_1}{N_1}))} + C_\ast AA_1^{-50} N_1 \nonumber\\
 \leq& C_\ast N_1^{-1} \| \widetilde{P}_{N_1} (\omega, \nabla H)(t_1) \|_{L^\infty(B(x_1,\frac{A_1}{N_1}))} + \frac12 A_1^{-1} N_1.
\end{align*}
Here and in the follow, $C_\ast>0$ is a  constant and independent of $A_1, N_1$ , which may vary from line to line. Similarly, we also have
\begin{align*}
A_1^{-1} N_1\leq&\|P_{N_1} \widetilde{P}_{N_1} H(t_1,x)\|_{L^\infty(B(x_1,\frac{A_1}{2N_1}))}\nonumber\\
 \leq& C_\ast\| \widetilde{P}_{N_1} H(t_1) \|_{L^\infty(B(x_1,\frac{A_1}{N_1}))} + C_\ast A_1^{-50} AN_1 \nonumber\\
\leq& C_\ast\| \widetilde{P}_{N_1} H(t_1)\|_{L^\infty(B(x_1,\frac{A_1}{N_1}))} + \frac12A_1^{-1} N_1.
\end{align*}
Thus, for some $x'_1, x'_2\in B(x_1,\frac{A_1}{N_1}) \subset B(0, A_4^{\frac52} T_1^{\frac12} )$ one has
$$
|\widetilde{P}_{N_1} (\omega(t_1,x'_1),\nabla H(t_1,x'_1))| \geq C_\ast A_1^{-1} N_1^2; \quad  |\widetilde{P}_{N_1}H(t_1,x'_2)|\geq C_\ast A_1^{-1}N_1.
$$
Notices that by using \eqref{bern}, \eqref{able} and \eqref{dtPNvH}, it follows that for any $(t,x)\in [-1,1]\times \R^3$
$$
 |\nabla \widetilde{P}_{N_1} (\omega,\nabla H)| \leq C_\ast A N_1^3; \quad \quad |\nabla \widetilde{P}_{N_1}H| \leq C_\ast A N_1^2
 $$
 and
 $$
 |\partial_t \widetilde{P}_{N_1} (\omega,\nabla H)| \leq C_\ast A^2 N_1^4;\quad \quad |\partial_t \widetilde{P}_{N_1}H| \leq C_\ast A^2 N_1^3.
$$
Then, for any
$(t,x) \in [t_1, t_1 + A_1^{-2} N_1^{-2}] \times B( x'_1, A_1^{-2} N_1^{-1} )$,
 we have
\begin{align*}
 C_\ast A_1^{-1} N_1^2-|\widetilde{P}_{N_1} (\omega,\nabla H)(t,x)|&\leq
 |\widetilde{P}_{N_1} (\omega,\nabla H)(t_1,x'_1)|-|\widetilde{P}_{N_1} (\omega,\nabla H)(t,x)|\nonumber\\
 &\leq|\nabla\widetilde{P}_{N_1} (\omega,\nabla H)||x-x'_1|+|\partial_t\widetilde{P}_{N_1} (\omega,\nabla H)||t-t_1|\nonumber\\
 &\leq C_\ast AN_1^3A_1^{-2}N_1^{-1}+C_\ast A^2N_1^4A_1^{-2}N_1^{-2}\nonumber\\
 &\leq \frac{C_\ast }{2}A_1^{-1}N_1^{2}.
\end{align*}
This implies that on $[t_1, t_1 + A_1^{-2} N_1^{-2}] \times B( x'_1, A_1^{-2} N_1^{-1})$
\begin{equation*}
 |\widetilde{P}_{N_1} (\omega,\nabla H)| \geq\frac{C_\ast}{2}A_1^{-1} N_1^2.
\end{equation*}
From this and \eqref{nando}, one has for $t\in [t_1, t_1 + A_1^{-2} N_1^{-2}]$
\EQS{\label{pn1w}
 \int_{B(0, A_4^{3} T_1^{\frac{1}{2}})} |\widetilde{P}_{N_1} (\omega, \nabla H)(t,x)|^2\ \mathrm{d}x
\geq C_\ast A_4^{\frac{17}2}T_1^{-\frac{1}{2}},
}
due to $ B( x'_1, A_1^{-2} N_1^{-1})\subset  B(0, A_4^{3} T_1^{\frac12}).$
Similarly, one also has
\begin{equation*}
 |\widetilde{P}_{N_1} H(t,x)| \geq \frac{C_\ast}{2}A_1^{-1} N_1,
\end{equation*}
for any  $(t,x) \in [t_1, t_1 + A_1^{-2} N_1^{-2}] \times B( x'_2, A_1^{-2} N_1^{-1})$, and further derives with help of \eqref{nando}
\begin{align*}
 \int_{B(0, A_4^{3} T_1^{\frac{1}{2}})} |\widetilde{P}_{N_1} H(t,x)|^2\ \mathrm{d}x
 \geq C_\ast A_4^{\frac{17}2}T_1^{\frac{1}{2}},\quad t\in [t_1, t_1 + A_1^{-2} N_1^{-2}].
\end{align*}
On the other hand,  Proposition \ref{iii} implies that there is an interval
$$I' \subset I=[t_1, t_1 + A_1^{-2} N_1^{-2}] \cap [-T_1, -A_3^{-3} T_1]\subset[-1,0],$$
with $|I'| = A^{-8}|I|$ such that for every $(t,x)\in I' \times \R^3$
\begin{equation}\label{omega-bound}
\begin{aligned}
&\|\nabla^j(v,H)(t,x) \|_{L^\infty_tL^\infty_x(I' \times \R^3)}\leq C_{\ast} A^{35} |I|^{-\frac{j+1}{2}}  \quad\text{for}\quad j=0,1,2.
 \end{aligned}
\end{equation}
Now we take
$$|I''|=\frac{1}{64C^2_{\ast}A^{70}}|I'|,$$
and derive by \eqref{omega-bound}
\begin{equation}\label{hawt}
\begin{aligned}
\|\nabla^j(v,H)(t,x) \|_{L^\infty_tL^\infty_x(I'' \times \R^3)}&\leq C_{\ast} A^{35-4(j+1)} |I'|^{-\frac{j+1}{2}}\\
&\quad =C_{\ast} A^{35-4(j+1)} (64C_{\ast}^2A^{70})^{-\frac{j+1}{2}}|I''|^{-\frac{j+1}{2}}\\
&\leq \frac{1}{8^{j+1}} |I''|^{-\frac{j+1}{2}}\quad\text{for}\quad j=0,1,2.
 \end{aligned}
\end{equation}
 Finally, by \eqref{local}, \eqref{pn1w} and \eqref{hawt} we have  for any $t \in I''$
\begin{align*}
C_\ast^{\frac12}A_4^{\frac{17}4} T_1^{-\frac{1}{4}}&\leq \|\widetilde{P}_{N_1} (\omega,\nabla H)(t,\cdot)\|_{L^2_x(B(0,A_4^{3}T_1^{\frac{1}{2}}))}
\nonumber\\
&\leq C_\ast\|(\omega,\nabla H)\|_{L^2_x(B(0,A_4^{4}T_1^{\frac{1}{2}}))}
+C_\ast A_4^{-50}\Big(A_4^{3}T_1^{\frac{1}{2}}\Big)^{\frac32}\|(\omega,\nabla H)(t,x) \|_{L^\infty_x(\R^3)}
\nonumber\\
&\leq C_\ast\|(\omega,\nabla H)\|_{L^2_x(B(0,A_4^{4}T_1^{\frac{1}{2}}))}
+C_\ast A_4^{-40}\Big(A_4^{3}T_1^{\frac{1}{2}}\Big)^{\frac32}T_1^{-1}\nonumber\\
&\leq C_\ast\|(\omega,\nabla H)\|_{L^2_x(B(0,A_4^{4}T_1^{\frac{1}{2}}))}
+\frac{C_\ast^{\frac12}}{2}A_4^{\frac{17}4} T_1^{-\frac{1}{4}},\nonumber\\
C_\ast^{\frac12}A_4^{\frac{17}4} T_1^{\frac{1}{4}}&\leq \|\widetilde{P}_{N_1} H\|_{L^2_x(B(0,A_4^{3}T_1^{\frac{1}{2}}))}
\nonumber\\
&\leq C_\ast\|H\|_{L^2_x(B(0,A_4^{4}T_1^{\frac{1}{2}}))}
+C_\ast A_4^{-50}\Big(A_4^{3}T_1^{\frac{1}{2}}\Big)^{\frac32}
\|H \|_{L^\infty_x(\R^3)}
\nonumber\\
&\leq C_\ast\|H\|_{L^2_x(B(0,A_4^{4}T_1^{\frac{1}{2}}))}
+C_\ast A_4^{-40}\Big(A_4^{3}T_1^{\frac{1}{2}}\Big)^{\frac32}T_1^{-\frac12}
\nonumber\\
&\leq C_\ast\|H\|_{L^2_x(B(0,A_4^{4}T_1^{\frac{1}{2}}))}
+\frac{C_\ast^{\frac12}}{2}A_4^{\frac{17}4} T_1^{\frac{1}{4}}.
\end{align*}
Combining the above calculations, we obtain the desired estimate \eqref{stip}. The following step use the Carleman inequality to transfer the concentration \eqref{stip} from the small scales $B(0, A_4^{4} T_1^{\frac{1}{2}})$ to large scales.

{\bf Step 2}.\quad {\em Large-scale propagation of concentration by using second Carleman inequality.}  The goal of this step to prove the following claim:
\begin{align}\label{step1main}
\int_{-T_1}^{-A_4^{-1}T_1} \int_{B(0,2R) \backslash B(0,\frac{R}{2})}  |\omega|^2+|\nabla H|^2+T_2^{-1}|H|^2\ \mathrm{d}y\mathrm{d}\tau \geq C_\ast A^{7}_4 T_1^{\frac12}e^{- \frac{A_5^{4}R^2}{T_1} }.
\end{align}
for all $A_4 N_0^{-2} \leq T_1 \leq A_4^{-1}, R\geq A_5 T_1^{\frac{1}{2}}$.
Denote by $I''\triangleq [t'_1-T_2, t'_1]$ for convenience, and introduce a new 15-component vector $W=(H, \omega, H_{x_1}, H_{x_2}, H_{x_3})$, where $H_{x_k}~(k=1,~2,~3)$ satisfy the  system
\begin{equation}\label{Hxkeq}
 \partial_tH_{x_k}-\Delta H_{x_k}=(H_{x_k}\cdot\nabla)v
 +(H\cdot\nabla)v_{x_k}-(v_{x_k}\cdot\nabla)H-(v\cdot\nabla)H_{x_k}.
\end{equation}
To obtain the desired result, we take the following scaling transformations
$$(v_\lambda, H_\lambda)(t,x)=\lambda (v,H)(t'_1-\lambda^2t,x_*+\lambda x),
\quad  (\omega_\lambda,J_\lambda) (t,x)=\lambda^2 (\omega,J)(t'_1-\lambda^2t,x_*+\lambda x),$$
$$((v_\lambda)_{x_k},(H_\lambda)_{x_k}) (t,x)=\lambda^2 (v_{(x_*+\lambda x)_k}, H_{(x_*+\lambda x)_k})(t'_1-\lambda^2t,x_*+\lambda x),$$
with $\lambda=\sqrt{T_2}$.
It is clear that $(v_\lambda, H_\lambda)$ is also a solution of system $\eqref{mhdeq}$ in $[0,1]\times \R^3$, and, by $\eqref{mhdeq}_2$, $\eqref{wJeq}_1$, \eqref{Hxkeq} and \eqref{hawt}, fulfils
\begin{align*}
|\partial_t H_\lambda+\Delta H_\lambda|\leq &|\nabla H_\lambda||v_\lambda|+||H_\lambda|\nabla v_\lambda|\leq\frac{1}{8}|\nabla H_\lambda|+\frac{1}{64} |H_\lambda|,\nonumber\\
|\partial_t \omega_\lambda+\Delta \omega_\lambda|\leq &|\nabla \omega_\lambda||v_\lambda|+|\omega_\lambda||\nabla v_\lambda|+|\nabla J_\lambda||H_\lambda|+|J_\lambda||\nabla H_\lambda|\nonumber\\
\leq& \frac{1}{8}|\nabla \omega_\lambda|+\frac{1}{64} |\omega_\lambda|+\frac{1}{512}|H_\lambda|+\frac{1}{64}|\nabla H_\lambda|,
\nonumber\\
|\partial_t (H_\lambda)_{x_k}+\Delta (H_\lambda)_{x_k}|\leq&|\nabla v_\lambda||(H_\lambda)_{x_k}|+|H_\lambda||\nabla (v_\lambda)_{x_k}|+|\nabla H_\lambda||(v_\lambda)_{x_k}|+|v_\lambda||\nabla (H_\lambda)_{x_k}|\nonumber\\
\leq& \frac{1}{64}|(H_\lambda)_{x_k}|+\frac{1}{512}|H_\lambda|+\frac{1}{64}|\nabla H_\lambda|+\frac{1}{64}|\nabla (H_\lambda)_{x_k}|,
\end{align*}
i.e.,
\begin{align*}
|\partial_t W_\lambda+\Delta W_\lambda|\leq \frac{1}{4} |W_\lambda|+\frac{1}{2}|\nabla W_\lambda|\quad\text{on}\quad {[0,1] \times \R^3}.
\end{align*}
with $W_\lambda( t, x )\triangleq(H_\lambda, \omega_\lambda, (H_\lambda)_{x_1}, (H_\lambda)_{x_2}, (H_\lambda)_{x_3})$. Therefore  \eqref{weifen-3} is satisfied with $C_{carl}=4$.
 We now apply Lemma \ref{carl-second} on the slab $[0,1] \times B(0,r)$ with $r \triangleq \frac{A_5 |x_*|}{\sqrt{T_2}}, |x_*| \geq A_5 T_1^{\frac{1}{2}}, t_0 \triangleq \frac{1}{20000}$, and $t_1 \triangleq A^{-4}_5 $,
to concluded that
\begin{align}\label{1-carleman}
&\int_{\frac{1}{20000}}^{\frac{1}{10000}} \int_{|x|\leq \frac{r}{2}}\left(|W_\lambda(t, x)|^2+|\nabla W_\lambda(t, x)|^2\right)e^{-\frac{|x|^2}{4t}}\ \mathrm{d}x \mathrm{d}t\nonumber\\
&\quad\leq C_{\ast} e^{-\frac{40A_5|x_*|^2}{T_2}}\int_0^{1} \int_{|x|\leq r} \left(|W_\lambda( t, x)|^2
+|\nabla W_\lambda( t, x)|^2\right)\ \mathrm{d}x \mathrm{d}t\nonumber\\
&\quad\quad+C_{\ast}A_5^{6} \left(\frac{A_5^4e}{20000}\right)^{\frac{20000C_{\ast}A_5^2|x_*|^2}{T_2}}\int_{|x|\leq r} |W_\lambda(0, x)|^2e^{-\frac{A_5^4|x|^2}{4}}
\ \mathrm{d}x.
\end{align}
Let $y=x_*+\lambda x$ and $\tau=t'_1-\lambda^2 t$, then \eqref{1-carleman} can be rewritten as
\begin{align}\label{ZXY1}
&Z_1 \triangleq \int_{t'_1-\frac{T_2}{10000}}^{t'_1-\frac{T_2}{20000}} \int_{B(x_*, \frac{A_5 |x_*|}{2})} \Big[(T_2)^{-\frac32} |H(\tau,y)|^2
+(T_2)^{-\frac12} |\omega(\tau,y)|^2\nonumber\\
&\quad\quad+(T_2)^{-\frac12} |H_{y_k}(\tau,y)|^2\Big]
e^{-\frac{|y-x_*|^2}{4(t'_1-\tau)}}\ \mathrm{d}y \mathrm{d}\tau\nonumber\\
&\quad\leq C_\ast e^{-\frac{40A_5^2|x_*|^2}{T_2}}X_1+C_\ast A_5^{6} \left(\frac{A_5^4e}{20000}\right)^{\frac{20000C_{\ast}A_5^2|x_*|^2}{T_2}}Y_1\nonumber\\
&\quad\leq C_\ast e^{-\frac{40A_5^2|x_*|^2}{T_2}}X_1+C_\ast e^{\frac{A_5^{\frac52}|x_*|^2}{T_2}}Y_1,
\end{align}
where
\begin{align*}
&X_1\triangleq \int_{t'_1-T_2}^{t'_1} \int_{B(x_*, A_5 |x_*|)}
\Big[T_2^{-\frac12}\big(T_2^{-1} |H(\tau,y)|^2+|\omega( \tau,y)|^2+|H_{y_k}(\tau,y)|^2\big)\\
&\quad\quad\quad+(T_2)^{-\frac12}|\nabla H(\tau,y)|^2
+(T_2)^{\frac12}|\nabla \omega(\tau,y)|^2+(T_2)^{\frac12}| \nabla H_{y_k}(\tau,y)|^2\Big]\ \mathrm{d}y \mathrm{d}\tau\nonumber\\
&Y_1\triangleq\int_{B(x_*, A_5 |x_*|)}
\Big[(T_2)^{-\frac12}|H(t'_1,y)|^2+(T_2)^{\frac12}|\omega(t'_1,y)|^2+(T_2)^{\frac12}|H_{y_k}(t'_1 ,y)|^2\Big]e^{-\frac{A_5^4|y-x_*|^2}{4T_2}}\ \mathrm{d}y.
\end{align*}
First, from \eqref{stip}
\begin{align*}
Z_1 &\geq (T_2)^{-\frac12}\int_{t'_1-\frac{T_2}{10000}}^{t'_1-\frac{T_2}{20000}} \int_{B(x_*, 2|x_*|)}\Big[|\omega(\tau,y)|^2+| \nabla H(\tau,y)|^2+T_2^{-1}|H(\tau,y)|^2\Big] e^{-\frac{|y-x_*|^2}{4(t'_1-\tau)}}\ \mathrm{d}y \mathrm{d}\tau\nonumber\\
&\geq (T_2)^{-\frac12}e^{-\frac{10000|x_*|^2}{T_2}}
\int_{t'_1-\frac{T_2}{10000}}^{t'_1-\frac{T_2}{20000}}
\int_{B(0, A_4^{4}T_1^{\frac{1}{2}})}  \Big[|\omega(\tau,y)|^2+| \nabla H(\tau,y)|^2+T_2^{-1}|H(\tau,y)|^2\Big] \ \mathrm{d}y \nonumber\\
&\geq T_2^{-\frac{1}{2}}\cdot\frac{T_2}{20000}\cdot C_\ast A_4^{\frac{17}2}T_1^{-\frac{1}{2}}e^{- \frac{10000|x_*|^2}{T_2} }
\geq C_\ast A_4^{8}e^{- \frac{10000|x_*|^2}{T_2} },
\end{align*}
secondly, by \eqref{hawt}
\begin{align*}
C_\ast e^{-\frac{40A_5^2|x_*|^2}{T_2}}X_1 &\leq C_\ast e^{-\frac{40|A_5x_*|^2}{T_2}}\left(\frac{|A_5x_*|}{\sqrt{T_2}}\right)^3 \leq\frac{C_\ast A_4^{8}}{2}e^{-\frac{10000|x_*|^2}{T_2}}.
\end{align*}
Therefore, by \eqref{ZXY1} and \eqref{hawt}, we conclude that
$ Y_1 \geq C_\ast A_4^{8}e^{- \frac{A_5^3|x_*|^2}{T_2} },$
 and then
\begin{align*}
&C_\ast A_4^{8}e^{- \frac{A_5^3|x_*|^2}{T_2} }\\
&\quad\leq
\int_{B(x_*,A_5|x_*|)} |\widetilde{W}(t'_1,y)|e^{-\frac{A_5^{4} |y-x_*|^2 } {4 T_2}}\ \mathrm{d}y\nonumber\\
&\quad\leq\int_{B(x_*,\frac{|x_*|}{2})} |\widetilde{W}(t'_1,y)|\ \mathrm{d}y
+\int_{B(x_*,A_5|x_*|) \backslash B(x_*,\frac{|x_*|}{2})} |\widetilde{W}(t'_1,y)|e^{-\frac{A_5^{4} |y-x_*|^2 } {4 T_2}}\ \mathrm{d}y\nonumber\\
&\quad\leq \int_{B(x_*,\frac{|x_*|}{2})} |\widetilde{W}(t'_1,y)|\ \mathrm{d}y +C_\ast \Big(\frac{A_5^{2}|x_*|^2}{T_2}\Big)^{\frac32}e^{- \frac{A_5^{4} |x_*|^2}{16T_2} }\nonumber\\
&\quad\leq T_2^{\frac{1}{2}}\int_{B(x_*,\frac{|x_*|}{2})} |\omega(t'_1,y)|^2+|\nabla H(t'_1,y)|^2+T_2^{-1}|H(t'_1,y)|^2|\ \mathrm{d}y +\frac{1}{2}C_\ast A_4^{8}e^{- \frac{A_5^3|x_*|^2}{T_2} },
\end{align*}
with
\begin{align*}
&|\widetilde{W}(t'_1,y)|\triangleq
T_2^{\frac12}(|\omega(t'_1,y)|^2+|\nabla H(t'_1,y)|^2+T_2^{-1}|H(t'_1,y)|^2).
\end{align*}
Thus,
\EQS{\label{2R-R-0}
\int_{B(x_*,\frac{|x_*|}{2})}|\omega(t'_1,y)|^2+|\nabla H(t'_1,y)|^2+T_2^{-1}|H(t'_1,y)|^2\ \mathrm{d}y\geq \frac{C_\ast}{2} A_4^{8}e^{- \frac{A_5^3|x_*|^2}{T_2} }T_2^{-\frac{1}{2}}.
}
Now for any $\tau\in [t'_1-\frac{T_2}{8},t'_1]$, repeating the above procedure verbatim with $t'_1, I''$ replaced by $\tau,  [t'_1-T_2, \tau]$, respectively, we can also derive the estimate \eqref{2R-R-0}. Denote by $|x_*|=R$,  one has for any $\tau\in[t'_1-\frac{T_2}{8},t'_1]$
\begin{align}\label{2R-R}
 \int_{B(0,2R) \backslash B(0,\frac{R}{2})} |\omega(\tau,y)|^2+|\nabla H(\tau,y)|^2+T_2^{-1}|H(\tau,y)|^2\ \mathrm{d}y\geq\frac{C_{\ast}}{2}A_4^{8}T_2^{-\frac12}e^{- \frac{A_5^3|x_*|^2}{T_2}}
\end{align}
due to
$$
B(x_*,\frac{|x_*|}{2})\subset B(0,2R) \backslash B(0,\frac{R}{2}).
$$
On the other hand, since $[t'_1-T_2,t'_1]\subset[-T_1,-A_3^{-3}T_1]$, we have
$$-T_1\leq-T_1+T_2\leq t'_1\leq -A_3^{-3}T_1\leq-A_4^{-1}T_1.$$
Then, by integrating \eqref{2R-R} respect to $\tau$ on $[-T_1,-A_4^{-1}T_1]$, we finally conclude \eqref{step1main}.

\medskip
{\bf Step 3}.\quad {\em Forward propagation of concentration via first Carleman inequality.} The goal of this step and Step 4 below is to prove the following estimate
\EQS{\label{W0x2}
 \int_{5 \widetilde{R} \leq |x| \leq \frac{3A_6 \widetilde{R}}{10}} T_3^{-1}|H(0,\theta)|^2
+|\omega(0,\theta)|^2+|H_{\theta_k}(0,\theta)|^2\ \mathrm{d}\theta \geq C_{\ast}e^{-e^{A_6^{10}}} T_3^{-\frac{1}{2}}
}
for all
\begin{equation*}
 A_4^2 N_0^{-2} \leq T_3 \leq A_4^{-1}.
\end{equation*}
Notices that by Proposition \ref{vi}, there exist absolute constants $\varepsilon'_\ast=\frac1{4A_6^6}$, $R_0=\frac{A_6}{16}$ and $\mu'=\frac{\log(4A_6^6)}{\log A}$
such that on the cylindrical annulus
$$ \Omega := \Big\{ (t,x) \in [-T_3,0] \times \R^3: \widetilde{R}\leq |x|\leq A_6^6 \widetilde{R} \Big\},$$
one has the estimates
\begin{equation}\label{annv}
\begin{aligned}
\|\nabla^j(v,H)(t,x) \|_{L^\infty_tL^\infty_x(\Omega)}\leq C_{\ast} A_6^{-2} T_3^{-\frac{j+1}{2}}\quad\text{for}\quad j=0,1,2,
 \end{aligned}
\end{equation}
with
\begin{equation}\label{rb0}
 A_6 T_3^{\frac{1}{2}} \leq \widetilde{R} \leq e^{A_6^{7}} T_3^{\frac{1}{2}}.
\end{equation}
As Step 1, we take a transformation
\EQs{
&(v_\mu,H_\mu) (t,x)=\mu (v,H)(-\mu^2t,\mu x), \quad(\omega_\mu,J_\mu) (t,x)=\mu^2 (\omega,J)(-\mu^2t,\mu x),\\&
((v_\mu)_{x_k}, (H_\mu)_{x_k} )(t,x)=\mu^2 (v_{\mu x_k},H_{\mu x_k})(-\mu^2t,\mu x),\quad \mu=\sqrt{T_3}.
}
It is clear that $v_\mu,H_\mu$ is a solution of system $\eqref{mhdeq}$ in $[0,1]\times \R^3$ such that by \eqref{annv}
\EQs{
|\partial_t H_\mu+\Delta H_\mu|\leq |\nabla H_\mu||v_\mu|+||H_\mu|\nabla v_\mu|
\leq C_{\ast}A_6^{-2}(|\nabla H_\mu|+|H_\mu|);}
\EQs{
|\partial_t \omega_\mu+\Delta \omega_\mu|&\leq |\nabla \omega_\mu||v_\mu|+|\omega_\mu||\nabla v_\mu|+|\nabla J_\mu||H_\mu|+|J_\mu||\nabla H_\mu|
\\&\leq C_{\ast}A_6^{-2}(|\nabla \omega_\mu|+ |\omega_\mu|+|H_\mu|+|\nabla H_\mu|);
}
and
\EQs{
|\partial_t (H_\mu)_{x_k}+\Delta (H_\mu)_{x_k}|&\leq|\nabla v_\mu||(H_\mu)_{x_k}|+|H_\mu||\nabla (v_\mu)_{x_k}|+|\nabla H_\mu||(v_\mu)_{x_k}|+|v_\mu||\nabla (H_\mu)_{x_k}|
\\&\leq C_{\ast}A_6^{-2}( |(H_\mu)_{x_k}|+|H_\mu|+|\nabla H_\mu|+|\nabla (H_\mu)_{x_k}|),
}
for $(t,x)\in [0,1]\times \Big(B(0,\frac{A_6 \widetilde{R}}{10\sqrt{T_3}}) \backslash B(0,\frac{10 \widetilde{R}}{\sqrt{T_3}})\Big)$. That is
\begin{align*}
|\partial_tW_\mu+\Delta W_\mu|\leq C_{\ast}A_6^{-1}|\nabla W_\mu|+C_{\ast}^2A_6^{-2}|W_\mu|
\end{align*}
with
$$
W_\mu( t, x )\triangleq(H_\mu, \omega_\mu, (H_\mu)_{x_1}, (H_\mu)_{x_2}, (H_\mu)_{x_3}).
$$
We now apply Lemma \ref{carl-first} on the slab $[0,\frac1{A_2}] \times (B(0,r_+) \backslash B(0,r_-))$ with
$C_{carl}=A_2$, $r_- \triangleq \frac{10 \widetilde{R}}{\sqrt{T_3}}$, $r_+ \triangleq\frac{A_6 \widetilde{R}}{10\sqrt{T_3}}$, to
 conclude that
\begin{align}\label{1-carleman-1}
&\int_{0}^{\frac{1}{4A_2}} \int_{\frac{100 \widetilde{R}}{\sqrt{T_3}}\leq|x|\leq \frac{ A_6 \widetilde{R}}{20\sqrt{T_3}}}|W_\mu(t, x)|^2+|\nabla W_\mu(t, x)|^2\ \mathrm{d}x \mathrm{d}t\nonumber\\
&\quad\leq C_{\ast}A_2^2 e^{-\frac{A_6\widetilde{R}^2}{4T_3}}\Big(\int_0^{\frac1{A_2}} \int_{\frac{10\widetilde{R}}{\sqrt{T_3}}\leq|x|\leq\frac{A_6 \widetilde{R}}{10\sqrt{T_3}} }e^{2|x|^2} \left(A_2|W_\mu( t, x)|^2
+|\nabla W_\mu( t, x)|^2\right)\ \mathrm{d}x\mathrm{d}t\nonumber\\
&\quad\quad+e^{\frac{A_6^{2}\widetilde{R}^2}{50T_3}} \int_{\frac{10 \widetilde{R}}{\sqrt{T_3}}\leq|x|\leq\frac{A_6 \widetilde{R}}{10\sqrt{T_3}} } |W_\mu(0, x)|^2\ \mathrm{d}x\Big).
\end{align}
Let $z=\mu x$ and $s=-\mu^2 t$, then \eqref{1-carleman-1} can be rewritten as
\begin{align}\label{Z1X1Y1}
&Z_2 \triangleq \int_{-\frac{T_3}{4A_2}}^{0} \int_{100\widetilde{R}\leq|z|\leq\frac{A_6\widetilde{R}}{20}} T_3^{-1} |H(s,z)|^2
+ |\omega(s,z)|^2+ |H_{z_k}(s,z)|^2
\ \mathrm{d}z \mathrm{d}s\nonumber\\
&\quad\leq C_{\ast}A_2^{3} e^{-\frac{A_6\widetilde{R}^2}{4T_3}}X_2+C_{\ast}e^{e^{A_6^{9}}}T_3Y_2,
\end{align}
where
\begin{align*}
&X_2\triangleq \int_{-\frac{T_3}{A_2}}^{0} \int_{10\widetilde{R}\leq|z|\leq \frac{A_6\widetilde{R}}{10}}
e^{\frac{2|z|^2}{T_3}} \Big(T_3^{-1} |H(s, z)|^2+|\omega(s, z)|^2\nonumber\\
&\quad\quad+|H_{z_k}( s,z)|^2+|\nabla H(s, z)|^2
+T_3|\nabla \omega(s, z)|^2+T_3| \nabla H_{z_k}(s, z)|^2\Big)\ \mathrm{d}z \mathrm{d}s,\nonumber\\
&Y_2\triangleq\int_{10\widetilde{R}\leq|z|\leq \frac{A_6\widetilde{R}}{10}}
T_3^{-1}|H(0,z)|^2+|\omega(0,z)|^2+|H_{z_k}(0,z)|^2\ \mathrm{d}z.
\end{align*}
From \eqref{step1main} with $R=200\widetilde{R}$ and $\widetilde{T}_1=\frac{T_3}{4A_2}$, we have
 \begin{align*}
Z_2 &\geq \frac{1}{4} A_2^{-1}\int_{-\widetilde{T}_1}^{0} \int_{\frac{R}{2} \leq |z| \leq \frac{A_6R}{4000}}\widetilde{T}_1^{-1} |H(s,z)|^2+ |\omega(s,z)|^2+ |H_{z_k}(s,z)|^2
\ \mathrm{d}z \mathrm{d}s\nonumber\\
&\geq\frac14A_2^{-1}
 \int_{-\widetilde{T}_1}^{-A_4^{-1}\widetilde{T}_1} \int_{B(0,2R) \backslash B(0,\frac{R}{2})} \widetilde{T}_1^{-1} |H(s,z)|^2+ |\omega(s,z)|^2+ |H_{z_k}(s,z)|^2
\ \mathrm{d}z \mathrm{d}s\nonumber\\
&\geq \frac{C_\ast}{4}A_2^{-1}A_4^7\widetilde{T}_1^{\frac12}e^{- \frac{A_5^{4}R^2}{\widetilde{T}_1} }\geq C_\ast T_3^{\frac12}e^{- \frac{A_6^{\frac12}\widetilde{R}^2}{T_3} }.
 \end{align*}
This together with \eqref{Z1X1Y1} yields
$$
C_{\ast}T_3^{\frac12}e^{- \frac{A_6^{\frac12}\widetilde{R}^2}{T_3} }\leq
C_{\ast}A_2^{3} e^{-\frac{A_6\widetilde{R}^2}{4T_3}}X_2+C_{\ast}e^{e^{A_6^{9}}}T_3Y_2.
$$
Thus we either have
\begin{equation}\label{option-1}
X_2\geq C_{\ast} e^{\frac{A_6^{\frac12} \widetilde{R}^2}{T_3}} T_3^{\frac{1}{2}},
\end{equation}
or
\begin{equation}\label{option-2}
Y_2 \geq C_{\ast}e^{-e^{A_6^{10}}} T_3^{-\frac{1}{2}}.
\end{equation}
It is clear that \eqref{option-2} implies our desired estimate \eqref{W0x2}. Therefore, we only consider the case that the bounded \eqref{option-1} holds which can be rewritten as
\begin{align*}
C_{\ast} e^{\frac{A_6^{\frac12} \widetilde{R}^2}{T_3}} T_3^{\frac{1}{2}}
&\leq  \sum_{k=0}^{\lceil\mathrm{log}_2\frac{A_6}{200}\rceil}\int_{-\frac{T_3}{A_2}}^0 \int_{(10 \widetilde{R})\cdot2^k \leq |z| \leq (10\widetilde{R})\cdot2^{k+1}} e^{\frac{2|z|^2}{T_3}} \Big(T_3^{-1} |H(s, z)|^2+|\omega(s, z)|^2\nonumber\\
&\quad\quad+|H_{z_k}( s,z)|^2+|\nabla H(s, z)|^2
+T_3|\nabla \omega(s, z)|^2+T_3| \nabla H_{z_k}(s, z)|^2\Big)\ \mathrm{d}z \mathrm{d}s,
\end{align*}
From the pigeonhole principle, there exists $k_0\in\{0,1,\ldots,\lceil\mathrm{log}_2\frac{A_6}{200}\rceil\}$
and $10\widetilde{R}\leq \widetilde{R}_1=10\widetilde{R}\cdot2^{k_0}\leq \frac{A_6 \widetilde{R}}{10}$ such that
\begin{align*}
 &\int_{-\frac{T_3}{A_2}}^0 \int_{\widetilde{R}_1 \leq |z| \leq 2\widetilde{R}_1} e^{\frac{2|z|^2}{T_3}} \Big(T_3^{-1} |H(s, z)|^2+|\omega(s, z)|^2\nonumber\\
&\quad+|H_{z_k}( s,z)|^2+|\nabla H(s, z)|^2
+T_3|\nabla \omega(s, z)|^2+T_3| \nabla H_{z_k}(s, z)|^2\Big)\ \mathrm{d}z \mathrm{d}s,\nonumber\\
&\geq\frac{C_{\ast}}{\lceil\mathrm{log}_2\frac{A_6}{200}\rceil}
T_3^{\frac{1}{2}} e^{\frac{A_6^{\frac{1}{2}} \widetilde{R}^2}{T_3}} \geq C_{\ast}
T_3^{\frac{1}{2}} e^{\frac{A_6^{-\frac{3}{2}} \widetilde{R}_1^2}{T_3}} ,
 \end{align*}
due to $\widetilde{R}\geq\frac{10\widetilde{R}_1}{A_6}$. From this, we immediately derive
 \begin{align}\label{R2R}
 &\int_{-\frac{T_3}{A_2}}^0 \int_{\widetilde{R}_1 \leq |x| \leq 2\widetilde{R}_1}  \Big(T_3^{-1} |H(s, z)|^2+|\omega(s, z)|^2+|H_{z_k}( s,z)|^2+|\nabla H(s, z)|^2\nonumber\\
&\quad
+T_3|\nabla \omega(s, z)|^2+T_3| \nabla H_{z_k}(s, z)|^2\Big)\ \mathrm{d}z \mathrm{d}s\geq C_{\ast}
T_3^{\frac{1}{2}} e^{-\frac{10\widetilde{R}_1^2}{T_3}},
 \end{align}
because of $e^{\frac{2|z|^2}{T_3}}\leq e^{\frac{8 \widetilde{R}_1^2}{T_3}}$. On the other hand, from \eqref{annv}
\begin{align*}
&\int_{-T_3e^{-\frac{20{\widetilde{R}_1}^2}{T_3}}}^0 \int_{\widetilde{R}_1 \leq |x| \leq 2\widetilde{R}_1} \Big(T_3^{-1} |H(s, z)|^2+|\omega(s, z)|^2+|H_{z_k}( s,z)|^2+|\nabla H(s, z)|^2\nonumber\\
&\quad
+T_3|\nabla \omega(s, z)|^2+T_3| \nabla H_{z_k}(s, z)|^2\Big)\ \mathrm{d}z \mathrm{d}s\nonumber\\
&\leq C_{\ast}A_6^{-2} \Big(\frac{{\widetilde{R}_1}^2}{T_3}\Big)^{\frac32}
{T_3^\frac12}e^{-\frac{20{\widetilde{R}_1}^2}{T_3}}\leq \frac12C_{\ast}
T_3^{\frac{1}{2}} e^{-\frac{10\widetilde{R}_1^2}{T_3}}.
\end{align*}
Hence, from \eqref{R2R}
\begin{align*}
&\frac12C_{\ast}
T_3^{\frac{1}{2}} e^{-\frac{10\widetilde{R}_1^2}{T_3}}\nonumber\\
 &\quad\leq
 \int_{-\frac{T_3}{A_2}}^{-T_3e^{-\frac{20{\widetilde{R}_1}^2}{T_3}}} \int_{\widetilde{R}_1 \leq |x| \leq 2\widetilde{R}_1} \Big(T_3^{-1} |H(s, z)|^2+|\omega(s, z)|^2+|H_{z_k}( s,z)|^2\nonumber\\
&\quad\quad+|\nabla H(s, z)|^2
+T_3|\nabla \omega(s, z)|^2+T_3| \nabla H_{z_k}(s, z)|^2\Big)\ \mathrm{d}z \mathrm{d}s\nonumber\\
 &\quad\leq \sum_{\lambda=0}^{\lceil\mathrm{log}_2({A_2}^{-1}e^{
 \frac{20{\widetilde{R}_1}^2}{T_3}})-1\rceil}
 \int_{-[T_3e^{-\frac{20{\widetilde{R}_1}^2}{T_3}}]\cdot2^{\lambda+1}}
 ^{-[T_3e^{-\frac{20{\widetilde{R}_1}^2}{T_3}}]\cdot2^\lambda} \int_{\widetilde{R}_1 \leq |x| \leq 2\widetilde{R}_1}\Big(T_3^{-1} |H(s, z)|^2\nonumber\\
&\quad\quad+|\omega(s, z)|^2+|H_{z_k}( s,z)|^2+|\nabla H(s, z)|^2
+T_3|\nabla \omega(s, z)|^2+T_3| \nabla H_{z_k}(s, z)|^2\Big)\ \mathrm{d}z \mathrm{d}s .
\end{align*}
by a further application of the pigeonhole principle,  there exists $\lambda_0\in\{0,1,\ldots,\lceil\mathrm{log}_2(A_2^{-1}
e^{\frac{20{\widetilde{R}_1}^2}{T_3}})-1\rceil\}$ and a locate time scale $t_3=e^{-\frac{20{\widetilde{R}_1}^2}{T_3}}T_3\cdot2^{\lambda_0}$ such that
\begin{equation}\label{tang}
e^{\frac{-20\widetilde{R}_1^2}{T_3}}T_3 \leq t_3 \leq \frac{T_3}{A_2},
\end{equation}
then
\begin{align*}
 &\int_{-2t_3}^{-t_3} \int_{\widetilde{R}_1 \leq |x| \leq 2\widetilde{R}_1} \Big(T_3^{-1} |H(s, z)|^2+|\omega(s, z)|^2+|H_{z_k}( s,z)|^2+|\nabla H(s, z)|^2\nonumber\\
&\quad
+T_3|\nabla \omega(s, z)|^2+T_3| \nabla H_{z_k}(s, z)|^2\Big)\ \mathrm{d}z \mathrm{d}s \geq C_{\ast}T_3^{\frac{1}{2}}e^{- \frac{15\widetilde{R}_1^2}{T_3}}.
\end{align*}
We cover the annulus $B(0,2\widetilde{R}_1) \backslash B(0,\widetilde{R}_1)$ with
$$C_{\ast}\frac{8\widetilde{R}_1^3-\widetilde{R}_1^3}{t_3^{\frac32}}\leq
7C_{\ast}\frac{\widetilde{R}_1^3}{T_3^{\frac32}}e^{\frac{30\widetilde{R}_1^2}{T_3}}
\leq
7C_{\ast}e^{\frac{31\widetilde{R}_1^2}{T_3}},$$
 balls of radius $t_3^{\frac{1}{2}}$ and apply the pigeonhole principle to find that there exists $x_3\in \{x~:~\widetilde{R}_1 \leq |x| \leq 2\widetilde{R}_1\}$ such that
 \begin{align}\label{cor}
 &\int_{-2t_3}^{-t_3} \int_{B(x_3, t_3^{\frac{1}{2}})} T_3^{-1} |H|^2+|\omega|^2+|\nabla H|^2+T_3\Big(|\nabla \omega|^2+| \nabla H_{z_k}|^2\Big)\ \mathrm{d}z \mathrm{d}s \geq C_{\ast}T_3^{\frac{1}{2}}e^{- \frac{46\widetilde{R}_1^2}{T_3}}.
\end{align}
In the Step 4 below, we continue to  apply the Carleman inequality to transfer the above low bound to the time 0.

{\bf Step 4}.\quad {\em The concentration continues to propagate on a large-scale
by using second Carleman inequality.}\quad
In this step, we will continue to use the second Carleman inequality to derive \eqref{W0x2} if the case \eqref{option-1} holds.
As Step 1, we take a transformation
$$(v_\nu, H_\nu) (t,x)=\nu (v,H)(-\nu^2t,x_3+\nu x),
\quad (\omega_\nu, J_\nu)(t,x)=\nu^2 (\omega,J)(-\nu^2t,x_3+\nu x),
$$
$$((v_\nu)_{x_k},(H_\nu)_{x_k}) (t,x)=\nu^2 (v_{(x_3+\nu x)_k},H_{(x_3+\nu x)_k})(-\nu^2t,x_3+\nu x),\quad\nu=\sqrt{20000t_3}.$$
It is clear that $v_\nu, H_\nu$ is a solution of system \ref{mhdeq} in $[0,1]\times \R^3$.
Notice that due to \eqref{tang} , one has
\begin{align*}
&[-\nu^2,0]=[-20000t_3,0]\subset [-\frac{20000T_3}{A_2},0]\subset [-T_3,0],\\
&\nu r=\frac{A_2^{\frac{1}{4}}\widetilde{R}_1\sqrt{20000t_3}}{\sqrt{T_3}}\leq
\frac{A_2^{\frac{1}{4}}\sqrt{20000}\widetilde{R}_1}{A_2^{\frac{1}{2}}}\leq \frac{\widetilde{R}_1}{2}\leq \frac{|x_3|}{2},\quad \text{with}\quad r \triangleq  \frac{A_2^{\frac{1}{4}}\widetilde{R}_1}{\sqrt{T_3}},
\\
&B(x_3,\nu r)\subset B(x_3,\frac{|x_3|}{2})\subset \Big\{\frac{\widetilde{R}_1}{2}\leq |y|\leq 3\widetilde{R}_1\Big\}\subset \Big\{5\widetilde{R}\leq |y|\leq \frac{3A_6\widetilde{R}}{10}\Big\}
\subset \Big\{\widetilde{R}\leq |y|\leq A_6^6\widetilde{R}\Big\},
\end{align*}
which implies
\begin{align*}
&\|\nabla^j(v_\nu,H_\nu)(t,x)\|_{L^\infty_tL^\infty_x([0,1]\times B(0,r))}\\
&\quad=\nu^{j+1}
\|\nabla^j(v,H)(s,y)\|_{L^\infty_sL^\infty_y(([-\nu^2,0]\times B(x_3,\nu r))}\\
&\quad\leq \nu^{j+1}
\|\nabla^j(v,H)(s,y)\|_{L^\infty_sL^\infty_y(\Omega)}\leq
 C_\ast A_6^{-2}A_2^{-\frac{j+1}{2}},\quad j=0,1,2.
\end{align*}
Thus, we have
\EQs{
|\partial_t H_\nu+\Delta H_\nu|\leq |\nabla H_\nu||v_\nu|+||H_\nu|\nabla v_\nu|
\leq C_{\ast}A_6^{-2}A_2^{-\frac12}(|\nabla H_\nu|+|H_\nu|);}
\EQs{
|\partial_t \omega_\nu+\Delta \omega_\nu|&\leq |\nabla \omega_\nu||v_\mu|+|\omega_\nu||\nabla v_\nu|+|\nabla J_\nu||H_\nu|+|J_\nu||\nabla H_\nu|
\\&\leq C_{\ast}A_6^{-2}A_2^{-\frac12}(|\nabla \omega_\nu|+ |\omega_\nu|+|H_\nu|+|\nabla H_\nu|);
}
and
\EQs{
|\partial_t (H_\nu)_{x_k}+\Delta (H_\nu)_{x_k}|&\leq|\nabla v_\nu||(H_\nu)_{x_k}|+|H_\nu||\nabla (v_\nu)_{x_k}|+|\nabla H_\nu||(v_\nu)_{x_k}|+|v_\nu||\nabla (H_\nu)_{x_k}|
\\&\leq C_{\ast}A_6^{-2}A_2^{-\frac12}( |(H_\nu)_{x_k}|+|H_\nu|+|\nabla H_\nu|+|\nabla (H_\nu)_{x_k}|),
}
for $(t,x)\in [0,1]\times B(0,r) $. That is
$$
|\partial_tW_\nu+\Delta W_\nu|\leq C_{\ast}A_6^{-1}A_2^{-\frac14}|\nabla W_\mu|+C_{\ast}^2A_6^{-2}A_2^{-\frac12}|W_\mu|,
$$
with
$$
W_\nu( t, x )\triangleq(H_\nu, \omega_\nu, (H_v)_{x_1}, (H_\nu)_{x_2}, (H_\nu)_{x_3}).
$$
We now apply Lemma \ref{carl-second} on the slab $[0,1] \times B(0,r)$ with
$C_{carl}=1$,  $t_0=t_1=\frac{1}{20000}$, to conclude that
\begin{equation*}\label{zpp}
\begin{aligned}
&\int_{\frac{1}{20000}}^{\frac{1}{10000}} \int_{|x|\leq\frac{r}{2}} \Big[ |W_\nu(t,x)|^2 + |\nabla W_\nu(t,x)|^2\Big] e^{-\frac{|x|^2}{4t}}\ \mathrm{d}x \mathrm{d}t\\
&\quad\leq C_{\ast} e^{-\frac{40A_2^{\frac12}\widetilde{R}_1^2}{T_3}}\int_{0}^{1} \int_{|x|\leq r} \Big[|W_\nu(t,x)|^2 + |\nabla W_\nu(t,x)|^2\Big] \ \mathrm{d}x \mathrm{d}t\\
&\quad\quad +C_{\ast}e^{\frac{A_2\widetilde{R}_1^2}{T_3}} \int_{|x|\leq r} |W_\nu(0,x)|^2 e^{-5000|x|^2}\ \mathrm{d}x.
\end{aligned}
\end{equation*}
Scaling back to the original variables leads to (i.e.,$\theta=x_3+\nu x$ and $\sigma=-\nu^2 t$)
\begin{align}\label{zpp}
&Z_3\triangleq\int_{-2t_3}^{-t_3} \int_{B(x_3, t_3^{\frac{1}{2}})} \Big(t_3^{-1} |H(\sigma,\theta)|^2
+ |\omega(\sigma,\theta)|^2+ |H_{\theta_k}(\sigma,\theta)|^2\nonumber\\
&\quad\quad+|\nabla H(\sigma,\theta)|^2
+ t_3|\nabla \omega(\sigma,\theta)|^2+t_3 |\nabla H_{\theta_k}(\sigma,\theta)|^2
\Big) e^{\frac{|\theta-x_3|^2}{4\sigma}}\ \mathrm{d}\theta \mathrm{d}\sigma\nonumber\\
&\quad\leq C_{\ast}\Big(e^{-\frac{40A_2^{\frac{1}{2}} \widetilde{R}_1^2}{ T_3}}X_3+t_3e^{\frac{A_2\widetilde{R}_1^2}{T_3}} Y_3\Big),
\end{align}
where we have used the fact due to  \eqref{rb0}
\begin{align*}
\frac{\nu r}{2}=\frac{\sqrt{20000t_3}A_2^{\frac{1}{4}}\widetilde{R}_1}{2\sqrt{T_3}}\geq \frac{10\sqrt{20000}A_2^{\frac{1}{4}}\widetilde{R}}{2\sqrt{T_3}}t_3^{\frac12}
\geq 5\sqrt{20000}A_6A_2^{\frac{1}{4}}t_3^{\frac12} \geq t_3^{\frac12}.
\end{align*}
Here
\begin{align*}
&X_3\triangleq \int_{-20000t_3}^{0} \int_{B(x_3,\frac{|x_3|}{2})}
t_3^{-1} |H(\sigma,\theta)|^2
+ |\omega(\sigma,\theta)|^2+ |H_{\theta_k}(\sigma,\theta)|^2\\
&\quad\quad+|\nabla H(\sigma,\theta)|^2
+ t_3|\nabla \omega(\sigma,\theta)|^2+t_3 |\nabla H_{\theta_k}(\sigma,\theta)|^2
\ \mathrm{d}\theta \mathrm{d}\sigma,\nonumber\\
&Y_3\triangleq\int_{B(x_3,\frac{|x_3|}{2})}
\Big(t_3^{-1}|H(0,\theta)|^2+|\omega(0,\theta)|^2+|H_{\theta_k}(0,\theta)|^2
\Big)e^{-\frac{|\theta-x_3|^2}{4t_3}}\ \mathrm{d}\theta.
\end{align*}
From \eqref{tang} and \eqref{cor},  one has
\begin{align}\label{Z-3est}
Z_3 &\geq \int_{-2t_3}^{-t_3} \int_{B(x_3, t_3^{\frac{1}{2}})} \Big(A_2T_3^{-1} |H(\sigma,\theta)|^2
+ |\omega(\sigma,\theta)|^2+ |H_{\theta_k}(\sigma,\theta)|^2\nonumber\\
&\quad\quad+|\nabla H(\sigma,\theta)|^2
+ T_3e^{-\frac{20\widetilde{R}_1^2}{T_3}}\Big(|\nabla \omega(\sigma,\theta)|^2+ |\nabla H_{\theta_k}(\sigma,\theta)|^2
\Big) e^{\frac{|\theta-x_3|^2}{4\sigma}}\ \mathrm{d}\theta \mathrm{d}\sigma\nonumber\\
& \geq e^{-\frac18}e^{-\frac{20\widetilde{R}_1^2}{T_3}}
 \int_{-2t_3}^{-t_3} \int_{B(x_3, t_3^{\frac{1}{2}})}  T_3^{-1} |H|^2+|\omega|^2+|\nabla H|^2+T_3\Big(|\nabla \omega|^2+| \nabla H_{z_k}|^2\Big)\ \mathrm{d}\theta \mathrm{d}\sigma\nonumber\\
& \geq C_{\ast}T_3^{\frac{1}{2}}e^{- \frac{66\widetilde{R}_1^2}{T_3}}.
 \end{align}
From \eqref{annv} and \eqref{tang} we have
$$ C_{\ast}e^{-\frac{40A_2^{\frac{1}{2}} \widetilde{R}_1^2}{T_3}}X_3\leq C_{\ast} T_3^{\frac12}\Big(\frac{|x_3|^2}{T_3}\Big)^{\frac{3}{2}} e^{-\frac{40A_2^{\frac{1}{2}}\widetilde{R}_1^2}{T_3}} e^{\frac{20\widetilde{R}_1^2}{T_3}}\leq \frac{C_{\ast}}{2}e^{- \frac{66\widetilde{R}_1^2}{T_3}}T_3^{\frac{1}{2}},
$$
which along with \eqref{zpp} and \eqref{Z-3est}, yields
$$ Y_3 \geq \frac{C_{\ast}}{2} t_3^{-1}T_3^{\frac{1}{2}}e^{- \frac{66\widetilde{R}_1^2}{T_3}} e^{-\frac{A_2\widetilde{R}_1^2}{T_3} } \geq C_{\ast}T_3^{-\frac{1}{2}} e^{-\frac{A_2^{\frac32}\widetilde{R}_1^2}{T_3} }.$$
Denote by
$$W_c(0,\theta)=T_3^{-1}|H(0,\theta)|^2
+|\omega(0,\theta)|^2+|H_{\theta_k}(0,\theta)|^2,$$
we have by using \eqref{tang}
\begin{align*}
e^{\frac{20\widetilde{R}_1^2}{T_3}}\int_{5 \widetilde{R} \leq |\theta| \leq \frac{3A_6 \widetilde{R}}{10}} W_c(0, \theta)\ \mathrm{d}\theta  \geq
Y_3 \geq C_{\ast}T_3^{-\frac{1}{2}} e^{-\frac{A_2^{\frac32}\widetilde{R}_1^2}{T_3} }.
 \end{align*}
This, together with $10A_6T_3^{1/2}\leq \widetilde{R}_1\leq \frac{e^{A_6}A_6T_3^{1/2}}{10}$, implies \eqref{W0x2}.
 \medskip

{\bf Step 5}.\quad {\em Conclusion: summing of scales to derive the upper bound for $N_0$.}
 First, we note that  the volume of the annulus $\{x:~5\widetilde{R} \leq |x| \leq \frac{3A_6 \widetilde{R}}{10}\}$ is bounded by $T_3^{\frac{3}{2}}e^{e^{A_6^8}}$ by \eqref{rb0}, which enables us find  a point $\tilde{x}\in (B(0,\frac{3A_6 \widetilde{R}}{10}) \backslash B(0,5\widetilde{R}))$ such that
  $$
  T_3^{-\frac12}|H(0,\tilde{x})|+|\omega(0,\tilde{x})|+|H_{\theta_k}(0,\tilde{x})|
  \gtrsim e^{-e^{A_6^{11}}} T_3^{-1}
  $$
 due to the pigeonhole principle and \eqref{W0x2}. This, together with \eqref{annv}, yields
\begin{align}
 T_3^{-\frac12}\Big|\int_{\R^3} H(0, \tilde{x} - \widetilde{r} y) \xi(y)\mathrm{d}y\Big|&\geq T_3^{-\frac12}\Big|\int_{\R^3} H(0, \tilde{x}) \xi(y)\mathrm{d}y\Big|\nonumber\\
 &\quad-T_3^{-\frac12}\Big|\int_{\R^3} (H(0, \tilde{x})-H(0, \tilde{x} - \widetilde{r} y)) \xi(y)\mathrm{d}y\Big|\nonumber\\
 &\gtrsim T_3^{-\frac12}|H(0,\tilde{x})|-e^{-e^{A_6^{11}}}A_6^{-2} T_3^{-1},\label{TH-eat}\\
\Big|\int_{\R^3} \omega(0, \tilde{x} - \widetilde{r} y) \eta(y)\mathrm{d}y\Big|&\geq \Big|\int_{\R^3} \omega(0, \tilde{x}) \eta(y)\mathrm{d}y\Big|\nonumber\\
 &\quad-\Big|\int_{\R^3} (\omega(0, \tilde{x})-\omega(0, \tilde{x} - \widetilde{r} y)) \eta(y)\mathrm{d}y\Big|\nonumber\\
 &\gtrsim |\omega(0,\tilde{x})|-e^{-e^{A_6^{11}}}A_6^{-2} T_3^{-1},\label{w-eat}\\
 \Big|\int_{\R^3} H_{\theta_k}(0, \tilde{x} - \widetilde{r} y) \varphi(y)\mathrm{d}y\Big|&\geq \Big|\int_{\R^3}H_{\theta_k}(0, \tilde{x}) \varphi(y)\mathrm{d}y\Big|\nonumber\\
 &\quad-\Big|\int_{\R^3} (H_{\theta_k}(0, \tilde{x})-H_{\theta_k}(0, \tilde{x} - \widetilde{r} y)) \varphi(y)\mathrm{d}y\Big|\nonumber\\
 &\gtrsim |H_{\theta_k}(0,\tilde{x})|-e^{-e^{A_6^{11}}}A_6^{-2} T_3^{-1}\label{pH-eat}
 \end{align}
with $\widetilde{r} = e^{-e^{A_6^{11}}} T_3^{\frac12}$. Here, the bump function $\Phi=(\xi,\eta,\varphi)$ is smooth in $\R^3$ with compact support such that $\Phi\equiv1$ on $B(0,1)$, and~$\xi,\eta,\varphi$ are the 3-component vector, respectively. By adding \eqref{TH-eat}-\eqref{pH-eat} and integrating by parts, we conclude that
\begin{align*}
e^{-2e^{A_6^{11}}} T_3^{-\frac{1}{2}}&\lesssim
\Big|\int_{\R^3} H(0, \tilde{x} - \widetilde{r} y) \xi(y)\ \mathrm{d}y\Big| +\Big|\int_{\R^3} v(0, \tilde{x} - \widetilde{r} y) \nabla \times \eta(y)\ \mathrm{d}y\Big|\\
&\quad\quad+\Big|\int_{\R^3} H(0, \tilde{x} - \widetilde{r} y) \partial_{y_i}\varphi_i(y)\ \mathrm{d}y\Big|,
\end{align*}
and hence by H\"older's inequality
\begin{equation*}\label{tango}
\begin{aligned}
 e^{-9e^{A_6^{11}}} &\lesssim
\int_{B(\tilde{x}, \widetilde{r})} |(v,H)(0,x)|^3\ \mathrm{d}x \lesssim
\int_{5\widetilde{R}-\widetilde{r}\leq |x|\leq \frac{3A_6\widetilde{R}}{10}+\widetilde{r}} |(v,H)(0,x)|^3\ \mathrm{d}x \\
&\lesssim\int_{T_3^{\frac{1}{2}} \leq |x| \leq (e^{A_7} T_3)^{\frac{1}{2}}} |(v,H)(0,x)|^3\ \mathrm{d}x.
\end{aligned}
\end{equation*}
for all $A_4^2N_0^{-2}\leq T_3\leq A_4^{-1}$. Summing over a set of such scales $T_3$ increasing geometrically at ratio $e^{A_7}$, we conclude that
\begin{align*}
 &A_7^{-1}\log(A_4^{-3} N_0^2)e^{-9e^{A_6^{11}}} \\
 &\quad\lesssim \Big(\int\limits_{A_4N_0^{-1} \leq |x| \leq e^{\frac{A_7}{2}} A_4N_0^{-1}} +\cdots+\int\limits_{e^{\frac{A_7}{2}\cdot (m-1)}A_4N_0^{-1} \leq |x| \leq e^{\frac{A_7}{2}}A_4^{-\frac12}}\Big) |(v,H)(0,x)|^3\ \mathrm{d}x\\
 &\quad\lesssim\sum_{m=0}^{\lceil\mathrm{log}(A_4^{-3}N_0^2)\rceil} \int\limits_{e^{\frac{A_7}{2}\cdot m}A_4N_0^{-1} \leq |x| \leq  e^{\frac{A_7}{2}\cdot (m+1)}A_4N_0^{-1}}|(v,H)(0,x)|^3\ \mathrm{d}x.\\
 &\quad\lesssim
 \int_{\R^3}|(v,H)(0,x)|^3\ \mathrm{d}x\lesssim A,
 \end{align*}
which finally leads to
$$ N_0^2 \leq e^{e^{e^{A_6^{12}}}}.$$
\end{proof}
\begin{remark}\label{from-exp}
The triply exponential nature of the bounds in Proposition \ref{main-est} can be explained as follows. The first exponential factor originates from Proposition \ref{vi}, which helps identify an appropriate spatial scale $\widetilde{R}$. The second exponential factor is derived from the quantitative Carleman inequalities. Finally, the third exponential factor results from the need to identify a sufficient number of disjoint spatial scales to contradict \eqref{able}.
\end{remark}

\section{Appendix A. Proof the Lemma \ref{theo.higher}}
\renewcommand{\theequation}{A.\arabic{equation}}
\renewcommand{\thethm}{A.\arabic{thm}}
\label{Sec5}

\label{app.a}

The following lemma plays a key role in the proof of Lemma \ref{theo.higher}.
\begin{lem}\label{lem.vwcontr}
Let $B=B_1(0)\subset \mathbb{R}^3$, $B'=B_{1-\delta}(0)$ with $0<\delta<1$. Let $v\in L^2(B)$ be divergence-free and $\omega\triangleq \nabla\times v$.  Then, for $k=1,2,\ldots$ we have
\begin{align*}
\|D^kv\|_{L^p(B')}\leq c(\|D^{k-1}\omega\|_{L^p(B)}+\| v\|_{L^2(B)}+\|\omega\|_{L^p(B)})
\end{align*}
and
\begin{align}\label{vwcontr2}
\|D^kv\|_{C^{0,\alpha}(B')}\leq c(\|D^{k-1}\omega\|_{C^{0,\alpha}(B)}+\|v\|_{L^2(B)})
\end{align}
with $1<p<\infty$, $0<\alpha<1$. Here $c$ denotes the generic constant depending only on $\delta,~p$ and $\alpha$.
\end{lem}

The proof this lemma is standard which is well-known consequence of the classical $L^p$- and $C^{\alpha}$-estimates for the Laplace equation. However, since
we cannot find it in the literature, we give a full proof for reader's convenience.

\begin{proof}[The proof of Lemma \ref{lem.vwcontr}.]
 we first notice that  $\Delta v=-\nabla\times\omega$ due to $\div v=0$.  To derive the desired result,  we set $\tau$ be a smooth cut-off function that equals $1$ in $B'\subset\subset B$ and vanishes outside $B$, and then
$$D^kv=D^{k-1}(-\Delta)^{-1}\partial_{x_i}(\nabla\times(\omega\tau))+D^kA,$$
where $A$ is harmonic on $B'$. From the elliptic regularity for harmonic functions and Calder\'{o}n-Zygmund inequality, we have:

 Case 1.\quad $1<p<3$.
\begin{align*}
\|D^k A\|_{L^\infty(B')}&\leq c\|A\|_{L^1(B'')}\leq c(\|(-\Delta)^{-1}(\nabla\times(\omega\tau))\|_{L^{1}(B'')}+\|v\|_{L^1(B'')})\\
&\leq c(\|(-\Delta)^{-1}(\nabla\times(\omega\tau))\|_{L^{p^\ast}(\R^3)}+\|v\|_{L^2(B)})\\
&\leq c(\|\nabla(-\Delta)^{-1}(\nabla\times(\omega\tau))\|_{L^{p}(\R^3)}+\|v\|_{L^2(B)})\\
&\leq c(\|\omega\|_{L^{p}(B)}+\|v\|_{L^2(B)}),
\end{align*}
where $p^\ast=\frac{3p}{3-p}>1$, $B'\subset\subset B''\subset\subset B$.

 Case 2.\quad $1<\widetilde{p}<3\leq p<\infty$.
\begin{align*}
\|D^k A\|_{L^\infty(B')}&\leq c\|A\|_{L^1(B'')}\leq c(\|(-\Delta)^{-1}(\nabla\times(\omega\tau))\|_{L^{\widetilde{p}^\ast}(B'')}
+\|v\|_{L^1(B'')})\\
&\leq c(\|(-\Delta)^{-1}(\nabla\times(\omega\tau))\|_{L^{\widetilde{p}^\ast}(\R^3)}
+\|v\|_{L^2(B)})\\
&\leq c(\|\nabla(-\Delta)^{-1}(\nabla\times(\omega\tau))\|_{L^{\widetilde{p}}(\R^3)}
+\|v\|_{L^2(B)})\\
&\leq c(\|\omega\|_{L^{p}(B)}+\|v\|_{L^2(B)})
\end{align*}
where $\widetilde{p}^\ast=\frac{3\widetilde{p}}{3-\widetilde{p}}>1$. This, along with the Calder\'{o}n-Zygmund inequality, yields
$$
\|D^k v\|_{L^p(B')}\leq c(\|D^{k-1}\omega\|_{L^{p}(B)}+\|v\|_{L^2(B)}+\|\omega\|_{L^p(B)}).
$$
which is a desired result. Now we turn to the proof of \eqref{vwcontr2}. In fact,
\begin{align*}
\|D^k A\|_{C^{0,\alpha}(B')}&\leq\|D^k A\|_{C^{1}(B')}
\leq c\|D^{k+1}A\|_{L^\infty(B')}\leq c\|A\|_{L^1(B'')} \\
&\leq c(\|\omega\|_{L^{p}(B)}+\|v\|_{L^2(B)})\leq c(\|\omega\|_{L^{\infty}(B)}+\|v\|_{L^2(B)}).
\end{align*}
This, together with the classical Schauder estimate (for example, see \cite{MR1406091})
$$
\|(-\Delta)^{-1}\partial_{x_i}(\nabla\times(\omega\tau))\|_{C^{0,\alpha}(\R^3)}\leq C \|\omega\tau\|_{C^{0,\alpha}(\R^3)}
$$
derives
\begin{align*}
\|D^k v\|_{C^{0,\alpha}(B')}&\leq \|D^k (-\Delta)^{-1}(\nabla\times(\omega\tau))\|_{C^{0,\alpha}(B')}+
\|D^k A\|_{C^{0,\alpha}(B')}\\
&\leq c(\|D^{k-1}\omega\|_{C^{0,\alpha}(B)}+\|v\|_{L^2(B)})
\end{align*}
which completes the proof of Lemma \ref{lem.vwcontr}.
\end{proof}

\begin{proof}[The proof of Lemma \ref{theo.higher}.]
In order to obtain the desired result, we consider equation $\eqref{wJeq}_1$, $\eqref{Hxkeq}$ as a heat equation with force term
$$
(\omega\cdot\nabla)v-(J\cdot\nabla)H-(v\cdot\nabla)\omega+(H\cdot\nabla)J,$$
and
$$(H_{x_k}\cdot\nabla)v
 +(H\cdot\nabla)v_{x_k}-(v_{x_k}\cdot\nabla)H-(v\cdot\nabla)H_{x_k},$$
respectively, which has some known regularity.
Smoothing properties of the heat operator enable us to improve the
regularity of $\omega$ and $\nabla H$.

{\bf Step 1.} {\em Localisation.}  To derive the $L_t^\infty L_x^\infty$ interior estimates of $\omega$ and $\nabla H$,  we introduce  a cutoff function $\chi_1(t,x)\in C_c^\infty(\R^4)$ such that $0\leq\chi_1(t,x)\leq1$,
\begin{align*}
\chi_1(t,x)=\begin{cases}
1, & (t,x)\in Q_{\frac{5}{12}}(0,0),\\
0, & (t,x)\in \mathbb{R}^4\backslash Q_{\frac{11}{24}}(0,0).
\end{cases}
\end{align*}
By a simple calculation, we see that the  $\chi\omega_i$ satisfies
\begin{align}\label{wi}
&(\partial_t-\Delta)(\chi_1\omega_i)\nonumber\\
&\quad=\partial_t(\chi_1\omega_i)-\partial_{jj}(\chi_1\omega_i)\nonumber\\
&\quad=\omega_i\partial_t\chi_1+\chi_1\partial_t\omega_i
-(\omega_i\partial_{jj}\chi_1
+\partial_{j}\chi_1\partial_{j}\omega_i)
-(\partial_{j}\chi_1\partial_{j}\omega_i+\chi_1\partial_{jj}\omega_i)\nonumber\\
&\quad=\chi_1(\partial_t\omega_i-\partial_{jj}\omega_i)+(\partial_t\chi_1
-\partial_{jj}\chi_1)\omega_i-2\partial_{j}\chi_1\partial_{j}\omega_i\nonumber\\
&\quad=\chi_1(\omega_j\partial_{j}v_i+H_j\partial_{j}J_i
-v_j\partial_{j}\omega_i-J_j\partial_{j}H_i)
+\omega_i\partial_{t}\chi_1+\omega_i\partial_{jj}\chi_1
-2\partial_{j}\Big(\omega_i\partial_{j}\chi_1\Big)\nonumber\\
&\quad=\partial_{j}\Big(\chi_1\omega_jv_i-\chi_1\omega_iv_j\Big)+
\partial_{j}\Big[\chi_1(H_jJ_i-H_iJ_j)\Big]\nonumber\\
&\quad\quad -\partial_{j}\chi_1
\Big(\omega_iv_j+J_jH_i
-\omega_jv_i-J_iH_j\Big) +(\partial_t\chi_1+\Delta\chi_1)\omega_i
-2\partial_j(\partial_j\chi_1\omega_i)\nonumber\\
&\quad
\triangleq F_i^{(1)},
\end{align}
and $\chi_1 (H_{x_k})_i$ fulfills
\begin{align}\label{Hxki}
(\partial_t-\Delta)(\chi_1 (H_{x_k})_i)
&=\partial_{j}\Big(\chi_1 (H_{x_k})_jv_i-\chi_1 (H_{x_k})_iv_j\Big)+
\partial_{j}\Big[\chi_1(H_j(v_{x_k})_i-H_i(v_{x_k})_j)\Big]\nonumber\\
&\quad-\partial_{j}\chi_1
\Big[H_i(v_{x_k})_j+(H_{x_k})_iv_j
-(v_{x_k})_iH_j-v_i(H_{x_k})_j\Big]\nonumber\\
&\quad +
(\partial_t\chi_1+\Delta\chi_1)(H_{x_k})_i
-2\partial_j(\partial_j\chi_1(H_{x_k})_i)\nonumber\\
&\triangleq G_i^{(1)},
\end{align}
where $\omega_i, (H_{x_k})_i$ are the  $i$th component of $\omega, H_{x_k}$, respectively. By the uniqueness of the solution to the the heat operator (see the page 393 of \cite{MR3616490})), one has
$$
(\chi_1\omega_i,\chi_1 (H_{x_k})_i)=(\partial_t-\Delta)^{-1}(F_i^{(1)}, G_i^{(1)}).
$$

\medskip

{\bf Step 2.} {\em Bootstrapping arguments: from $L^2$ to $L^{\infty}$.} First, by using the Lemma \ref{lem.heat1}-\ref{lem.heat2} and \eqref{small-v-H}, it is clear that
\begin{align*}
&\|(\partial_t-\Delta)^{-1}F_i^{(1)}\|_{L^3_tL^3_x(\R^4)}\\
&\quad\lesssim\|\chi_1\omega_jv_i-\chi_1\omega_iv_j\|_{L^2_tL^2_x(Q_{\frac{11}{24}}(0,0))}
+\|\chi_1 H_jJ_i-\chi_1 H_iJ_j\|_{L^2_tL^2_x(Q_{\frac{11}{24}}(0,0))}
\\&\quad\quad+\|(\partial_{j}\chi_1)
(\omega_iv_j+J_jH_i-\omega_jv_i-J_iH_j)\|_{L^2_tL^2_x(Q_{\frac{11}{24}}(0,0))}
+\|(\partial_t\chi_1+\Delta\chi_1)\omega_i\|_{L^2_tL^2_x(Q_{\frac{11}{24}}(0,0))}\\
&\quad\lesssim\|v\|_{L^\infty_tL^\infty_x(Q_{\frac12}(0,0))}+\|H\|_{L^\infty_tL^\infty_x
(Q_{\frac12}(0,0))}+\|\omega\|_{L^2_tL^2_x(Q_{\frac12}(0,0))}
+\|J\|_{L^2_tL^2_x(Q_{\frac12}(0,0))},
\end{align*}
similarly,
\begin{align*}
&\|(\partial_t-\Delta)^{-1}G_i^{(1)}\|_{L^3_tL^3_x(\R^4)}\\
&\quad\lesssim\|v\|_{L^\infty_tL^\infty_x(Q_{\frac12}(0,0))}
+\|H\|_{L^\infty_tL^\infty_x(Q_{\frac12}(0,0))}
+\|\omega\|_{L^2_tL^2_x(Q_{\frac12}(0,0))}
+\|J\|_{L^2_tL^2_x(Q_{\frac12}(0,0))}.
\end{align*}
Here the Lemma \ref{lem.vwcontr} has been used in the last inequality. Then, one has
\begin{align*}
&\|(\omega_i,(H_{x_k})_i)\|_{L^3_tL^3_x(Q_{\frac5{12}}(0,0))}\\&
\quad\leq \|(\chi_1\omega_i,\chi_1 (H_{x_k})_i)\|_{L^3_tL^3_x(\R^4)}\\&
\quad\lesssim\|v\|_{L^\infty_tL^\infty_x(Q_{\frac12}(0,0))}
+\|H\|_{L^\infty_tL^\infty_x(Q_{\frac12}(0,0))}+\|\omega\|
_{L^2_tL^2_x(Q_{\frac12}(0,0))}+\|J\|_{L^2_tL^2_x(Q_{\frac12}(0,0))}.
\end{align*}
Running the localisation argument above again, we take a cutoff function $\chi_2(t,x)\in C_c^\infty(\R^4)$ such that $0\leq\chi_2(t,x)\leq1$,
\begin{align*}
\chi_2(t,x)=\begin{cases}
1, & (t,x)\in Q_{\frac{1}{3}}(0,0),\\
0, & (t,x)\in \mathbb{R}^4\backslash Q_{\frac{3}{8}}(0,0),
\end{cases}
\end{align*}
then
\begin{align*}
&\|(\partial_t-\Delta)^{-1}F_i^{(2)}\|_{L^6_tL^6_x(\R^4)}\\
&\quad\lesssim\|\chi_2\omega_jv_i
-\chi_2\omega_iv_j\|_{L^3_tL^3_x(Q_{\frac{3}{8}}(0,0))}
+\|\chi_2H_jJ_i-\chi_2H_iJ_j\|_{L^3_tL^3_x(Q_{\frac{3}{8}}(0,0))}
\\&\quad\quad+\|(\partial_{j}\chi_2)
(\omega_iv_j+J_jH_i-\omega_jv_i-J_iH_j)\|_{L^3_tL^3_x(Q_{\frac{3}{8}}(0,0))}
+\|(\partial_t\chi_2+\Delta\chi_2)\omega_i\|_{L^3_tL^3_x(Q_{\frac{3}{8}}(0,0))}\\
&\quad\lesssim\|v\|_{L^\infty_tL^\infty_x(Q_{\frac12}(0,0))}
+\|H\|_{L^\infty_tL^\infty_x(Q_{\frac12}(0,0))}
+\|\omega\|_{L^2_tL^2_x(Q_{\frac12}(0,0))}+\|J\|_{L^2_tL^2_x(Q_{\frac12}(0,0))},
\end{align*}
Similarly, by using Lemma \ref{lem.heat1}-\ref{lem.heat2}, Lemma \ref{lem.vwcontr} and \eqref{small-v-H}, we get
\begin{align*}
&\|(\partial_t-\Delta)^{-1}G_i^{(2)}\|_{L^6_tL^6_x(\R^4)}\\
&\quad\lesssim\|\nabla H\|_{L^3_tL^3_x(Q_{\frac5{12}}(0,0))}
+\|\nabla v\|_{L^3_tL^3_x(Q_{\frac38}(0,0))}\nonumber\\
&\quad\lesssim\|\nabla H\|_{L^3_tL^3_x(Q_{\frac5{12}}(0,0))}+\|\omega\|_{L^3(Q_{\frac5{12}}(0,0))}
+\|v\|_{L^2_tL^2_x(Q_{\frac5{12}}(0,0))}\\
&\quad\lesssim\|v\|_{L^\infty(Q_{\frac12}(0,0))}+\|H\|_{L^\infty(Q_{\frac12}(0,0))}
+\|\omega\|_{L^2_tL^2_x(Q_{\frac12}(0,0))}+\|J\|_{L^2_tL^2_x(Q_{\frac12}(0,0))}.
\end{align*}
We use  Lemma \ref{lem.heat1} again to derive that
\begin{align*}
&\|(\omega_i,(H_{x_k})_i)\|_{L^6_tL^6_x(Q_{\frac13}(0,0))}\\&
\quad\leq \|(\chi_2\omega_i,\chi_2 (H_{x_k})_i)\|_{L^6_tL^6_x(\R^4)}\\&
\quad\lesssim\|v\|_{L^\infty_tL^\infty_x(Q_{\frac12}(0,0))}
+\|H\|_{L^\infty_tL^\infty_x(Q_{\frac12}(0,0))}+\|\omega\|_
{L^2_tL^2_x(Q_{\frac12}(0,0))}+\|J\|_{L^2_tL^2_x(Q_{\frac12}(0,0))}.
\end{align*}
Now, to finish our proof of this step, we take $\chi_3(t,x)\in C_c^\infty(\R^4)$ such that $0\leq\chi_3(t,x)\leq1$,
\begin{align*}
\chi_3(t,x)=\begin{cases}
1, & (t,x)\in Q_{\frac{1}{4}}(0,0),\\
0, & (t,x)\in \mathbb{R}^4\backslash Q_{\frac{7}{24}}(0,0),
\end{cases}
\end{align*}
then
\begin{align*}
&\|(\partial_t-\Delta)^{-1}F_i^{(3)}\|_{L^\infty_tL^\infty_x(\R^4)}\nonumber\\
&\quad\lesssim\|\chi_3\omega_jv_i-\chi_3\omega_iv_j\|
_{L^6_tL^6_x(Q_{\frac7{24}}(0,0))}
+\|\chi_3H_jJ_i-\chi_3H_iJ_j\|_{L^6_tL^6_x(Q_{\frac7{24}}(0,0))}
\nonumber\\
&\quad\quad +\|(\partial_{j}\chi_3)
(\omega_iv_j+J_jH_i-\omega_jv_i-J_iH_j)\|_{L^6_tL^6_x(Q_{\frac7{24}}(0,0))}\nonumber\\
&\quad\quad+\|
(\partial_t\chi_3+\Delta\chi_3)\omega_i\|
_{L^6_tL^6_x(Q_{\frac7{24}}(0,0))}
+\|(\partial_j\chi_3)\omega_i\|
_{L^6_tL^6_x(Q_{\frac7{24}}(0,0))}\nonumber\\
&\quad\lesssim\|\omega\|_{L^6_tL^6_x(Q_{\frac1{3}}(0,0))}+
\|\nabla H\|_{L^6_tL^6_x(Q_{\frac1{3}}(0,0))}\nonumber\\
&\quad\lesssim\|v\|_{L^\infty_tL^\infty_x(Q_{\frac12}(0,0))}
+\|H\|_{L^\infty_tL^\infty_x(Q_{\frac12}(0,0))}
+\|\omega\|_{L^2_tL^2_x(Q_{\frac12}(0,0))}+\|J\|_{L^2_tL^2_x(Q_{\frac12}(0,0))},
\end{align*}
and we used Lemma \ref{lem.vwcontr}, this suggests that
\begin{align*}
&\|(\partial_t-\Delta)^{-1}G_i^{(3)}\|_{L^\infty_tL^\infty_x(\R^4)}\nonumber\\
&\quad\lesssim\|\nabla H\|_{L^6_tL^6_x(Q_{\frac7{24}}(0,0))}
+\|\nabla v\|_{L^6_tL^6_x(Q_{\frac{7}{24}}(0,0))}\nonumber\\
&\quad\lesssim\|\nabla H\|_{L^6_tL^6_x(Q_{\frac1{3}}(0,0))}
+\|\omega\|_{L^6_tL^6_x(Q_{\frac1{3}}(0,0))}
+\|v\|_{L^2_tL^2_x(Q_{\frac1{3}}(0,0))}
\nonumber\\
&\quad\lesssim\|v\|_{L^\infty_tL^\infty_x(Q_{\frac12}(0,0))}
+\|H\|_{L^\infty_tL^\infty_x(Q_{\frac12}(0,0))}
+\|\omega\|_{L^2_tL^2_x(Q_{\frac12}(0,0))}+\|J\|_{L^2_tL^2_x(Q_{\frac12}(0,0))}.
\end{align*}
Together with Lemma \ref{lem.heat1}, we can obtain
\begin{align}\label{whiL3}
&\|(\omega_i,(H_{x_k})_i)\|_{L^\infty_tL^\infty_x(Q_{\frac1{4}}(0,0))}\nonumber\\
&\quad\leq \|(\chi_3\omega_i,\chi_3 (H_{x_k})_i)\|_{L^\infty_tL^\infty_x(\R^4)}\nonumber\\
&\quad\lesssim\|v\|_{L^\infty_tL^\infty_x(Q_{\frac12}(0,0))}
+\|H\|_{L^\infty_tL^\infty_x(Q_{\frac12}(0,0))}
+\|\omega\|_{L^2_tL^2_x(Q_{\frac12}(0,0))}+\|J\|_{L^2_tL^2_x(Q_{\frac12}(0,0))}.
\end{align}
{\bf Step 3.} {\em From $L^{\infty}$ to $C^{1}$}.
Let us apply these results to equation \eqref{wi} and \eqref{Hxki}, which we
rewrites as
\begin{align*}
(\partial_t-\Delta)(\eta_1\omega_i)=\partial_{j}f_1+g_1,
\end{align*}
and
\begin{align*}
(\partial_t-\Delta)(\eta_2(H_{x_k})_i)=\partial_{j}f_2+g_2,
\end{align*}
where $\eta_1(t,x),\eta_2(t,x)\in C_c^\infty(\R^4)$ satisfies $0\leq\eta_1, \eta_2\leq1$ with
\begin{align*}
\eta_1=\begin{cases}
1, & (t,x)\in Q_{\frac{1}{4}-[\frac{1}{2^3}(1-\frac12)]}(0,0),\\
0, & (t,x)\in \mathbb{R}^4\backslash Q_{\frac{1}{4}}(0,0),
\end{cases}\quad
\eta_2=\begin{cases}
1, & (t,x)\in Q_{\frac{1}{4}-[\frac{1}{2^3}(1-\frac1{2^3})]}(0,0),\\
0, & (t,x)\in \mathbb{R}^4\backslash Q_{\frac{1}{4}-[\frac{1}{2^3}(1-\frac1{2^2})]}(0,0),
\end{cases}
\end{align*}
and
\begin{align*}
&f_1=\eta_1\omega_jv_i-\eta_1\omega_iv_j+\eta_1H_jJ_i-\eta_1 H_iJ_j-2(\partial_j\eta_1)\omega_i,\\
&f_2=\eta_2(H_{x_k})_jv_i-\eta_2(H_{x_k})_iv_j+
\eta_2H_j(v_{x_k})_i-\eta_2H_i(v_{x_k})_j)-2(\partial_j\eta_2)(H_{x_k})_i,\\
&g_1=\partial_{j}\eta_1
\Big(\omega_iv_j+J_jH_i
-\omega_jv_i-J_iH_j\Big)+(\partial_t\eta_1+\Delta\eta_1)\omega_i\\
&g_2=\partial_{j}\eta_2
\Big[H_i(v_{x_k})_j+(H_{x_k})_iv_j
-(v_{x_k})_iH_j-v_i(H_{x_k})_j\Big]+
(\partial_t\eta_2+\Delta\eta_2)(H_{x_k})_i.
\end{align*}
Based on the result of the Step $2$, we know that $v,~H,~\omega,~\nabla H\in L^\infty_tL^\infty_x(Q_{\frac14}(0,0))$. This, along with the Lemma \ref{lem.heat2-1}, \eqref{small-v-H} and \eqref{whiL3}, yields for any $0<\alpha<1$,
\begin{align*}
 &\|\omega\|_{L_t^\infty C^{0,\alpha}_x(Q_{\frac{1}{4}-[\frac{1}{2^3}(1-\frac12)]}(0,0))}\nonumber\\
 &\quad\lesssim \|f_1\|_{L_t^{\infty}L_x^{\infty}(Q_{\frac14}(0,0))}+\|g_1\|
_{L_t^{\infty}L_x^{\infty}(Q_{\frac14}(0,0))}\nonumber\\
&\quad\lesssim  \|\omega\|_{L_t^{\infty}L_x^{\infty}(Q_{\frac14}(0,0))}
+ \|\nabla H\|_{L_t^{\infty}L_x^{\infty}(Q_{\frac14}(0,0))}\nonumber\\
&\quad\lesssim \|v\|_{L_t^{\infty}L_x^{\infty}(Q_{\frac12}(0,0))}
+\|H\|_{L_t^{\infty}L_x^{\infty}(Q_{\frac12}(0,0))}
+\|\omega\|_{L^2_tL^2_x(Q_{\frac12}(0,0))}
+\|J\|_{L^2_tL^2_x(Q_{\frac12}(0,0))}.
\end{align*}
By using Lemma \ref{lem.vwcontr} again, one has
\begin{align*}
&\|\nabla v\|_{L_t^\infty L^{\infty}_x(Q_{\frac{1}{4}-[\frac{1}{2^3}(1-\frac1{2^2})]}(0,0))}\nonumber\\
&\quad\lesssim\|\nabla v\|_{L_t^\infty C^{0,\alpha}_x(Q_{\frac{1}{4}-[\frac{1}{2^3}(1-\frac1{2^2})]}(0,0))}\nonumber\\
&\quad\lesssim\|\omega\|_{L_t^\infty C^{0,\alpha}_x(Q_{\frac{1}{4}-[\frac{1}{2^3}(1-\frac1{2})]}(0,0))}
+\|v\|_{L_t^\infty L^\infty_x(Q_{\frac1{2}}(0,0))}\nonumber\\
&\quad\lesssim \|v\|_{L_t^{\infty}L_x^{\infty}(Q_{\frac12}(0,0))}
+\|H\|_{L_t^{\infty}L_x^{\infty}(Q_{\frac12}(0,0))}
+\|\omega\|_{L^2_tL^2_x(Q_{\frac12}(0,0))}
+\|J\|_{L^2_tL^2_x(Q_{\frac12}(0,0))}.
\end{align*}
On the other hand, Lemma \ref{lem.heat2-1} gives
\begin{align*}
 &\|\nabla H\|_{L_t^\infty C^{0,\alpha}_x(Q_{\frac{1}{4}-[\frac{1}{2^3}(1-\frac1{2^3})]}(0,0))}\nonumber\\
 &\quad\lesssim\|f_2\|_{L_t^{\infty}L_x^{\infty}
 (Q_{\frac{1}{4}-[\frac{1}{2^3}(1-\frac1{2^2})]}(0,0))}
 +\|g_2\|_{L_t^{\infty}L_x^{\infty}
(Q_{\frac{1}{4}-[\frac{1}{2^3}(1-\frac1{2^2})]}(0,0))}\nonumber\\
&\quad\lesssim  \|\nabla H\|_{L_t^{\infty}L_x^{\infty}(Q_{\frac{1}{4}-[\frac{1}{2^3}(1-\frac1{2^2})]}(0,0))}
+ \|\nabla v\|_{L_t^\infty L^{\infty}_x(Q_{\frac{1}{4}-[\frac{1}{2^3}(1-\frac1{2^2})]}(0,0))}\nonumber\\
&\quad\lesssim \|v\|_{L_t^{\infty}L_x^{\infty}(Q_{\frac12}(0,0))}
+\|H\|_{L_t^{\infty}L_x^{\infty}(Q_{\frac12}(0,0))}
+\|\omega\|_{L^2_tL^2_x(Q_{\frac12}(0,0))}+\|J\|_{L^2_tL^2_x(Q_{\frac12}(0,0))}.
\end{align*}

{\bf Step 4.} {\em Conclusion by induction.} We conclude the proof by induction on $k$ in this step.
 Now, let's assume that the following estimates hold up to the $m$-order derivative, i.e. for any $2\leq k\leq m$
\begin{align}\label{v-H-est-k}
&\|\nabla^k (v,H)\|_{L_t^\infty L^{\infty}_x(Q_{\frac{1}{4}-[\frac{1}{2^3}(1-\frac1{2^{6k-6}})]}(0,0))}\nonumber\\
&\quad\lesssim \|v\|_{L_t^{\infty}L_x^{\infty}(Q_{\frac12}(0,0))}
+\|H\|_{L_t^{\infty}L_x^{\infty}(Q_{\frac12}(0,0))}
+\|\omega\|_{L^2_tL^2_x(Q_{\frac12}(0,0))}
+\|J\|_{L^2_tL^2_x(Q_{\frac12}(0,0))}.
\end{align}
We will prove the following estimates
\begin{align*}
&\|\nabla^{m+1} (v,H)\|_{L_t^\infty L^{\infty}_x(Q_{\frac{1}{4}-[\frac{1}{2^3}(1-\frac1{2^{6m}})]}(0,0))}\\
&\quad\lesssim \|v\|_{L_t^{\infty}L_x^{\infty}(Q_{\frac12}(0,0))}
+\|H\|_{L_t^{\infty}L_x^{\infty}(Q_{\frac12}(0,0))}
+\|\omega\|_{L^2_tL^2_x(Q_{\frac12}(0,0))}
+\|J\|_{L^2_tL^2_x(Q_{\frac12}(0,0))}.
\end{align*}
By using \eqref{v-H-est-k}, we can obtain
 \begin{align*}
&\|\nabla^{m-1} (\omega_i,(H_{x_k})_i)\|_{L_t^\infty L^{\infty}_x(Q_{\frac{1}{4}-[\frac{1}{2^3}(1-\frac1{2^{6m-6}})]}(0,0))}\nonumber\\
&\quad\lesssim \|v\|_{L_t^{\infty}L_x^{\infty}(Q_{\frac12}(0,0))}
+\|H\|_{L_t^{\infty}L_x^{\infty}(Q_{\frac12}(0,0))}
+\|\omega\|_{L^2_tL^2_x(Q_{\frac12}(0,0))}
+\|J\|_{L^2_tL^2_x(Q_{\frac12}(0,0))}.
\end{align*}
Next, according to the calculation process of equation \eqref{wi}-\eqref{Hxki}, we get
\begin{align*}
(\partial_t-\Delta)(\varphi_1\partial_{x_r}^{m-1}\omega_i)=\partial_{j}F_1+G_1,
\end{align*}
and
\begin{align*}
(\partial_t-\Delta)[\varphi_2\partial_{x_r}^{m-1}(H_{x_k})_i]=\partial_{j}F_2+G_2,
\end{align*}
where $\varphi_1(t,x),\varphi_2(t,x)\in C_c^\infty(\R^4)$ satisfies $0\leq\varphi_1, \varphi_2\leq1$ with
\begin{align*}
\varphi_1(t,x)=\begin{cases}
1, & (t,x)\in Q_{\frac{1}{4}-[\frac{1}{2^3}(1-\frac1{2^{6m-5}})]}(0,0),\\
0, & (t,x)\in \mathbb{R}^4\backslash Q_{\frac{1}{4}-[\frac{1}{2^3}(1-\frac1{2^{6m-6}})]}(0,0),
\end{cases}
\end{align*}
\begin{align*}
\varphi_2(t,x)=\begin{cases}
1, & (t,x)\in Q_{\frac{1}{4}-[\frac{1}{2^3}(1-\frac1{2^{6m-3}})]}(0,0),\\
0, & (t,x)\in \mathbb{R}^4\backslash Q_{\frac{1}{4}-[\frac{1}{2^3}(1-\frac1{2^{6m-4}})]}(0,0),
\end{cases}
\end{align*}
and
\begin{align*}
F_1&=\sum_{s=0}^{m-1}\binom{m-1}{s} \Big[\varphi_1(\partial_{x_r}^s\omega_j)
(\partial_{x_r}^{m-1-s}v_i)
+\varphi_1(\partial^s_{x_r}H_j)(\partial_{x_r}^{m-1-s}J_i) \\ &\quad\quad\quad-\varphi_1 (\partial_{x_r}^sv_j)(\partial_{x_r}^{m-1-s}\omega_i)
-\varphi_1(\partial^s_{x_r}J_j)(\partial_{x_r}^{m-1-s}H_i)\Big]- 2(\partial_j\varphi_1)(\partial_{x_r}^{m-1}\omega_i),\\
F_2&=\sum_{s=0}^{m-1}\binom{m-1}{s}\Big[\varphi_2\partial_{x_r}^s(H_{x_k})_j
(\partial_{x_r}^{m-1-s}v_i)
+\varphi_2(\partial^s_{x_r}H_j)\partial_{x_r}^{m-1-s}(v_{x_k})_i\\ &\quad\quad\quad-\varphi_2 \partial_{x_r}^s(v_{x_k})_j(\partial_{x_r}^{m-1-s}H_i) -\varphi_2(\partial^s_{x_r}v_j)
\partial_{x_r}^{m-1-s}(H_{x_k})_i\Big]- 2(\partial_j\varphi_2)(\partial_{x_r}^{m-1}(H_{x_k})_i),\\
G_1&=\partial_{j}\varphi_1\Big[\sum_{s=0}^{m-1}\binom{m-1}{s}\Big(
\partial_{x_r}^s\omega_j\partial_{x_r}^{m-1-s}v_i+\partial_{x_r}^sH_j
\partial_{x_r}^{m-1-s}J_i
-\partial_{x_r}^sv_j\partial_{x_r}^{m-1-s}\omega_i\\
&\quad\quad\quad-\partial_{x_r}^sJ_j\partial_{x_r}^{m-1-s}H_i
\Big)\Big]+(\partial_t\varphi_1+\Delta\varphi_1)(\partial_{x_r}^{m-1}\omega_i)\\
G_2&=\partial_{j}\varphi_2\Big[\sum_{s=0}^{m-1}\binom{m-1}{s}\Big(
\partial_{x_r}^s(H_{x_k})_j\partial_{x_r}^{m-1-s}v_i
+\partial_{x_r}^sH_j\partial_{x_r}^{m-1-s}(v_{x_k})_i
\\
&\quad\quad\quad-\partial_{x_r}^s(v_{x_k})_j\partial_{x_r}^{m-1-s}H_i
-\partial_{x_r}^sv_j\partial_{x_r}^{m-1-s}(H_{x_k})_i
\Big)\Big]
+(\partial_t\varphi_2+\Delta\varphi_2)[\partial_{x_r}^{m-1}(H_{x_k})_i].
\end{align*}
Then, following the process of the Step 3 of the above proof, we can obtain
\begin{align*}
&\quad\quad\quad\boxed{\mbox{the estimate}\, \|\nabla^{m-1} (\omega_i,(H_{x_k})_i)\|_{L_t^\infty L^{\infty}_x(Q_{\frac{1}{4}-[\frac{1}{2^3}(1-\frac1{2^{6m-6}})]}(0,0))}}\\&
\quad\overset{(1)}{\Longrightarrow}\boxed{\mbox{the estimate}\,\|\nabla^{m-1}\omega\|_{L_t^\infty C^{0,\alpha}_x(Q_{\frac{1}{4}-[\frac{1}{2^3}(1-\frac1{2^{6m-5}})]}(0,0))}}\\&
\quad\overset{(2)}{\Longrightarrow}\boxed{\mbox{the estimate}\,\|\nabla^{m}v\|_{L_t^\infty C^{0,\alpha}_x(Q_{\frac{1}{4}-[\frac{1}{2^3}(1-\frac1{2^{6m-4}})]}(0,0))}}\\&
\quad\overset{(3)}{\Longrightarrow}\boxed{\mbox{the estimate}\,\|\nabla^{m}H\|_{L_t^\infty C^{0,\alpha}_x(Q_{\frac{1}{4}-[\frac{1}{2^3}(1-\frac1{2^{6m-3}})]}(0,0))}}\\
&\quad\overset{(4)}{\Longrightarrow}\boxed{\mbox{the estimate}\,\|(\nabla^{m}\omega,\nabla^{m+1}H)\|_{L_t^\infty L^\infty_x(Q_{\frac{1}{4}-[\frac{1}{2^3}(1-\frac1{2^{6m-2}})]}(0,0))}}\\&
\quad\overset{(5)}{\Longrightarrow}\boxed{\mbox{the estimate}\,\|\nabla^{m}\omega\|_{L_t^\infty C^{0,\alpha}_x(Q_{\frac{1}{4}-[\frac{1}{2^3}(1-\frac1{2^{6m-1}})]}(0,0))}}\\&
\quad\overset{(6)}{\Longrightarrow}\boxed{\mbox{the estimate}\,\|\nabla^{m+1} v\|_{L_t^\infty L^{\infty}_x(Q_{\frac{1}{4}-[\frac{1}{2^3}(1-\frac1{2^{6m}})]}(0,0))}}.
\end{align*}
Notices that
\begin{align*}
\frac14-\Big(\frac1{2^4}+\frac1{2^5}+\cdots+\frac1{2^{6m+3}}\Big)
=\frac{1}{4}-[\frac{1}{2^3}(1-\frac1{2^{6m}})]
=\frac{1}{8}+\frac1{2^{6m+3}}>\frac18,
\end{align*}
which, along with Lemma \ref{theo.higher},  leads to that for any $m\geq 0$
\begin{align*}
&\|\nabla^m (v,H)\|_{L_t^\infty L^{\infty}_x(Q_{\frac{1}{8}}(0,0))}\nonumber\\
&\quad\lesssim \|v\|_{L_t^{\infty}L_x^{\infty}(Q_{\frac12}(0,0))}
+\|H\|_{L_t^{\infty}L_x^{\infty}(Q_{\frac12}(0,0))}
+\|\omega\|_{L^2_tL^2_x(Q_{\frac12}(0,0))}
+\|J\|_{L^2_tL^2_x(Q_{\frac12}(0,0))}.
\end{align*}
This concludes the proof of Lemma \ref{theo.higher}.
\end{proof}

\section*{Acknowledgments} The research of BL was partially supported by NSFC-$12371202$ and Hunan provincial NSF-
2022jj10032, 22A0057.

\addcontentsline{toc}{section}{\protect\numberline{}{References}}
\bibliographystyle{abbrv}\bibliography{Green2021}

\begin{thebibliography}{10}

\bibitem{MR1393572}
F.~Armero and J.~C. Simo.
\newblock Long-term dissipativity of time-stepping algorithms for an abstract
  evolution equation with applications to the incompressible {MHD} and
  {N}avier-{S}tokes equations.
\newblock {\em Comput. Methods Appl. Mech. Engrg.}, 131(1-2):41--90, 1996.

\bibitem{MR4278282}
T.~Barker and C.~Prange.
\newblock Quantitative regularity for the {N}avier-{S}tokes equations via
  spatial concentration.
\newblock {\em Comm. Math. Phys.}, 385(2):717--792, 2021.

\bibitem{MR3713543}
T.~Barker and G.~Seregin.
\newblock A necessary condition of potential blowup for the {N}avier-{S}tokes
  system in half-space.
\newblock {\em Math. Ann.}, 369(3-4):1327--1352, 2017.

\bibitem{MR673830}
L.~Caffarelli, R.~Kohn, and L.~Nirenberg.
\newblock Partial regularity of suitable weak solutions of the
  {N}avier-{S}tokes equations.
\newblock {\em Comm. Pure Appl. Math.}, 35(6):771--831, 1982.

\bibitem{MR2595721}
C.~Cao and J.~Wu.
\newblock Two regularity criteria for the 3{D} {MHD} equations.
\newblock {\em J. Differential Equations}, 248(9):2263--2274, 2010.

\bibitem{MR2336368}
Q.~Chen, C.~Miao, and Z.~Zhang.
\newblock The {B}eale-{K}ato-{M}ajda criterion for the 3{D}
  magneto-hydrodynamics equations.
\newblock {\em Comm. Math. Phys.}, 275(3):861--872, 2007.

\bibitem{MR2452599}
Q.~Chen, C.~Miao, and Z.~Zhang.
\newblock On the regularity criterion of weak solution for the 3{D} viscous
  magneto-hydrodynamics equations.
\newblock {\em Comm. Math. Phys.}, 284(3):919--930, 2008.

\bibitem{MR98556}
T.~G. Cowling.
\newblock {\em Magnetohydrodynamics}.
\newblock Interscience Tracts on Physics and Astronomy, No. 4. Interscience
  Publishers, Inc., New York; Interscience Publishers, Ltd., London, 1957.

\bibitem{MR2551795}
H.~Dong and D.~Du.
\newblock The {N}avier-{S}tokes equations in the critical {L}ebesgue space.
\newblock {\em Comm. Math. Phys.}, 292(3):811--827, 2009.

\bibitem{MR346289}
G.~Duvaut and J.~L. Lions.
\newblock In\'{e}quations en thermo\'{e}lasticit\'{e} et
  magn\'{e}tohydrodynamique.
\newblock {\em Arch. Rational Mech. Anal.}, 46:241--279, 1972.

\bibitem{MR2005639}
L.~Escauriaza, G.~Seregin, and V.~\v{S}ver\'{a}k.
\newblock Backward uniqueness for parabolic equations.
\newblock {\em Arch. Ration. Mech. Anal.}, 169(2):147--157, 2003.

\bibitem{MR607552}
C.~Foias, C.~Guillop\'e, and R.~Temam.
\newblock New a priori estimates for {N}avier-{S}tokes equations in dimension
  {$3$}.
\newblock {\em Comm. Partial Differential Equations}, 6(3):329--359, 1981.

\bibitem{MR3475661}
I.~Gallagher, G.~S. Koch, and F.~Planchon.
\newblock Blow-up of critical {B}esov norms at a potential {N}avier-{S}tokes
  singularity.
\newblock {\em Comm. Math. Phys.}, 343(1):39--82, 2016.

\bibitem{MR2334589}
C.~He and Y.~Wang.
\newblock On the regularity criteria for weak solutions to the
  magnetohydrodynamic equations.
\newblock {\em J. Differential Equations}, 238(1):1--17, 2007.

\bibitem{MR2142366}
C.~He and Z.~Xin.
\newblock On the regularity of weak solutions to the magnetohydrodynamic
  equations.
\newblock {\em J. Differential Equations}, 213(2):235--254, 2005.

\bibitem{MR2165089}
C.~He and Z.~Xin.
\newblock Partial regularity of suitable weak solutions to the incompressible
  magnetohydrodynamic equations.
\newblock {\em J. Funct. Anal.}, 227(1):113--152, 2005.

\bibitem{hu2024quantitativeboundsboundedsolutions}
R.~Hu, P.~T. Nguyen, Q.~H. Nguyen, and P.~Zhang.
\newblock Quantitative bounds for bounded solutions to the navier-stokes
  equations in endpoint critical besov spaces, 2024.

\bibitem{MR3179576}
H.~Jia and V.~\v{S}ver\'{a}k.
\newblock Local-in-space estimates near initial time for weak solutions of the
  {N}avier-{S}tokes equations and forward self-similar solutions.
\newblock {\em Invent. Math.}, 196(1):233--265, 2014.

\bibitem{MR1406091}
N.~V. Krylov.
\newblock {\em Lectures on elliptic and parabolic equations in {H}\"older
  spaces}, volume~12 of {\em Graduate Studies in Mathematics}.
\newblock American Mathematical Society, Providence, RI, 1996.

\bibitem{MR236541}
O.~A. Lady\v{z}enskaja.
\newblock Uniqueness and smoothness of generalized solutions of
  {N}avier-{S}tokes equations.
\newblock {\em Zap. Nau\v{c}n. Sem. Leningrad. Otdel. Mat. Inst. Steklov.
  (LOMI)}, 5:169--185, 1967.

\bibitem{MR1555394}
J.~Leray.
\newblock Sur le mouvement d'un liquide visqueux emplissant l'espace.
\newblock {\em Acta Math.}, 63(1):193--248, 1934.

\bibitem{MR2270882}
A.~Mahalov, B.~Nicolaenko, and T.~Shilkin.
\newblock {$L_{3,\infty}$}-solutions to the {MHD} equations.
\newblock {\em Zap. Nauchn. Sem. S.-Peterburg. Otdel. Mat. Inst. Steklov.
  (POMI)}, 336:112--132, 275--276, 2006.

\bibitem{MR4334731}
S.~Palasek.
\newblock Improved quantitative regularity for the {N}avier-{S}tokes equations
  in a scale of critical spaces.
\newblock {\em Arch. Ration. Mech. Anal.}, 242(3):1479--1531, 2021.

\bibitem{MR4502800}
S.~Palasek.
\newblock A minimum critical blowup rate for the high-dimensional
  {N}avier-{S}tokes equations.
\newblock {\em J. Math. Fluid Mech.}, 24(4):Paper No. 108, 28, 2022.

\bibitem{MR1343429}
H.~Politano, A.~Pouquet, and P.-L. Sulem.
\newblock Current and vorticity dynamics in three-dimensional
  magnetohydrodynamic turbulence.
\newblock {\em Phys. Plasmas}, 2(8):2931--2939, 1995.

\bibitem{MR126088}
G.~Prodi.
\newblock Un teorema di unicit\`a{} per le equazioni di {N}avier-{S}tokes.
\newblock {\em Ann. Mat. Pura Appl. (4)}, 48:173--182, 1959.

\bibitem{MR3616490}
J.~C. Robinson, J.~L. Rodrigo, and W.~Sadowski.
\newblock {\em The three-dimensional {N}avier-{S}tokes equations}, volume 157
  of {\em Cambridge Studies in Advanced Mathematics}.
\newblock Cambridge University Press, Cambridge, 2016.
\newblock Classical theory.

\bibitem{MR2925135}
G.~Seregin.
\newblock A certain necessary condition of potential blow up for
  {N}avier-{S}tokes equations.
\newblock {\em Comm. Math. Phys.}, 312(3):833--845, 2012.

\bibitem{MR716200}
M.~Sermange and R.~Temam.
\newblock Some mathematical questions related to the {MHD} equations.
\newblock {\em Comm. Pure Appl. Math.}, 36(5):635--664, 1983.

\bibitem{MR136885}
J.~Serrin.
\newblock On the interior regularity of weak solutions of the {N}avier-{S}tokes
  equations.
\newblock {\em Arch. Rational Mech. Anal.}, 9:187--195, 1962.

\bibitem{MR4337421}
T.~Tao.
\newblock Quantitative bounds for critically bounded solutions to the
  {N}avier-{S}tokes equations.
\newblock In {\em Nine mathematical challenges---an elucidation}, volume 104 of
  {\em Proc. Sympos. Pure Math.}, pages 149--193. Amer. Math. Soc., Providence,
  RI, [2021] \copyright 2021.

\bibitem{MR1915942}
J.~Wu.
\newblock Bounds and new approaches for the 3{D} {MHD} equations.
\newblock {\em J. Nonlinear Sci.}, 12(4):395--413, 2002.

\bibitem{MR2026210}
J.~Wu.
\newblock Regularity results for weak solutions of the 3{D} {MHD} equations.
\newblock {\em Discrete Contin. Dyn. Syst.}, 10(1-2):543--556, 2004.

\bibitem{MR2398230}
J.~Wu.
\newblock Regularity criteria for the generalized {MHD} equations.
\newblock {\em Comm. Partial Differential Equations}, 33(1-3):285--306, 2008.

\bibitem{MR2128731}
Y.~Zhou.
\newblock Remarks on regularities for the 3{D} {MHD} equations.
\newblock {\em Discrete Contin. Dyn. Syst.}, 12(5):881--886, 2005.

\end{thebibliography}
\end{document}